\documentclass[11pt, eqno]{article}
\usepackage{bbm}
\usepackage{mathrsfs}
\usepackage{amsfonts}
\usepackage{amssymb}
\usepackage{graphicx}
\usepackage[all]{xy}

\usepackage{amsthm}
\usepackage{amsmath}
\usepackage{amsmath,amssymb,latexsym,color}
\usepackage[mathscr]{eucal}
\usepackage{CJK}
\usepackage{cases}
\usepackage{graphics}

\usepackage{psfrag}
\usepackage{subfigure}
\usepackage{color}

\usepackage{amssymb,latexsym}
\usepackage{amsmath,latexsym}
\usepackage{amscd}
\newcommand{\R}{{\mathbb R}}

\newtheorem{theorem}{Theorem}[section]

\newtheorem{corollary}[theorem]{Corollary}

\newtheorem{thm}[theorem]{Theorem}
\newtheorem{definition}[theorem]{Definition}

\newtheorem{remark}[theorem]{Remark}

\newtheorem{lemma}[theorem]{Lemma}

\newtheorem{proposition}[theorem]{Proposition}
\newtheorem{claim}[theorem]{Claim}

\begin{document}

\title{Representation formula for symmetric symplectic capacity\\ and applications}
\date{April 16, 2020}
\author{Rongrong Jin and Guangcun Lu
\thanks{Corresponding author
\endgraf \hspace{2mm} Partially supported
by the NNSF  11271044 of China.
\endgraf\hspace{2mm} 2010 {\it Mathematics Subject Classification.}
 53D35, 53C23 (primary), 70H05, 37J05, 57R17 (secondary).}}
  \maketitle \vspace{-0.3in}

 \abstract{This is the second installment in a series of papers aimed at generalizing symplectic
 capacities and homologies. We study symmetric versions of symplectic capacities
 for real symplectic manifolds, and obtain corresponding results for them to those of
 the first \cite{JinLu19} of this series (such as representation formula, a theorem by Evgeni Neduv, Brunn-Minkowski type inequality and Minkowski billiard trajectories proposed by Artstein-Avidan-Ostrover). }
\vspace{-0.1in}

\medskip\vspace{12mm}
%\tableofcontents

\section{Introduction and main results}\label{sec:Intro}

When constructing symplectic capacities for symplectic manifolds with some kinds of symmetry,
if the symmetry is considered then it is natural to obtain the refined symplectic capacities
which can be used to yield better results. For example, Liu and Wang \cite{LiuW12}
introduced a symmetric version of the Hofer-Zehnder symplectic capacity on a real symplectic
manifold,   Figalli-Palmer-Pelayo \cite{FigPP18} defined symplectic $G$-capacities and
studied their applications for integrable systems.  In this paper we introduce  a symmetric
version of the Ekeland-Hofer symplectic capacity on
$2n$-dimensional Euclid space with standard symplectic structure and linear anti-symplectic involutions
and give
representation formula for the symmetric capacities and corresponding results
  to the authors' article \cite{JinLu19}.

\vspace*{4pt}\noindent{\bf Notation.} The closure of a set $S$ is denoted by $\overline{S}$ or $Cl(S)$.
 The transpose of a matrix $A$ is denoted by $A^T$ without special statements.
 We always use $J_0$ to denote  standard  complex structure on $\mathbb{R}^{2n}$, which is given by the matrix
\begin{equation}\label{e:standcompl}
J_0=\left(
           \begin{array}{cc}
             0 & -I_n \\
             I_n & 0 \\
           \end{array}
         \right)
\end{equation}
in the linear coordinates $(q_1,\cdots,q_n,p_1,\cdots,p_n)$,
where $I_n$ denotes the identity matrix of order $n$.
That is, $J_0(q_1,\cdots, q_n, p_1,\cdots, p_n)=(-p_1,\cdots,-p_n,q_1,\cdots,q_n)$.
Moreover, for each natural number $m$ we use  $\langle\cdot,\cdot\rangle_{\mathbb{R}^{m}}$
to denote  the standard inner product in $\mathbb{R}^{m}$, and
$|\cdot|$ the induced norm.

\subsection{Symmetrical Hofer-Zehnder capacity}\label{sec:Intro1}

A {\bf real symplectic manifold} is a triple $(M,\omega,\tau)$
consisting of a symplectic manifold $(M,\omega)$ and an
anti-symplectic involution $\tau$ on $(M,\omega)$, i.e.
$\tau^\ast\omega=-\omega$ and $\tau^2=id_M$.
The fixed point set $L:={\rm Fix}(\tau)$ of $\tau$ is
called the {\bf real part} of $M$. It is either empty or a Lagrange
submanifold (cf. \cite{Vi99}). The standard linear symplectic space
$(\mathbb{R}^{2n},\omega_0)$ with $\omega_0=\sum_{i=1}^{n}dx_i\wedge dy_i$
is real with respect to the canonical involution
$\tau_0\in \mathcal{L}(\mathbb{R}^{2n})$ given by
\begin{equation}\label{e:1.can-inv}
\tau_0(x,y)=(x,-y),
\end{equation}
 and $L_0:={\rm Fix}(\tau_0)=\{(x,y)\in\mathbb{R}^{2n}\,|\,y=0\}$.
(We also denote $\hat{\tau_0}$ by the anti-symplectic involution on $(\mathbb{R}^{2n},\omega_0)$
given  by $\hat{\tau_0}(q,p)=(-q,p)$.
Clearly, $\tau_0$ and $\hat{\tau_0}$ are symmetric with respect to $\langle\cdot,\cdot\rangle_{\mathbb{R}^{2n}}$.)
%Hereafter $\langle\cdot,\cdot\rangle_{\mathbb{R}^{2n}}$ denotes the standard inner product
%in $\mathbb{R}^{2n}$.)
Moreover, without occurring of confusions  we also use $\tau_0$ to denote
 the canonical  anti-symplectic involutions on the standard
 symplectic spaces  $\mathbb{R}^{2l}$ of different dimensions.
By Lemma 2.29 in \cite{R12}, for every linear anti-symplectic involution $\tau$
on $(\mathbb{R}^{2n},\omega_0)$ there exists a linear symplectic isomorphism $\Psi$
of $(\mathbb{R}^{2n},\omega_0)$ such that
$\Psi \tau_0=\tau\Psi$, i.e., $\Psi$ is a linear real symplectic isomorphism
from $(\mathbb{R}^{2n},\omega_0,\tau)$ to $(\mathbb{R}^{2n},\omega_0,\tau_0)$.
Moreover, in this paper without special statements we always
%$\bullet$ use $J_0$ to denote the standard complex structure on $\mathbb{R}^{2n}$ given by  $J_0(x,y)=(-y,x)$,\\
%$\bullet$
identify a linear anti-symplectic involution $\tau$
on $(\mathbb{R}^{2n},\omega_0)$ with an anti-symplectic matrix $\tau\in\mathbb{R}^{2n\times 2n}$
(i.e. $\tau^TJ_0\tau=-J_0$) satisfying $\tau^2=I_{2n}$.
Hereafter  $\tau^T$ always denote  the adjoint of $\tau$  with respect to
the standard inner product $\langle\cdot,\cdot\rangle_{\mathbb{R}^{2n}}$.

%(Actually, there exist many anti-symplectic involutions on
%$(\mathbb{R}^{2n},\omega_0)$ for which  $(\mathbb{R}^{2n},\omega_0)$ is real.)

For a real symplectic manifold $(M,\omega,\tau)$ with nonempty real part $L$,
let $\mathcal{H}(M,\omega,\tau)$ denote the set of $\tau$-invariant smooth functions $H \colon M\to\R$ for which
there exists a nonempty open subset $U=U(H)$ with $L\cap U\ne\emptyset$ and a compact subset
$K=K(H)\subset M\setminus\partial M$ such that $H|_U=0$,
 $H|_{M\setminus K}=m(H):=\max H$ and  $0\leq H\leq m(H)$.
%\begin{description}
%  \item[(i)] $H|_U=0$,
%\item[(ii)] $H|_{M\setminus K}=\max H$,
% \item[(iii)]  $0\leq H\leq m(H)$,
% \item[(iv)] H(
%\end{description}
For $H\in\mathcal{H}(M,\omega,\tau)$, since $\tau^2=id_M$ the associated Hamiltonian vector field
 $X_H$  defined by $\omega(X_H, v)=-dH(v)$ for $v \in TM$ satisfies $X_H=-\tau^\ast X_H$, i.e.,
$X_{H}(x)=-d\tau(\tau(x))X_{H}(\tau(x))\;\;\forall
x\in M$. A $T$-periodic trajectory $x(t)$ of $X_H$ is called
a {\bf $\tau$-brake orbit} if
\begin{equation}\label{e:brake}
 x(T-t)=x(-t)=\tau(x(t)),\;\;\forall t\in\R.
\end{equation}
(This implies that $x(0), x(T/2)\in L$). A function $H\in\mathcal{H}(M,\omega,\tau)$ is said to be
{\bf admissible} if $X_H$
 has no nonconstant $\tau$-brake orbit with period $T\in (0,1]$.
Denote  by  $\mathcal{H}_{ad}(M,\omega,\tau)$ the set of admissible Hamiltonian functions.
As a refinement of the Hofer-Zehnder symplectic capacity for the
real symplectic manifold $(M,\omega,\tau)$, the quantity
  \begin{equation}\label{e:symcapacity}
  c_{\rm HZ,\tau}(M,\omega):=\sup\{m(H)\,|\,H\in\mathcal{H}_{ad}(M,\omega,\tau)\}
  \end{equation}
is called the {\bf symmetrical Hofer-Zehnder symplectic capacity} of a real symplectic manifold $(M,\omega,\tau)$, which was introduced and denoted by $c_\tau$ in \cite{LiuW12}.

\begin{remark}\label{rem:variantDefinition}
{\rm There exists a variant of $c_{\rm HZ,\tau}$.
 If the assumption ``$\tau$-invariant" is removed
in the definition of $\mathcal{H}(M,\omega,\tau)$ then the corresponding function space obtained
 is denoted by $\mathcal{H}^L(M,\omega)$. For a $H\in \mathcal{H}^L(M,\omega)$
 and  a solution  of $\dot\gamma=X_H(\gamma)$ with $\gamma(0)\in L$,
 $\gamma:\mathbb{R}\to M$, the {\bf return time} of $\gamma$ is defined by
 $T_\gamma=\inf\{t\,|\,t>0,\,\gamma(t)\in L\}$. We call $H\in\mathcal{H}^L(M,\omega)$
 {\bf admissible} if any solution $\gamma$ of $\dot\gamma=X_H(\gamma)$ with $\gamma(0)\in L$
 is either constant or has the return time $T_\gamma>1/2$. Denote by
 $\mathcal{H}^L_{ad}(M,\omega)$ the set of all admissible functions in
 $\mathcal{H}^L(M,\omega)$. Define
 \begin{equation}\label{e:symcapacityL}
  c^L_{\rm HZ}(M,\omega):=\sup\{m(H)\,|\,H\in\mathcal{H}^L_{ad}(M,\omega)\}.
  \end{equation}
Clearly, $c_{\rm HZ,\tau}(M,\omega)\le c^L_{\rm HZ}(M,\omega)$.
Note that $c^L_{\rm HZ}(M,\omega)$ is exactly the two times of
the coisotropic capacity $c(M, L,\omega,\sim)$ with trivial equivalence relation $\sim$ on $L$
 defined by Lisi and Rieser \cite[Definition~1.13]{LiRi13}. Thus
 some of our results can be naturally extended to their coisotropic capacity. }
\end{remark}

The following proposition collects some properties of
 $c_{\rm HZ,\tau}$.

\begin{proposition}\label{MonComf}
%\begin{description}
\begin{enumerate}
  \item[\bf (i)]{\rm (Conformality)}. $c_{\rm HZ,\tau}(M,\alpha\omega)=|\alpha|c_{\rm HZ,\tau}(M,\omega)$ for any $\alpha\in\mathbb{R}\setminus\{0\}$.
\vspace*{4pt}  \item[\bf (ii)]{\rm (Monotonicity)}. For two real symplectic manifolds $(M_1,\omega_1,\tau_1)$ and  $(M_2,\omega_2,\break\tau_2)$, if there exists a {\bf real symplectic embedding} $\psi:(M_1,\omega_1,\tau_1)\rightarrow (M_2,\omega_2,\break\tau_2)$ (i.e., a symplectic embedding $\psi$ satisfying $\psi\circ \tau_1=\tau_2\circ\psi$), then $c_{\rm HZ,\tau_1}(M_1,\break\omega_1)\le c_{\rm HZ,\tau_2}(M_2,\omega_2)$.
\vspace*{4pt}  \item[\bf (iii)]{\rm (Continuity)}.
    Given a bounded $\tau_0$-invariant convex domain $A\subset\mathbb{R}^{2n}$,  for every $\varepsilon>0$ there exists some $\delta>0$ such that
 \begin{equation}\label{realcap4+}
 |c_{\rm HZ,\tau_0}(O,\omega_0)-c_{\rm HZ,\tau_0}(A,\omega_0)|<\varepsilon
\end{equation}
for all bounded $\tau_0$-invariant convex domain $O\subset\mathbb{R}^{2n}$ whose
Hausdorff distance to $A$, $d_{\rm H}(A,O)$, is less than $\delta$.
%, it holds that
%provided that $A$ and $O$ have  the

\vspace*{4pt}\item[\bf (iv)]{\rm (Inner regularity)}.
\begin{equation}\label{e:innerregu}
c_{\rm HZ,\tau}(M,\omega)=\sup\{c_{\rm HZ,\tau}(U,\omega)\,|\,U\subset M\,\hbox{open},\,\tau U=U\,\hbox{and}\;U\in M\setminus\partial M\}.
\end{equation}
\end{enumerate}
%\end{description}
\end{proposition}

\vspace*{3pt}\begin{proof}
(i) and (ii) were proved in \cite[Theorem~2.4]{LiuW12}.
 For (iii), let $p\in A\cap L_0$. Replacing $A$ and $O$
with $A-p$ and $O-p$ respectively, we can assume $0\in A$. For any $0<\epsilon\ll1$,
by \cite[Lemma~1.8.14]{Sch93} there exists $\delta>0$ such that any bounded $\tau_0$-invariant convex domain $O\subset\mathbb{R}^{2n}$
with $d_{\rm H}(A,O)<\delta$ satisfies
$$
(1-\epsilon)A\subset O\subset (1+\epsilon)A.
$$
Then (iii) may easily follow from this and (i)--(ii).

 To prove (iv), we first assume that $c_{\rm HZ,\tau}(M,\omega)<\infty$. Then for any $\varepsilon>0$ there exists $H\in\mathcal{H}_{ad}(M,\omega,\tau)$ such that $m(H)>c_{\rm HZ,\tau}(M,\omega)-\varepsilon$. Since $H\circ\tau=H$, $\widehat{K}:={\rm supp}(H-m(H))$ is $\tau$-invariant. Let $\widehat{U}$ be a $\tau$-invariant open neighborhood of $\widehat{K}$. Then $H\in \mathcal{H}_{ad}(\widehat{U},\omega,\tau)$
and it follows that
$$
c_{\rm HZ,\tau}(\widehat{U},\omega)\ge m(H)>c_{\rm HZ,\tau}(M,\omega)-\varepsilon.
$$
 This shows (\ref{e:innerregu}). When $c_{\rm HZ,\tau}(M,\omega)=\infty$, similar arguments also leads to
(\ref{e:innerregu}).
\end{proof}\vspace*{3pt}

Recall that a smooth connected compact hypersurface $\mathcal{S}$ in $(\mathbb{R}^{2n},\omega_0)$
is said to be of {\bf restricted contact type} if there exists a vector field  $X$ on $\mathbb{R}^{2n}$
  which is transverse to $\mathcal{S}$  and  satisfies $\mathcal{L}_X\omega_0=\omega_0$.
 A vector field $X$ with the latter property is called  a global Liouville vector field on $\mathbb{R}^{2n}$.
 For a hypersurface $\mathcal{S}$ of restricted contact type in $(\mathbb{R}^{2n},\omega_0)$
 let $B_{\mathcal{S}}$ be the bounded component of $\mathbb{R}^{2n}\setminus \mathcal{S}$.
  Suppose that $B_{\mathcal{S}}$ is $\tau_0$-invariant, $B_{\mathcal{S}}\cap L_0\neq\emptyset$
  and $X(\tau_0z)=\tau_0(X(z))$ for all $z$ near ${\mathcal S}$.
 As in the proof of \cite[Proposition~2.3]{Ba95} we deduce that $c_{\rm HZ,\tau_0}$ has
the following exterior regularity:
{\small$$
c_{\rm HZ,\tau_0}(B_{\mathcal{S}},\omega_0)=\inf\{c_{\rm HZ,\tau_0}(V,\omega_0)\,|\,V\subset \mathbb{R}^{2n}\;\hbox{is open and $\tau_0$-invariant, and}\;\overline{B_{\mathcal{S}}}\subset V\}.
$$}
\vspace*{-4pt}
%(i) and (ii) were proved in \cite[Theorem~2.4]{LiuW12}, and (iii) follows from them easily. To prove (iv), we first assume that $c_{\rm HZ,\tau}(M,\omega)<\infty$. Then for any $\varepsilon>0$ there exists $H\in\mathcal{H}_{ad}(M,\omega,\tau)$ such that $m(H)>c_{\rm HZ,\tau}(M,\omega)-\varepsilon$. Since $H\circ\tau=H$, $\widehat{K}:={\rm supp}(H-m(H))$ is $\tau$-invariant. Let $\widehat{U}$ be a $\tau$-invariant open neighborhood of $\widehat{K}$. Then $H\in \mathcal{H}_{ad}(\widehat{U},\omega,\tau)$
%and it follows that $c_{\rm HZ,\tau}(\widehat{U},\omega)\ge m(H)>c_{\rm HZ,\tau}(M,\omega)-\varepsilon$. This shows (\ref{e:innerregu}). When $c_{\rm HZ,\tau}(M,\omega)=\infty$, similar arguments also leads to
%(\ref{e:innerregu}).

Note that $c_{\rm HZ,\tau}(M,\omega)$ is only invariant for symplectomorphisms on $(M,\omega)$
commuting with $\tau$.

 A {\bf $\tau$-brake closed characteristic} ({\bf $\tau$-BCC}, for short) on a $\tau$-invariant smooth hypersurface $\mathcal{ S}$ satisfying $\mathcal{S}\cap{\rm Fix}(\tau)\ne\emptyset$ in a real symplectic manifold $(M,\omega,\tau)$ is  a $C^1$ embedding $x$ from $\mathbb{R}/T\mathbb{Z}$ {\rm (for some $T>0$)} into $\mathcal{S}$ satisfying
 $$
 \dot{x}(t)\in({\mathcal L}_{\mathcal{S}})_{x(t)}\quad\hbox{and}\quad x(T-t) =\tau(x(t)),\;\forall t\in \mathbb{R},
$$
where ${\mathcal L}_{\mathcal{S}}$ is the characteristic line bundle on $\mathcal{S}$
whose fiber at $x\in \mathcal{S}$ is given  by
$$
({\mathcal L}_{\mathcal{S}})_x=\{v\in T_x\mathcal{S}\,|\,\omega_0(v,u)=0\;\forall u\in T_x\mathcal{S}\}.
$$
%defined by
%$({\mathcal L}_{\mathcal{S}})_x=\{v\in T_x\mathcal{S}\,|\,\omega_0(v,u)=0\;\forall u\in T_x\mathcal{S}\}$ for any $x\in \mathcal{S}$.
Moreover, given  a linear anti-symplectic involution  $\tau$ on $(\mathbb{R}^{2n},\omega_0)$
and a $\tau$-invariant convex body $D\subset \mathbb{R}^{2n}$ with boundary $\mathcal{S}$,
a nonconstant  absolutely continuous curve $x:\mathbb{R}/T\mathbb{Z}\to \mathcal{S}$ (for some $T>0$)
 is called a {\bf generalized $\tau$-brake closed characteristic}
 ({\bf generalized $\tau$-BCC}, for short)  on $\mathcal{S}$ if
  $$
  \dot{x}(t)\in J_0N_{\mathcal S}(x(t))\;\hbox{a.e. on $\mathbb{R}$}\quad\hbox{and}\quad x(T-t)=\tau (x(t))\;\forall t\in\mathbb{R},
  $$
   where  $N_{\mathcal S}(x)=\{y\in\mathbb{R}^{2n}\,|\, \langle u-x, y\rangle\le 0\;\forall u\in D\}$
    is the normal cone to $D$ at $x\in\mathcal{S}$.

Clearly, if ${\mathcal S}$ in the latter case is also  $C^{1,1}$  then  any
generalized $\tau$-brake closed characteristic on $\mathcal{S}$ may become a $\tau$-brake closed characteristic on $\mathcal{S}$
via reparame\-trization.
Let us define the {\bf action} of a $W^{1,1}$ curve
 $x:[0,T]\rightarrow \mathbb{R}^{2n}$  by
\begin{equation}\label{e:action1}
A(x)=\frac{1}{2}\int_0^T\langle-J_0\dot{x},x\rangle_{\mathbb{R}^{2n}} dt=
-\frac{1}{2}\int_0^T\omega_0(\dot{x},x)dt.
\end{equation}
 Note that $A(x)$ is invariant under reparameterization of $x$.
The following is  an analogue of the representation formula for $c_{\rm HZ}$ due to Hofer and Zehnder \cite[Propposition~4]{HoZe90},
and will be proved in Section~\ref{sec:convex}.

\begin{thm}\label{th:convex}
Let $\tau$ be a linear anti-symplectic involution on $(\mathbb{R}^{2n},\omega_0)$,
and let $D\subset\mathbb{R}^{2n}$ be a $\tau$-invariant convex bounded domain. Then there is a generalized $\tau$-brake characteristic
$x^{\ast}$ on $\partial D$ such that
\begin{eqnarray}
A(x^{\ast})&=&\min\{A(x)>0\,|\,x\;\text{is a generalized $\tau$-BCC on}\;\partial D\}\label{e:action2}\\
&=&c_{\rm HZ,\tau}(D,\omega_0)\label{e:action-capacity1}.
\end{eqnarray}
If $\partial D$ is of class $C^{1,1}$, then
there is a $\tau$-BCC $x^{\ast}$ on $\partial D$ such that
   \begin{equation}\label{e:action-capacity2}
   c_{\rm HZ,\tau}(D,\omega_0)=A(x^{\ast})=\min\{A(x)>0\,|\,x\;\text{is a $\tau$-BCC on}\;\partial D\}.
   \end{equation}
   \end{thm}

%\begin{remark}\label{rem:convex}
%{\rm By Lemma 2.29 in \cite{R12}, for every linear anti-symplectic involution $\tau$
%on $(\mathbb{R}^{2n},\omega_0)$ there exists a linear real symplectic isomorphism $\Psi$
%from $(\mathbb{R}^{2n},\omega_0,\tau)$ to $(\mathbb{R}^{2n},\omega_0,\tau_0)$.
%Note that $D$ is $\tau_0$-invariant if and only if $\Psi(D)$  is $\tau$-invariant,
%and that $x$ is a (generalized) $\tau_0$-brake closed characteristic on $\partial D$
%if and only if $\Psi(x)$ a (generalized) $\tau$-brake closed characteristic on $\partial \Psi(D)=\Psi(\partial D)$.
% Moreover, it is clear that $A(\Psi(x))=A(x)$ by (\ref{e:action1}).
%These, Proposition~\ref{MonComf}(ii) and Theorem~\ref{th:convex} imply:
% Theorem~\ref{th:convex} also holds if $\tau_0$ is replaced by any
% linear anti-symplectic involution $\tau$ on $(\mathbb{R}^{2n},\omega_0)$.
%}
%\end{remark}

\subsection{Symmetrical Ekeland-Hofer capacity}\label{sec:1.EH}

For $s\ge 0$ and $S^1=\mathbb{R}/\mathbb{Z}$ consider the Hilbert space
$$
E^{s}=\left\{x\in L^2(S^1;\mathbb{R}^{2n})\,\Bigm|\,x=\sum_{j\in\mathbb{Z}}e^{2\pi jtJ_0}x_j,\,x_j\in\mathbb{R}^{2n},\,\sum_{j\in\mathbb{Z}}
|j|^{2s}|x_j|^2<\infty\right\}
$$
with inner product and associated norm given by
\begin{eqnarray}\label{innerproduct}
&&\langle x,y\rangle_s=\langle x_0,y_0\rangle_{\mathbb{R}^{2n}}+ 2\pi \sum_{j\in\mathbb{Z}}
|j|^{2s}\langle x_j, y_j\rangle_{\mathbb{R}^{2n}},\nonumber\\
&&\|x\|_s^2=\langle x,x\rangle_s
\end{eqnarray}
(cf. \cite{HoZe94}). $\tau_0$ induces a natural Hilbert space isomorphism
$$
\check{\tau}_0:E^{s}\to E^{s},\;x=\sum_{j\in\mathbb{Z}}e^{2\pi jtJ_0}x_j\mapsto \sum_{j\in\mathbb{Z}}e^{2\pi jtJ_0}\tau_0x_j.
$$
Its fixed point set is the following closed linear subspace
\begin{equation}\label{e:spacees}
E^s_{\tau_0}=\left\{x\in L^2(S^1;\mathbb{R}^{2n})\,\Bigm|\,x=\sum_{j\in\mathbb{Z}}e^{2\pi jtJ_0}x_j,\,x_j\in L_0,\,\sum_{j\in\mathbb{Z}}
|j|^{2s}|x_j|^2<\infty\right\}.
\end{equation}
For the sake of simplicity, from now on we write $\mathbb{E}:=E^{1/2}_{\tau_0}$.   It  has the orthogonal splitting
\begin{equation}\label{e:spaceDecomposition}
\mathbb{E}=\mathbb{E}^{-}\oplus \mathbb{E}^0 \oplus \mathbb{E}^{+},
\end{equation}
where $\mathbb{E}^0=L_0$ and
$$
\mathbb{E}^-=\{x\in \mathbb{E}\,|\,x=\sum_{j<0}e^{2\pi jtJ_0}x_j\},\qquad \mathbb{E}^+=\{x\in \mathbb{E}\,|\,x=\sum_{j>0}e^{2\pi jtJ_0}x_j\}.
$$
Denote the associated projection on them by $P^-, P^0$ and $P^+$.
 For $x\in \mathbb{E}$, write
$x=x^-+x^0+x^+$, where $x^-\in \mathbb{E}^-$, $x^0\in \mathbb{E}^0$ and $x^+\in \mathbb{E}^+$.
Let $S^+$ be the unit sphere in $\mathbb{E}^+$.

%Let $\mathbb{E}\equiv E^{1/2}_{\tau_0}$ be as in (\ref{e:spacees}). It has the orthogonal splitting (\ref{e:spaceDecomposition}).
We closely follow  Sikorav's approach (\cite{Sik90}) to define Ekeland-Hofer capacity in \cite{EH89}.

\begin{definition}\label{def:deform}
  A continuous map $\gamma:\mathbb{E}\rightarrow \mathbb{E}$ is called an {\bf admissible deformation} if there exists an homotopy $(\gamma_u)_{0\le u\le 1}$ such that $\gamma_0={\rm id}$, $\gamma_1=\gamma$ and  the following conditions hold:
%\begin{description}
\begin{enumerate}
\item[\bf (i)] $\forall u\in [0,1]$, $\gamma_u(\mathbb{E}\setminus(\mathbb{E}^-\oplus \mathbb{E}^0))=\mathbb{E}\setminus(\mathbb{E}^-\oplus \mathbb{E}^0)$, i.e. for any $x\in \mathbb{E}$ such that $x^+\neq 0$, there holds $\gamma_u(x)^+\neq 0$.
\item[\bf (ii)] $\forall u\in [0,1]$, $\gamma_u(x)=a(x,u)x^++b(x,u)x^0+c(x,u)x^-+K(x,u)$, where $(a,b,c,\break K)$ is a continuous map from $\mathbb{E}\times [0,1]$ to $(0,+\infty)^3\times \mathbb{E}$ and maps any bounded sets to precompact sets.
    \end{enumerate}
%\end{description}
\end{definition}
Let $\Gamma$ be the set of all admissible deformations, and $\mathbb{R}_{+}=\{\lambda\in\mathbb{R}\,|\,\lambda\ge 0\}$ below.
For $H\in C^0(\mathbb{R}^{2n},\mathbb{R}_{+})$ satisfying:
%\begin{description}
\begin{enumerate}
\item[\bf (H1)] $H(z)=H(\tau_0z)\;\forall z\in\mathbb{R}^{2n}$,
\item[\bf (H2)] ${\rm Int}(H^{-1}(0))\cap L_0\ne \emptyset$,
\item[\bf (H3)] there exist $z_0\in L_0$, real numbers $a> \pi$ and $b$
such that $H(z)=a|z|^2+ \langle z, z_0\rangle+ b$ outside a compact subset of $\mathbb{R}^{2n}$,
\end{enumerate}
%\end{description}
we define $\Phi_H:\mathbb{E}\to\mathbb{R}$ by
\begin{eqnarray}\label{e:EH.1.1}
\Phi_H(x)= \frac{1}{2}(\|x^+\|^2_{1/2}-\|x^-\|^2_{1/2})-\int_0^1H(x(t))dt,
\end{eqnarray}
 and the {\bf $\tau_0$-symmetrical Ekeland-Hofer capacity} of $H$ by
\begin{eqnarray}\label{e:EH.1.2}
c_{\rm EH,\tau_0}(H)=\sup_{h\in\Gamma}\inf_{x\in h(S^+)}\Phi_H(x).
\end{eqnarray}
Then (H3) implies $c_{\rm EH,\tau_0}(H)<+\infty$  by
Proposition~\ref{prop:EH.1.3}. In fact there exists a constant $C>0$
such that $H(z)-a|z|^2-\langle z,z_0\rangle-b\ge -C$.
%Since  $a>2\pi$,
%using the inequality
%$$
%a|z|^2+\langle z,z_0\rangle+b\ge\frac{a}{2}|z|^2-\frac{|z_0|^2}{2a}+b
%$$
%we deduce that
%$$
%\pi |z_1|^2-H(z)\le \pi |z|^2-H(z)\le \pi |z|^2-(\frac{a}{2}|z|^2-\frac{|z_0|^2}{2a}+b)+C\le \frac{|z_0|^2}{2a}-b+C<\infty.
%$$
Since  $a>\pi$,
fixing $0<\varepsilon<a-\pi$ and using the inequality
$$
|\langle z,z_0\rangle|\le\varepsilon|z|^2+\frac{1}{4\varepsilon}|z_0|^2
$$
we deduce that
$$
-H(z)\le -a|z|^2-\langle z,z_0\rangle-b+ C\le -a|z|^2+ \varepsilon|z|^2+\frac{1}{4\varepsilon}|z_0|^2-b+ C
$$
and hence
\begin{eqnarray*}
\pi |z_1|^2-H(z)\le \pi |z|^2-H(z)&\le& \pi |z|^2-a|z|^2+ \varepsilon|z|^2+\frac{1}{4\varepsilon}|z_0|^2-b+ C\\
&\le&\frac{1}{4\varepsilon}|z_0|^2-b+ C<\infty.
\end{eqnarray*}
Moreover,  (H2)-(H3) imply $c_{\rm EH,\tau_0}(H)>0$  by Proposition~\ref{prop:EH.1.4}.
It is easy to prove:

\begin{proposition}\label{prop:EH.1.2}%[\hbox{\cite[Prop.3.2.1]{Sik90}}]
Let $H, K\in C^0(\mathbb{R}^{2n},\mathbb{R}_{+})$ satisfy (H1)-(H3). The value $c_{\rm EH,\tau_0}$ has the following properties:
%\begin{description}
\begin{enumerate}
\item[\bf (i)]{\rm (Monotonicity).} If $H\le K$ then $c_{\rm EH,\tau_0}(H)\ge c_{\rm EH,\tau_0}(K)$.
\item[\bf (ii)] {\rm (Continuity).} $|c_{\rm EH,\tau_0}(H)-c_{\rm EH,\tau_0}(K)|\le \sup_{z\in\mathbb{R}^{2n}}|H(z)-K(z)|$.
\item[\bf (iii)] {\rm (Homogeneity).}  $c_{\rm EH,\tau_0}(\lambda^2H(\cdot/\lambda))=\lambda^2 c_{\rm EH,\tau_0}(H)$ for $\lambda> 0$.
\end{enumerate}
%\end{description}
\end{proposition}
Let
\begin{eqnarray}\label{e:EH.1.5.1}
&&\mathscr{F}(\mathbb{R}^{2n},\tau_0)=\{H\in C^0(\mathbb{R}^{2n},\mathbb{R}_+)\,|\,H\;\hbox{satisfies (H1)-(H3)}\},
\\
&&\mathscr{F}(\mathbb{R}^{2n},\tau_0, B)=\{H\in \mathscr{F}(\mathbb{R}^{2n}, \tau_0)\,|\,H\;\hbox{vanishes near $\overline{B}$}\}\label{e:EH.1.5.2}
\end{eqnarray}
for each $B\subset\mathbb{R}^{2n}$ such that $\tau_0B=B$ and $B\cap L_0\neq \emptyset$.
 We define
\begin{equation}\label{e:EH.1.6}
c_{\rm EH,\tau_0}(B)=\inf\{c_{\rm EH,\tau_0}(H)\,|\, H\in \mathscr{F}(\mathbb{R}^{2n},\tau_0, B)\}
\end{equation}
if $B$ is bounded, $\tau_0$-invariant, and intersecting with $L_0$.
%$B\cap L_0\neq \emptyset$.
And define
\begin{equation}\label{e:EH.1.7}
c_{\rm EH,\tau_0}(B)=\sup\{c_{\rm EH,\tau_0}(B_0)\,|\, B_0\subset B,\;\hbox{$B_0$ is bounded, $\tau_0B=B$ and $B_0\cap L_0\neq \emptyset$}\}
\end{equation}
if $B$ is unbounded, $\tau_0$-invariant, and intersecting with $L_0$.
%$\tau_0B=B$ and satisfies $B\cap L_0\neq \emptyset$.
The value $c_{\rm EH,\tau_0}(B)$ is called the {\bf $\tau_0$-symmetrical Ekeland-Hofer capacity} of $B$.

We say $H\in C^2(\mathbb{R}^{2n},\mathbb{R}_+)$ to be $\tau_0$-{\bf nonresonant} if it satisfies (H3) with $a\notin \mathbb{Z}\pi$.
For each $B\subset\mathbb{R}^{2n}$ such that $\tau_0B=B$ and $B\cap L_0\ne\emptyset$, we write
\begin{eqnarray*}
\mathscr{E}(\mathbb{R}^{2n},\tau_0, B)=\{H\in \mathscr{F}(\mathbb{R}^{2n},\tau_0, B)\,|\,H\;\hbox{is $\tau_0$ -nonresonant}\}.
\end{eqnarray*}
Notice that
$\mathscr{E}(\mathbb{R}^{2n},\tau_0, B)$ is a {\bf cofinal family} of $\mathscr{F}(\mathbb{R}^{2n},\tau_0,
B)$, that is, for any $H\in\mathscr{F}(\mathbb{R}^{2n},\tau_0, B)$ there exists a function
$G\in \mathscr{E}(\mathbb{R}^{2n},\tau_0, B)$ such that $G\ge H$.
\begin{remark}\label{rem:EH6}
{\rm
%\begin{description}
\begin{enumerate}
\item[\bf (i)] $c_{\rm EH,\tau_0}(B)=c_{\rm EH,\tau_0}(\overline{B})$.
\item[\bf (ii)]  $\mathscr{F}(\mathbb{R}^{2n},\tau_0, B)$ in  (\ref{e:EH.1.6})-(\ref{e:EH.1.7}) can be replaced by its cofinal subset $\mathscr{E}(\mathbb{R}^{2n},\tau_0,\break B)$, and
    can also be replaced by a smaller cofinal subset
   $\mathscr{E}(\mathbb{R}^{2n},\tau_0, B)\cap C^\infty(\mathbb{R}^{2n},\mathbb{R}_+)$.
%\item[(iii)]  $c_{\rm EH,\tau_0}(B+w)=c_{\rm EH,\tau_0}(B)\;\forall w\in L_0$,
%where $B+w=\{z+w\,|\, z\in B\}$.
   %   \end{description}
   \end{enumerate}}
\end{remark}

%[\hbox{\cite[Prop.4.1.1, Th.4.2.1]{Sik90}}]

\begin{proposition}\label{prop:EH.1.7}
Suppose that $\tau_0$-invariant subsets $B\subset B'\subset\mathbb{R}^{2n}$
have nonempty intersection with $L_0$. Then the following holds:
%\begin{description}
\begin{enumerate}
\item[\bf (i)]{\rm (Translation invariance by elements of $L_0$)}. $c_{\rm EH,\tau_0}(B+w)=c_{\rm EH,\tau_0}(B)\;\forall w\in L_0$,
where $B+w=\{z+w\,|\, z\in B\}$.
\item[\bf (ii)]{\rm (Monotonicity)}. $c_{\rm EH,\tau_0}(B)\le c_{\rm EH,\tau_0}(B')$;
\item[\bf (iii)] {\rm (Conformality)}. $c_{\rm EH,\tau_0}(\lambda B)=\lambda^2 c_{\rm EH,\tau_0}(B)\;\forall
\lambda\in\mathbb{R}_+$;
\item[\bf (iv)] {\rm (Exterior regularity)}. $c_{\rm EH,\tau_0}(B)=\inf\{c_{\rm EH,\tau_0}(U_\epsilon(B))\,|\,\epsilon>0\}$,
where $U_\epsilon(B)$ is the $\epsilon$-neighborhood of $B$,  which is invariant under $\tau_0$
since $\tau_0$ is an isometry.
 \end{enumerate}
%\end{description}
\end{proposition}

Let $\mathcal{S}, B_{\mathcal{S}}\subset \mathbb{R}^{2n}$
 and $X$  be as below the proof of Proposition~\ref{MonComf}.
Using the similar proof to that of \cite[Proposition~2.3]{Ba95} we can obtain
the following inner regularity of $c_{\rm EH,\tau_0}$:
$$
c_{\rm EH,\tau_0}(B_{\mathcal{S}},\omega_0)=\sup\left\{c_{\rm EH,\tau_0}(V,\omega_0)\;\Bigg|\!\!\!\!\begin{array}{ll}
&V\subset \mathbb{R}^{2n}\;\hbox{is open and
$\tau_0$-invariant},\\
&\overline{V}\subset B_{\mathcal{S}},\;V\cap L_0\ne\emptyset
\end{array}\right\}.
$$

\begin{proof}[Proof of Proposition~\ref{prop:EH.1.7}]
We only prove (i). The others are easy.
In fact, for any $H\in\mathscr{F}(\mathbb{R}^{2n},\tau_0,B)$, define $\widehat{H}(z):=H(z-w)$, $\forall z\in\mathbb{R}^{2n}$. By the assumption $H$ is $\tau_0$-invariant and vanishes near $B$, and there holds $H(z)=a|z|^2+\langle z,z_0\rangle +b$ for $|z|$ sufficiently large, where $a>2\pi$, $z_0\in L_0$. Since $w\in L_0$, it follows that
$$
\widehat{H}(\tau_0z)=H(\tau_0z-w)=H(\tau_0(z-w))=H(z-w)=\widehat{H}(z).
$$
It is obvious that $\widehat{H}$ vanishes near $B+w$. Moreover,
when $|z|$ is sufficiently large,
\begin{eqnarray*}
\widehat{H}(z)&=&H(z-w)=a|z-w|^2+\langle z-w,z_0\rangle+b\\
&=&a|z|^2+\langle z,z_0-2aw\rangle+a|w|^2-\langle w,z_0\rangle+b,
\end{eqnarray*}
where $z_0-2aw\in L_0$. Hence $\widehat{H}\in\mathscr{F}(\mathbb{R}^{2n},\tau_0,B+w)$. Also we have
{\small\begin{eqnarray*}
c_{\rm EH,\tau_0}(\widehat{H})&=&\sup_{h\in\Gamma}\inf_{x\in h(S^+)}\Phi_{\widehat{H}}(x)\\
&=&\sup_{h\in\Gamma}\inf_{x\in h(S^+)}\frac{1}{2}(\|x^+\|^2_{1/2}-\|x^-\|^2_{1/2})-\int_0^1\widehat{H}(x(t))dt\\
&=& \sup_{h\in\Gamma}\inf_{x\in h(S^+)}\frac{1}{2}(\|(x-w)^+\|^2_{1/2}-\|(x-w)^-\|^2_{1/2})-\int_0^1H(x(t)-w)dt\\
&=& \sup_{h\in\Gamma}\inf_{x\in h(S^+)-w}\frac{1}{2}(\|x^+\|^2_{1/2}-\|x^-\|^2_{1/2})-\int_0^1H(x(t))dt\\
&=& \sup_{h\in\Gamma}\inf_{x\in h(S^+)}\Phi_H(x)\\
&=& c_{\rm EH,\tau_0}(H).
\end{eqnarray*}}Hence $c_{\rm EH,\tau_0}(B+w)\le c_{\rm EH,\tau_0}(B)$.
This and arbitrariness of choices of $w\in L_0$ lead to the conclusion.
\end{proof}

In Section~\ref{sec:EH.2} we shall prove the following result, which gives the
variational explanation for $c_{\rm EH,\tau_0}$.

\begin{thm}\label{th:EH.1.6}%[\hbox{\cite[Prop.3.4.1]{Sik90}}]
If $H\in C^\infty(\mathbb{R}^{2n},\mathbb{R})$ satisfies (H1)-(H3)
in Section~\ref{sec:1.EH}
and is also nonresonant, then $c_{\rm EH,\tau_0}(H)$ is a positive critical value of $\Phi_H$ on $\mathbb{E}$.
\end{thm}
This theorem is very important
for proofs of subsequent Theorems~\ref{th:EHconvex},~\ref{th:EHproduct} and \ref{th:EHcontact}.

%
%In Theorem~\ref{th:EH.1.6} we shall give the
%variational explanation for $c_{\rm EH,\tau_0}$,

By Lemma 2.29 in \cite{R12}, for every linear anti-symplectic involution $\tau$
on $(\mathbb{R}^{2n},\omega_0)$ there exists a symplectic matrix $\Psi$
of order $2n$  such that $\Psi \tau=\tau_0\Psi$. If $B\subset\mathbb{R}^{2n}$ is a
$\tau$-invariant subset (and thus $\Psi B$ is $\tau_0$-invariant), we define
 \begin{equation}\label{e:EH.1.8}
c_{\rm EH,\tau}(B)=c_{\rm EH,\tau_0}(\Psi B).
 \end{equation}
 In order to ensure that this is well-defined we need to prove
 \begin{equation}\label{e:EH.1.9}
c_{\rm EH,\tau_0}(\Psi_1 B)=c_{\rm EH,\tau_0}(\Psi_2 B)
\end{equation}
 for any two symplectic matrixs $\Psi_i$ of order $2n$  satisfying $\Psi_i \tau=\tau_0\Psi_i$, $i=1,2$.
 Since $\Psi:=\Psi_2\Psi_1^{-1}$ is a symplectic matrix
  satisfying $\Psi\tau_0=\tau_0\Psi$ and $\Psi(\Psi_1B)=\Psi_2B$,
   (\ref{e:EH.1.9}) can be assured by the following proposition.

\begin{proposition}\label{prop:EH.1.8}
  For a $\tau_0$-invariant set $B\subset\mathbb{R}^{2n}$  and a symplectic matrix $\Psi$ commuting with $\tau_0$, there holds
  $$
  c_{\rm EH,\tau_0}(B)=c_{\rm EH,\tau_0}(\Psi B).
  $$
\end{proposition}

This will be proved at the end of Section~\ref{sec:EH.2}.

\begin{proposition}\label{prop:EH.1.9}
 Proposition~\ref{prop:EH.1.7} is still true if $\tau_0$ and $L_0$
 are replaced by any linear anti-symplectic involution $\tau$ on $(\mathbb{R}^{2n},\omega_0)$
and $L={\rm Fix}(\tau)$, respectively.
\end{proposition}

The following result will be proved in Section~\ref{sec:EH.3.1}.

\begin{thm}\label{th:EHconvex}
For any linear anti-symplectic involution $\tau$ on $(\mathbb{R}^{2n},\omega_0)$
and any $\tau$-invariant convex bounded domain $D\subset \mathbb{R}^{2n}$  with
 $C^{1,1}$  boundary $S=\partial D$, then  there exists a $\tau$-brake closed characteristic $x^{\ast}$ on $\partial D$ such that
 \begin{eqnarray}\label{e:fixpt}
A(x^{\ast})&=&\min\{A(x)>0\,|\,x\;\text{is a }\;\hbox{$\tau$-brake closed characteristic on}\;\mathcal{S}\}\nonumber\\
&=&c_{\rm EH,\tau}(D)\\
&=&c_{\rm EH,\tau}(\partial D).\label{e:fixpt.1}
\end{eqnarray}
\end{thm}
%%%%%%%%%%%%%%%%%%%%%%%%%%%%
%\begin{remark}
%  {\rm We can replace $\tau_0$ in the above theorem by any linear anti-symplectic involution $\tau$ on $(\mathbb{R}^{2n},\omega_0)$.
%   }
%\end{remark}
%%%%%%%%%%%%%%%%%%%%%%%%%%%

\begin{remark}\label{rm:EHconvex}
{\rm   The argument at the beginning of Section~\ref{sec:convex} shows that the convexity and $\tau$-invariance of $D$ %and $\tau D=D$
imply that $\tau\partial D=\partial D$, $D\cap {\rm Fix}(\tau)\neq \emptyset$ and $\partial D\cap{\rm Fix}(\tau)\ne \emptyset$.}
 \end{remark}

Note that Proposition~\ref{prop:EH.1.7}(ii)-(iii)  leads to the
continuity of $c_{\rm EH,\tau_0}$ (and thus $c_{\rm EH,\tau}$)  in the category of convex sets.
As in Section~\ref{sec:convex4} an approximation argument shows that
the condition ``$C^{1,1}$" for $D$ in  Theorem~\ref{th:EHconvex} is
not needed if ``brake closed characteristic" is replaced by
``generalized brake closed characteristic".
Thus  Theorems~\ref{th:convex} and \ref{th:EHconvex} imply
$c_{\rm EH,\tau}(D)=c_{\rm HZ,\tau}(D,\omega_0)$
for any linear anti-symplectic involution $\tau$ on $(\mathbb{R}^{2n},\omega_0)$
and any $\tau$-invariant bounded convex  domain
$D\subset \mathbb{R}^{2n}$. By the definitions of both $c_{\rm EH,\tau}$ and $c_{\rm HZ,\tau}$
we deduce that for any $\tau$-invariant convex domain $D\subset \mathbb{R}^{2n}$,
\begin{eqnarray}\label{e:fixpt.2}
c_{\rm EH,\tau}(D)=c_{\rm HZ,\tau}(D,\omega_0).
\end{eqnarray}
Hereafter we shall use $c_{\rm EHZ,\tau}(D)$ to denote $c_{\rm EH,\tau}(D)=c_{\rm HZ,\tau}(D,\omega_0)$
  without special statements. A generalized $\tau$-brake closed characteristic on $\mathcal{S}$
   is  called a {\bf $c_{\rm EHZ,\tau}$-carrier} if its
   action is equal to $c_{\rm EHZ,\tau}(D)$.
The proof of Theorem~\ref{th:convex} also shows that
a generalized $\tau$-brake closed characteristic  on $\mathcal{S}$
 is a $c_{\rm EHZ,\tau}$-carrier for $D$ if and only if
it may be reparametrized as a $T$-periodic solution $x$ of
$-J_0\dot{x}(t)\in  \partial H(x(t))$ with $T=c_{\rm EHZ,\tau}(D)$ and satisfying
$x(T-t) =\tau(x(t))\;\forall t\in \mathbb{R}$, where $H=j^2_D$ and $j_D:\mathbb{R}^{2n}\to\mathbb{R}$ is the
{\bf Minkowski functional} (or {\bf gauge function}) of $D$ defined by
\begin{equation}\label{e:gauge}
j_D(x)=\inf\{\lambda\ge 0\,|\, x\in\lambda D\}.
\end{equation}
Hence from (\ref{e:homog3}) and Arzela-Ascoli theorem it follows that
all $c_{\rm EHZ,\tau}$-carriers for $D$ form a compact subset in $C^0(\mathbb{R}/\mu\mathbb{Z},\mathcal{S})$
(and $C^1(\mathbb{R}/\mu\mathbb{Z},\mathcal{S})$ if $\mathcal{S}$ is $C^1$), where $\mu=c_{\rm EHZ,\tau}(D)$.

Let $E(q):=\{z\in\mathbb{R}^{2n}\,|\, q(z)<1\}$  be the ellipsoid given
by a positive definite quadratic form $q(z)=\frac{1}{2}\langle Sz, z\rangle_{\mathbb{R}^{2n}}$ on $\mathbb{R}^{2n}$,
where $S\in\mathbb{R}^{2n\times 2n}$ is a positive definite symmetric matrix.
By the Williamson theorem there exists a symplectic matrix $\Psi$ (which is unique up to
compositions with orthogonal symplectic matrixes) such that
\begin{eqnarray}\label{e:ellipsoid.1}
\Psi^TS\Psi={\rm diag}\left(1/{r_1^2},\cdots, 1/{r_n^2}, 1/{r_1^2},\cdots, 1/{r_n^2}\right)
\end{eqnarray}
with $0\le r_1\le\cdots\le r_n$. It follows that
\begin{eqnarray}\label{e:ellipsoid.2}
\Psi^{-1}(E(q))=E(q\circ\Psi)=\left\{(x,y)\in\mathbb{R}^{2n}\,\Bigg|\, \sum^n_{j=1}
(x_j^2+y_j^2)/{r_j^2}<1\right\}.
\end{eqnarray}
Since $q$ is the square of the Minkowski functional of $E(q)$, for a linear anti-symplectic involution $\tau$
on $(\mathbb{R}^{2n},\omega_0)$ it is easily proved that
\begin{eqnarray}\label{e:ellipsoid.3}
\tau(E(q))=E(q)\quad\Longleftrightarrow\quad q\circ\tau=q\quad \Longleftrightarrow\quad \tau^TS\tau=S.
\end{eqnarray}
If $\tau$ is also symmetric with respect to $\langle\cdot,\cdot\rangle_{\mathbb{R}^{2n}}$, the last equality above becomes $S\tau=\tau S$, i.e. $S$ commutes with $\tau$.

By \cite[Lemma~1.17]{R12} there always exists a $\omega_0$-compatible linear complex structure on
$J$ on $\mathbb{R}^{2n}$ with $\tau J=-J\tau$. (Note that $J$ can be chosen as the
standard complex structure $J_0$ if $\tau$ is symmetric with respect to the standard inner product
$\langle\cdot,\cdot\rangle_{\mathbb{R}^{2n}}=\omega_0(\cdot,J_0\cdot)$ on $\mathbb{R}^{2n}$.)
Let $u_1,\cdots,u_n$ be a unit orthogonal basis of $L:={\rm Fix}(\tau)$
with respect to the metric $g_J=\omega_0(\cdot,J\cdot)$, and $v_j=Ju_j$, $j=1,\cdots,n$.
Then $u_1,\cdots,u_n,v_1,\cdots,v_n$ form a symplectic basis of $(\mathbb{R}^{2n},\omega_0)$
which is also orthogonal with respect to $g_J$.
(If $\tau=\tau_0$ and $J=J_0$ we may choose $u_1,\cdots,u_n, v_1,\cdots,v_n$ to be
the standard symplectic basis $e_1,\cdots,e_n,f_1,\cdots,f_n$ in $(\mathbb{R}^{2n},\omega_0)$.)
Then the linear symplectic isomorphism $\Psi:(\mathbb{R}^{2n},\omega_0)\to
(\mathbb{R}^{2n},\omega_0)$ defined by $\Psi(u_j)=e_j$ and $\Psi(v_j)=f_j$, $j=1,\cdots,n$,
satisfies $\Psi \tau=\tau_0\Psi$. (See \cite[Lemma~2.29]{R12} for a different proof of this fact).
In particular, if $\tau$ is symmetric with respect to $\langle\cdot,\cdot\rangle_{\mathbb{R}^{2n}}$
  we may take $J=J_0$ so that this $\Psi$ is also an orthogonal transformation
with respect to $\langle\cdot,\cdot\rangle_{\mathbb{R}^{2n}}$.

As a consequence of Theorem~\ref{th:convex} or Theorem~\ref{th:EHconvex}
we get:

\begin{corollary}\label{cor:ellipsoid}
Let $E(q)=\{z\in\mathbb{R}^{2n}\,|\, q(z)<1\}$  be the ellipsoid given
by a positive definite symmetric matrix $S$ as above.
For a symplectic matrix $\Psi$ satisfying (\ref{e:ellipsoid.1})
and any linear anti-symplectic involution $\tau$
on $(\mathbb{R}^{2n},\omega_0)$ preserving $E(q)$ it holds that
\begin{eqnarray}\label{e:ellCap}
{c}_{\rm EHZ,\tau}(E(q))={c}_{\rm EHZ,\Psi^{-1}\tau\Psi}(\Psi^{-1}(E(q))).
\end{eqnarray}
In particular, if $\Psi^{-1}\tau\Psi=\tau_0$ then ${c}_{\rm EHZ,\tau}(E(q))=\pi r_1^2$. Moreover
  \begin{equation}\label{e:ball}
  c_{\rm EHZ,\tau}(B^{2n}(r))=\pi r^2\quad\forall r>0
  \end{equation}
for  any linear anti-symplectic involution $\tau$
on $(\mathbb{R}^{2n},\omega_0)$ which is symmetric with respect to $\langle\cdot,\cdot\rangle_{\mathbb{R}^{2n}}$.
\end{corollary}

Equation (\ref{e:ball}) can be obtained as follows.
Since any linear involution on $\mathbb{R}^{2n}$
which is  symmetric with respect to $\langle\cdot,\cdot\rangle_{\mathbb{R}^{2n}}$
must be an orthogonal transformation with respect to $\langle\cdot,\cdot\rangle_{\mathbb{R}^{2n}}$,
we obtain that the involution $\tau$ satisfying the final assumptions is
an orthogonal transformation and hence
preserves $B^{2n}(r)$. For such $\tau$
we have showed above Corollary~\ref{cor:ellipsoid} that there exists a linear orthogonal and symplectic isomorphism $\Psi:(\mathbb{R}^{2n},\omega_0)\to (\mathbb{R}^{2n},\omega_0)$
satisfies $\Psi \tau=\tau_0\Psi$. It follows from (\ref{e:ellCap}) that
$$
 c_{\rm EHZ,\tau}(B^{2n}(r))=c_{\rm EHZ,\tau_0}(B^{2n}(r))=\pi r^2\quad\forall r>0.
$$
%For $c_{\rm HZ,\tau}(B^{2n}(r),\omega_0)$ this was proved in \cite[Theorem~2.4(C)]{LiuW12}.\\
The equality $c_{\rm HZ,\tau}(B^{2n}(r),\omega_0)=\pi r^2$ was proved in \cite[Theorem~2.4(C)]{LiuW12}.\\

% For a linear anti-symplectic involution $\tau$
%on $(\mathbb{R}^{2n},\omega_0)$,  (\ref{e:ellipsoid.3}) implies that
%\begin{eqnarray}\label{e:ball.1}
%\tau(B^{2n}(1))=B^{2n}(1)\quad\Longleftrightarrow\quad \tau^T\tau=I_{2n}\quad \Longleftrightarrow\quad \tau^T=\tau.
%\end{eqnarray}
As a generalization of (\ref{e:ball}) we have the next corollary.

%preserves $B^{2n}(1)$ if and only if it is also
%orthogonal (and so symmetric).

\begin{corollary}\label{cor:ellipsoid+}
Let the ellipsoid $E(q)=\{z\in\mathbb{R}^{2n}\,|\, q(z)<1\}$  be as in Corollary~\ref{cor:ellipsoid}.
For any linear anti-symplectic involution $\tau$ on $(\mathbb{R}^{2n},\omega_0)$,
%which is symmetric with respect to $\langle\cdot,\cdot\rangle_{\mathbb{R}^{2n}}$,
 if $\tau$ preserves $E(q)$
   then
\begin{equation}\label{e:ellipsoid}
{c}_{\rm EHZ,\tau}(E(q))=c_{\rm EHZ}(E(q)),
\end{equation}
where $c_{\rm EHZ}(E(q))$ denotes the common value of the Ekeland-Hofer capacity and
the Hofer-Zehnder capacity of $E(q)$.
\end{corollary}

%For $E(q)=B^{2n}(1)$ this was proved in \cite[Theorem~2.4(C)]{LiuW12}. Notice that a linear anti-symplectic involution $\tau$ on $(\mathbb{R}^{2n},\omega_0)$ preserves $B^{2n}(1)$ if and only if $\tau'=\tau$.

\begin{proof}
By the arguments above Corollary~\ref{cor:ellipsoid}
 there exists a symplectic matrix $\Psi$ such that $\Psi \tau=\tau_0\Psi$.  Then
 Proposition~\ref{MonComf}(ii) implies
  \begin{equation}\label{e:1.11.1}
  c_{\rm HZ,\tau}(E(q),\omega_0)=c_{\rm HZ,\tau_0}(\Psi E(q),\omega_0),
  \end{equation}
    where $\Psi(E(q))=\{z\in\mathbb{R}^{2n}\,|\,\langle(\Psi^{-1})^T S\Psi^{-1} z,z\rangle_{\mathbb{R}^{2n}}\le 1\}$ is also an ellipsoid.  Notice that $\tau_0(\Psi(E(q)))=\Psi\tau(E(q))=\Psi (E(q))$, i.e. $\tau_0$ preserves $\Psi(E(q))$. As claimed below (\ref{e:ellipsoid.3}) we obtain that
   $(\Psi^{-1})^T S\Psi^{-1} $ commutes with $\tau_0$.
   Using Proposition~\ref{prop:xu}, we can find a symplectic matrix ${\Phi}$ commuting with $\tau_0$ and satisfying
   $$
   \Phi^T(\Psi^{-1})^T S\Psi^{-1}\Phi={\rm diag}(r_1^2,\cdots,r_n^2, r_1^2,\cdots, r_n^2),
   $$
   where $0<r_1\le\cdots\le r_n$. That is,
   $$
   {\Phi}^{-1}\Psi(E(q))=\left\{(x,y)\in\mathbb{R}^{2n}\,\Bigg|\, \sum^n_{j=1}
(x_j^2+y_j^2)/{r_j^2}<1\right\}=:E(r_1,\cdots,r_n).
   $$
   Then Proposition~\ref{MonComf}(ii) and Theorem~\ref{th:convex} lead to
   \begin{eqnarray}\label{e:1.11.2}
c_{\rm HZ,\tau_0}(\Psi(E(q)),\omega_0)&=&c_{\rm HZ,\tau_0}({\Phi}^{-1}\Psi(E(q)),\omega_0)\nonumber\\
&=&c_{\rm HZ,\tau_0}((E(r_1,\cdots,r_n)),\omega_0)=\pi r_1^2.
   \end{eqnarray}
   Moreover,  because of  symplectic invariance of  the Hofer-Zehnder capacity, we have
   \begin{eqnarray}\label{e:1.11.3}
   c_{\rm HZ}(E(q),\omega_0)&=&c_{\rm HZ}({\Phi}^{-1}\Psi E(q),\omega_0)\nonumber\\
   &=&c_{\rm HZ}(E(r_1,\cdots, r_n),\omega_0)=\pi r_1^2
   \end{eqnarray}
   (\cite{HoZe90}).
   Hence  (\ref{e:1.11.1}), (\ref{e:1.11.2}) and (\ref{e:1.11.3}) yield
   $$
   c_{\rm HZ,\tau}(E(q),\omega_0)=c_{\rm HZ}(E(q),\omega_0).
   $$
(\ref{e:ellipsoid}) follows from this and (\ref{e:fixpt.2}).
\end{proof}

As stated below (\ref{e:1.can-inv}) %Without special statements
  sometimes we still use $\tau_0$ to denote
 the canonical  antisymplectic involutions on the standard
 symplectic spaces  $\mathbb{R}^{2l}$ of different dimensions.
The following theorem will be proved in Section~\ref{sec:EH.3.2}.

\begin{thm}\label{th:EHproduct}
 For $\tau_0$-invariant compact convex subsets $D_i\subset\mathbb{R}^{2n_i}$, $i=1,\cdots,k$,
 it holds that
 \begin{equation}\label{e:product1}
 c_{\rm EH,\tau_0}(D_1\times\cdots\times D_k)=\min_ic_{\rm EH,\tau_0}(D_i)=c_{\rm EH,\tau_0}(\partial D_1\times\cdots\times \partial D_k).
 \end{equation}
 \end{thm}

\begin{corollary}\label{cor:cylinder}
Let $\tau_1$ and $\tau_2$ be linear anti-symplectic involutions on $\mathbb{R}^2$
and $\mathbb{R}^{2n-2}$, respectively. If  $\tau_1$ is also symmetric with respect to $\langle\cdot,\cdot\rangle_{\mathbb{R}^{2}}$ then
$$
c_{\rm HZ,\tau_1\times \tau_2}(Z^{2n}(1),\omega_0)=c_{\rm EH,\tau_1\times \tau_2}(Z^{2n}(1))=\pi.
$$
\end{corollary}

With the similar arguments to those of \cite{HoZe90, HoZe94} it was proved in \cite{LiuW12} that
$c_{\rm HZ,\tau_0}(Z^{2n}(1),\omega_0)=\pi$.

\begin{proof}[Proof of Corollary~\ref{cor:cylinder}]
By the inner regularity of the symmetric Hofer-Zehnder capacity and the definition of the symmetric
 Ekeland-Hofer capacity of unbounded sets, we deduce
%\begin{eqnarray*}
%c_{\rm HZ,\tau_1\times \tau_2}(Z^{2n}(1),\omega_0)&=&c_{\rm EH,\tau_1\times \tau_2}(Z^{2n}(1))\\
%&=&\sup\{c_{\rm EHZ,\tau_1\times\tau_2}(B^{2}(1)\times B^{2n-2}(R))\,|\,R>0\}.
%\end{eqnarray*}
\begin{eqnarray*}
c_{\rm HZ,\tau_1\times \tau_2}(Z^{2n}(1),\omega_0)&=&c_{\rm EH,\tau_1\times \tau_2}(Z^{2n}(1))\\
&=&\sup\{c_{\rm EHZ,\tau_1\times\tau_2}(B^{2}(1)\times C^{2n-2}_{\tau_2}(R))\,|\,R>0\},
\end{eqnarray*}
where $C^{2n-2}_{\tau_2}(R):=B^{2n-2}(R)+\tau_2B^{2n-2}(R)$,
which is $\tau_2$-invariant and satisfies $C^{2n-2}_{\tau_2}(R)=RC^{2n-2}_{\tau_2}(1)$
for any $R>0$.
By the assumptions we have an orthogonal symplectic matrix
 $\Psi_1\in{\rm Sp(2,\mathbb{R})}$ and a symplectic matrix $\Psi_2\in{\rm Sp}(2n-2,\mathbb{R})$
 such that $\Psi_1\tau_1=\tau_0\Psi_1$ and $\Psi_2\tau_2=\tau_0\Psi_2$. Then
%\begin{eqnarray*}
%&&c_{\rm EHZ,\tau_1\times\tau_2}(B^{2}(1)\times B^{2n-2}(R))\\
%&=&c_{\rm EHZ,\tau_0}(\Psi_1 B^{2}(1)\times \Psi_2B^{2n-2}(R))\\
%&=&\min\{c_{\rm EHZ,\tau_0}(B^{2}(1)),c_{\rm EHZ,\tau_0}(\Psi_2B^{2n-2}(R))\}\\
%&=&\min\{c_{\rm EHZ,\tau_0}(B^{2}(1)),R^2c_{\rm EHZ,\tau_0}(\Psi_2B^{2n-2}(1))\}\\
%&=&c_{\rm EHZ,\tau_0}(B^{2}(1))=\pi
%\end{eqnarray*}
%for sufficiently large $R>0$ (since $c_{\rm EHZ,\tau_0}(\Psi_2B^{2n-2}(1))>0$),
\begin{eqnarray*}
&&c_{\rm EHZ,\tau_1\times\tau_2}(B^{2}(1)\times C^{2n-2}_{\tau_2}(R))\\
&=&c_{\rm EHZ,\tau_0}(\Psi_1 B^{2}(1)\times \Psi_2C^{2n-2}_{\tau_2}(R))\\
&=&\min\{c_{\rm EHZ,\tau_0}(B^{2}(1)),c_{\rm EHZ,\tau_0}(\Psi_2C^{2n-2}_{\tau_2}(R))\}\\
&=&\min\{c_{\rm EHZ,\tau_0}(B^{2}(1)),R^2c_{\rm EHZ,\tau_0}(\Psi_2C^{2n-2}_{\tau_2}(1))\}\\
&=&c_{\rm EHZ,\tau_0}(B^{2}(1))=\pi
\end{eqnarray*}
for sufficiently large $R>0$ (since $c_{\rm EHZ,\tau_0}(\Psi_2C^{2n-2}_{\tau_2}(1))>0$),
where the second equality follows from Theorem~\ref{th:EHproduct} and the third equality is a consequence of conformality of capacity. Finally, by the monotonicity of capacities we get
\begin{eqnarray*}
c_{\rm HZ,\tau_1\times \tau_2}(Z^{2n}(1),\omega_0)&=&c_{\rm EH,\tau_1\times \tau_2}(Z^{2n}(1))\\
&=&\lim_{R\to\infty}c_{\rm EHZ,\tau_1\times\tau_2}(B^{2}(1)\times C^{2n-2}_{\tau_2}(R))=\pi.
\end{eqnarray*}
\end{proof}
% \hfill$\Box$\vspace{2mm}

\begin{corollary}\label{cor:CrokeW1}
Let $\tau$ be a linear anti-symplectic involution on $\mathbb{R}^{2n}$ which is
symmetric with respect to the standard inner product $\langle\cdot,\cdot\rangle_{\mathbb{R}^{2n}}$,
 and let $D\subset\mathbb{R}^{2n}$ be a $\tau$-invariant convex bounded domain. Suppose that $p$ is a fixed point of $\tau$.
 Then the following holds.
%\begin{description}
\begin{enumerate}
\item[\bf (i)] If $D$ contains a ball $B^{2n}(p,r)$,  then for any  generalized
$\tau$-brake closed characteristic $x$ on $\partial D$ with positive action
  it holds that
\begin{equation}\label{e:croke1}
A(x)\ge \pi r^2.
\end{equation}
\item[\bf (ii)] If $D\subset B^{2n}(p,R)$,  there exists a
generalized  $\tau$-brake closed characteristic $x^\star$ on $\partial D$ such that
 \begin{equation}\label{e:croke2}
 0<A(x^\star)\le \pi R^2.
\end{equation}
\end{enumerate}
%\end{description}
 \end{corollary}
 \begin{proof}
   The conclusions follow easily from Theorem~\ref{th:convex} (or Theorem~\ref{th:EHconvex}) and Corollary~\ref{cor:ellipsoid}.
\end{proof}

\begin{thm}\label{th:EHcontact}
Let $\tau$ be a linear anti-symplectic involution  on $(\mathbb{R}^{2n},\omega_0)$,
and let $\mathcal{S}\subset\mathbb{R}^{2n}$ be a hypersurface of restricted contact type
that admits a globally defined Liouville vector field $X$ transversal to it such that
\begin{equation}\label{e:EHcontact}
 X(\tau (z))=\tau X(z),\,\forall z\in\mathbb{R}^{2n}.
 \end{equation}
Suppose that the bounded component  $B$  of $\mathbb{R}^{2n}\setminus\mathcal{S}$
is $\tau$-invariant and intersects with $L={\rm Fix}(\tau)$.  Then  $c_{\rm EH,\tau}(B)=c_{\rm EH,\tau}(\mathcal{S})$ belongs to
 \begin{equation}\label{e:EHcontact+}
  \Sigma^{\tau}_{\mathcal{S}}:=\{
   A(x)>0\,|\,x \;\hbox{is a $\tau$-brake closed characteristic on $\mathcal{ S}$}\}.
  \end{equation}
\end{thm}

This theorem  will be proved in Section~\ref{sec:EH.4}.
%\begin{remark}
%  We can replace $\tau_0$ by any linear anti-symplectic involution $\tau$ on $(\mathbb{R}^{2n},\omega_0)$.
%\end{remark}

\begin{remark}
%\begin{description}
%\item [(i)]
{\rm $\tau B=B$ and $B\cap L\ne\emptyset$ imply that $\tau\mathcal{S}=\mathcal{S}$ and $\mathcal{S}\cap L\ne\emptyset$.}
%\item[(ii)]
%  If $B$ is  simply connected, then $\tau B=B$ implies that $B\cap L\ne\emptyset$.
%  \end{description}}
\end{remark}

 Bates \cite{Ba98} extended \cite[Proposition~6]{EH89} to certain domains
 whose boundaries are not of restricted contact type. The corresponding  generalizations of
Theorem~\ref{th:EHcontact} are also possible.

\subsection{Evgeni Neduv theorem for symmetric capacities}\label{sec:1.Neduv}

Let $\tau$ be a linear anti-symplectic involution  on $(\mathbb{R}^{2n},\omega_0)$,
and let $\mathscr{H}\in C^2(\mathbb{R}^{2n},\mathbb{R}_+)$ be a $\tau$-invariant
proper and strictly convex Hamiltonian
such that $\mathscr{H}(0)=0$ and $\mathscr{H}''>0$ (and so
$\mathscr{H}\ge 0$ by the Taylor's formula).
If $e_0> 0$ is a regular value of $\mathscr{H}$ with
$\mathscr{H}^{-1}(e_0)\ne\emptyset$, then for each number $e$ near $e_0$
the set $D(e):=\{\mathscr{H}<e\}$ is a $\tau$-invariant strictly convex bounded domain
 in $\mathbb{R}^{2n}$ with $0\in D(e)$ and
with $C^2$-boundary $\mathcal{S}(e)=\mathscr{H}^{-1}(e)$.
For any $e$ near $e_0$ let $\mathscr{C}_\tau(e):=c_{\rm EHZ,\tau}(D(e))$.
As remarked below (\ref{e:fixpt.2}) all $c_{\rm EHZ,\tau}$-carriers for $D(e)$ form a compact subset
in $C^1(\mathbb{R}/(\mathscr{C}_\tau(e)\mathbb{Z}), \mathcal{S}(e))$. Hence
\begin{equation}\label{e:convexDiff}
\mathscr{I}_\tau(e):=\left\{T_x=2\int^{\mathscr{C}_\tau(e)}_0\frac{dt}{\langle\nabla\mathscr{H}(x(t)), x(t)\rangle_{\mathbb{R}^{2n}}}\,\Big|\,
 \hbox{$x$ is a $c_{\rm EHZ,\tau}$-carrier for $D(e)$}\right\}
\end{equation}
is a compact subset in $\mathbb{R}$. Denote by $T^{\max}_\tau(e)$ and $T^{\min}_\tau(e)$
 the largest and smallest numbers in $\mathscr{I}_\tau(e)$.
By the reparameterization every $c_{\rm EHZ,\tau}$-carrier $x$ for $D(e)$ yields  a $T_x$-periodic $\tau$-brake orbit of
\begin{equation}\label{e:convexDiff+}
-J_0\dot{y}(t)=\nabla\mathscr{H}(y(t)),
\end{equation}
which sits in $\mathcal{S}(e)=\mathscr{H}^{-1}(e)$ and has $T_{x}$ as the minimal period.
Slightly modifying the proof of  \cite[Theorem~4.4]{Ned01} in the setting of
 Section~\ref{sec:convex} we can obtain the following corresponding result with
 \cite[Theorem~4.4]{Ned01}.

\begin{thm}\label{th:convexDiff}
Under the above assumptions  $\mathscr{C}_\tau(e)$ has the left and right derivatives
at $e_0$, $\mathscr{C}'_{\tau-}(e_0)$ and $\mathscr{C}'_{\tau+}(e_0)$, and they satisfy
\begin{eqnarray*}
&&\mathscr{C}'_{\tau-}(e_0)=\lim_{\epsilon\to0-}T^{\max}_\tau(e_0+\epsilon)=T^{\max}_\tau(e_0)\quad\hbox{and}\\
&&\mathscr{C}'_{\tau+}(e_0)=\lim_{\epsilon\to0+}T^{\min}_\tau(e_0+\epsilon)=T^{\min}_\tau(e_0).
\end{eqnarray*}
Moreover, if  $[a,b]\subset (0,\sup\mathscr{H})$ is an regular interval of $\mathscr{H}$ such that
$\mathscr{C}'_{\tau+}(a)<\mathscr{C}'_{\tau-}(b)$, then for any  $r\in (\mathscr{C}'_{\tau+}(a),
\mathscr{C}'_{\tau-}(b))$ there exist $e'\in (a,b)$ such that $\mathscr{C}_\tau(e)$
is differentiable at $e'$ and $\mathscr{C}'_{\tau-}(e')=\mathscr{C}'_{\tau+}(e')=r=T^{\max}_\tau(e')=T^{\min}_\tau(e')$.
\end{thm}

As a monotone function on an regular interval $[a,b]$ of $\mathscr{H}$ as above,
  we have $\mathscr{C}'_{\tau-}(e)=\mathscr{C}'_{\tau+}(e)$
   for almost all values of $e\in [a, b]$.
Actually, the first claim of Theorem~\ref{th:convexDiff} and a recent result \cite[Corollary~6.4]{Bl14}
imply that both $T^{\max}_\tau$ and $T^{\min}_\tau$ have only at most countable discontinuous
points and are also Riemann integrable on $[a,b]$.

By Theorem~\ref{th:convexDiff}, for any regular interval $[a,b]\subset (0,\sup \mathscr{H})$
 of $\mathscr{H}$ with $\mathscr{C}'_{\tau+}(a)\le\mathscr{C}'_{\tau-}(b)$,
if $T\in [\mathscr{C}'_{\tau+}(a), \mathscr{C}'_{\tau-}(b)]$ then
(\ref{e:convexDiff+}) has a $T$-periodic $\tau$-brake orbit with
$T$ as the minimal period. For example, we have the following.

 \begin{corollary}\label{cor:convexDiff}
 Let $\tau$ be a linear anti-symplectic involution  on $(\mathbb{R}^{2n},\omega_0)$
 which is symmetric with respect to $\langle\cdot,\cdot\rangle_{\mathbb{R}^{2n}}$.
Suppose that a $\tau$-invariant proper and strictly convex Hamiltonian $\mathscr{H}\in C^2(\mathbb{R}^{2n},\mathbb{R}_+)$
satisfies the following conditions:
%\begin{description}
\begin{enumerate}
\item[\bf (i)] $\mathscr{H}(0)=0$, $\mathscr{H}''>0$ and every $e>0$ is a regular value of $\mathscr{H}$,
\item[\bf (ii)] there exist two positive definite symmetric matrixes $S_0, S_\infty\in\mathbb{R}^{2n\times 2n}$
commuting with $\tau$ such that $\mathscr{H}(x)$ is equal to  $q(x):=\frac{1}{2}\langle S_0x, x\rangle_{\mathbb{R}^{2n}}$ (resp.
    $Q(x):=\frac{1}{2}\langle S_\infty x, x\rangle_{\mathbb{R}^{2n}}$) for $|x|$ small (resp. large) enough.
%\end{description}
\end{enumerate}
Then when $c_{\rm EHZ}(E(q))\le c_{\rm EHZ}(E(Q))$,
for every $T\in [c_{\rm EHZ}(E(q)), c_{\rm EHZ}(E(Q))]$ the corresponding system
(\ref{e:convexDiff+}) has a $T$-periodic $\tau$-brake orbit with
$T$ as the minimal period.
\end{corollary}

Indeed, if $e>0$ is small (resp. large) enough then
$D(e)$ is equal to $D_q(e):=\{q<e\}=\sqrt{e}E(q)$ (resp. $D_Q(e):=\{Q< e\}=\sqrt{e}E(Q)$)
and so Corollary~\ref{cor:ellipsoid} implies
\begin{eqnarray*}
&&\hbox{$c_{\rm EHZ,\tau}(D(e))=c_{\rm EHZ,\tau}(D_q(e))=ec_{\rm EHZ}(E(q))$}\\
&&\hbox{(resp. $c_{\rm EHZ,\tau}(D(e))=c_{\rm EHZ,\tau}(D_Q(e))=ec_{\rm EHZ}(E(Q))$)}.
\end{eqnarray*}
 It follows that $\mathscr{C}'_\tau(a)=c_{\rm EHZ}(E(q))$ for $a>0$ small enough
and  $\mathscr{C}'_\tau(b)=c_{\rm EHZ}(E(Q))$ for $b>0$ large enough.  The conclusions of Corollary~\ref{cor:convexDiff}  follows from
Theorem~\ref{th:convexDiff} immediately.

\subsection{A Brunn-Minkowski inequality for $c_{\rm EHZ,\tau}$-capacity of convex domains}\label{sec:1.AAO}

 Recall that the associated support function of a {\bf convex body} $K\subset \mathbb{R}^{n}$
  (i.e., a compact convex subset with  interior points)  is defined by
 \begin{equation}\label{e:supp.1}
h_K(w)=\sup\{\langle x,w\rangle_{\mathbb{R}^n}\,|\,x\in K\},\quad\forall w\in\mathbb{R}^{n}.
\end{equation}
If $K$ contains $0$ in its interior,
 $K^{\circ}=\{x\in\mathbb{R}^{n}\,|\, \langle x,y\rangle_{\mathbb{R}^n}\le 1\;\forall y\in K\}$
is the polar body of $K$, and $j_{K^{\circ}}$ is
the gauge function of $K^{\circ}$ defined by (\ref{e:gauge}),
 then $h_K$ and $(h_K)^2$ are equal to
 $j_{K^{\circ}}$ (cf. \cite[Theorem~1.7.6]{Sch93}) and
the four times of the Legendre transform $H_K^{\ast}$ of $H_K:=(j_K)^2$
(see \cite[\S9.1]{JinLu19}), respectively. We say a convex body $K$
 to be {\bf centrally symmetric} if $-K=K$. The mean-width of such a convex body $K\subset \mathbb{R}^{n}$
 is defined by
  \begin{equation}\label{e:mean-width}
M^\ast(K)=\int_{S^{n-1}}h_K(x)d\sigma_n(x)=
\int_{O(n)}h_K(Ax)d\mu_n(A)\quad\forall x\in S^{n-1},
 \end{equation}
where $\sigma_n$ is the normalized rotation invariant measure on $S^{n-1}$
and $\mu_n$ is the Haar measure (cf. \cite{AAO08} and \cite[(17)]{BoLiMi88}).

For two convex bodies $D, K\subset \mathbb{R}^{n}$ containing  $0$ in their interiors
 and a real number $p\ge 1$,  there exists a unique convex body
 $D+_pK\subset\mathbb{R}^{n}$ with support function
   \begin{equation}\label{e:supp.2}
\mathbb{R}^{n}\ni w\mapsto h_{D+_pK}(w)=(h^p_{D}(w)+h_{K}^p(w))^{\frac{1}{p}}
 \end{equation}
  (cf. \cite[Theorem~1.7.1]{Sch93}),  $D+_pK$ is called
   the {\bf $p$-sum} of $D$ and $K$ by Firey (cf. \cite[(6.8.2)]{Sch93}).
Recently, Artstein-Avidan and Ostrover \cite{AAO08}  established the following
Brunn-Minkowski-type inequality for the Ekeland-Hofer-Zehnder capacity $c_{\rm EHZ}$.

\begin{thm}[\hbox{\cite{AAO08, AAO14}}]\label{th:Brun.1}
   Let $D, K\subset \mathbb{R}^{2n}$ be two convex bodies  containing  $0$ in their interiors. Then for any real $p\ge 1$ it holds that
   $$
   \left(c_{\rm EHZ}(D+_pK)\right)^{\frac{p}{2}}\ge \left(c_{\rm EHZ}(D)\right)^{\frac{p}{2}}+ \left(c_{\rm EHZ}(K)\right)^{\frac{p}{2}}.
   $$
   Moreover, the equality holds if and only if there exist
    $c_{\rm EHZ}$-carriers for $D$ and $K$,
      $\gamma_D:[0,T]\rightarrow \partial D$ and  $\gamma_K:[0,T]\rightarrow \partial K$, such that
      they coincide up to translation and dilation, i.e.,
   $\gamma_D=\alpha\gamma_K+ {\bf b}$ for some
   $\alpha\in\mathbb{R}$ and ${\bf b}\in\mathbb{R}^{2n}$.
  \end{thm}

Let $\tau$ be a linear anti-symplectic involution on $\mathbb{R}^{2n}$.
% and let $\tau^T$ be the adjoint of $\tau$  with respect to the standard inner product $\langle\cdot,\cdot\rangle_{\mathbb{R}^{2n}}$.
Then $h_K$ is $\tau$-invariant if and only if $K$ is $\tau^T$-invariant.
Notice that $D+_pK$ is $\tau$-invariant if both $D$ and $K$ are both $\tau$-invariant. As an analogy of Theorem~\ref{th:Brun.1}, we have

\begin{thm}\label{th:Brun.2}
Let $\tau$ be a linear anti-symplectic involution on $\mathbb{R}^{2n}$,
and let $D, K\subset \mathbb{R}^{2n}$ be two $\tau$-invariant convex bodies  containing  $0$ in their interiors. Then for any real $p\ge 1$ it holds that
   \begin{equation}\label{e:BrunA}
   \left(c_{\rm EHZ,\tau}(D+_pK)\right)^{\frac{p}{2}}\ge \left(c_{\rm EHZ,\tau}(D)\right)^{\frac{p}{2}}+ \left(c_{\rm EHZ,\tau}(K)\right)^{\frac{p}{2}}.
   \end{equation}
   Moreover, the equality holds if  there exist
    $c_{\rm EHZ,\tau}$-carriers for $D$ and $K$,
      $\gamma_D:[0,T]\rightarrow \partial D$ and  $\gamma_K:[0,T]\rightarrow \partial K$, such that
      they coincide up to dilation and translation by elements in $L={\rm Fix}(\tau)$, i.e.,
   $\gamma_D=\alpha\gamma_K+ {\bf b}$ for some
   $\alpha\in\mathbb{R}$ and ${\bf b}\in L$; and in the case $p>1$
   the latter condition is also necessary for the existence of equality in (\ref{e:BrunA}).
  \end{thm}

The proof of this theorem will be given in Section~\ref{sec:Brunn.1}.

\begin{corollary}\label{cor:Brun.1}
Let $\tau$ be a linear anti-symplectic involution on $\mathbb{R}^{2n}$,
and let $D, K\subset \mathbb{R}^{2n}$ be two $\tau$-invariant convex bodies.
%  containing
%a fixed point $q$ of $\tau$ as their common interior point.
 Then
\begin{equation}\label{e:BrunA+}
   \left(c_{\rm EHZ,\tau}(D+K)\right)^{\frac{1}{2}}\ge \left(c_{\rm EHZ,\tau}(D)\right)^{\frac{1}{2}}+ \left(c_{\rm EHZ,\tau}(K)\right)^{\frac{1}{2}},
   \end{equation}
   and the equality holds if  there exist
    $c_{\rm EHZ,\tau}$-carriers for $D$ and $K$ which coincide up to  dilation and translation by elements in
    $L={\rm Fix}(\tau)$.
\end{corollary}

Since $D, K\subset \mathbb{R}^{2n}$ are $\tau$-invariant, for each
$x\in{\rm Int}(D)$ (resp. ${\rm Int}(K)$),
$(x+\tau x)/2$ is a fixed point of $\tau$ contained in ${\rm Int}(D)$ (resp. ${\rm Int}(K)$).
Take $p\in{\rm Fix}(\tau)\cap{\rm Int}(D)$ and
$q\in{\rm Fix}(\tau)\cap{\rm Int}(K)$. Then
(\ref{e:BrunA}) implies
\begin{eqnarray*}
 \left(c_{\rm EHZ,\tau}(D+K-p-q)\right)^{\frac{1}{2}}&=&  \left(c_{\rm EHZ,\tau}((D-p)+(K-q))\right)^{\frac{1}{2}}\\
 &\ge& \left(c_{\rm EHZ,\tau}(D-p)\right)^{\frac{1}{2}}+ \left(c_{\rm EHZ,\tau}(K-q)\right)^{\frac{1}{2}},
 \end{eqnarray*}
Since $p+q$ is a fixed point of $\tau$, it is clear that
$c_{\rm EHZ,\tau}(D+K-p-q)=c_{\rm EHZ,\tau}(D+K)$, $c_{\rm EHZ,\tau}(D-p)=
c_{\rm EHZ,\tau}(D)$
and $c_{\rm EHZ,\tau}(K-q)=c_{\rm EHZ,\tau}(K)$ by
Proposition~\ref{MonComf}(ii).
%the arguments
%at the beginning of Section~\ref{sec:convex}.
Other claims easily follows from the arguments therein.

As in \cite{AAO08, AAO14}, from Corollary~\ref{cor:Brun.1} we can derive:

\begin{corollary}\label{cor:Brun.2}
 Let $D$, $K$ and $\tau$ be as in Corollary~\ref{cor:Brun.1}.
% \begin{description}
 \begin{enumerate}
\item[\bf (i)] For any $x,y\in{\rm Fix}(\tau)$ and $0\le \lambda\le 1$ it holds that
\begin{eqnarray}\label{e:BrunB}
&&\lambda\left(c_{\rm EHZ,\tau}(D\cap (x+K))\right)^{1/2}+
(1-\lambda)\left(c_{\rm EHZ,\tau}(D\cap (y+K))\right)^{1/2}\nonumber\\
&\le&\left(c_{\rm EHZ,\tau}(D\cap(\lambda x+(1-\lambda)y+K))\right)^{1/2}.
\end{eqnarray}
  In particular, $c_{\rm EHZ,\tau}(D\cap(x+K))\le c_{\rm EHZ,\tau}(D\cap K)$ provided that $D$ and $K$ are
   centrally symmetric.
\item[\bf (ii)] The limit
\begin{equation}\label{e:BrunC}
\lim_{\varepsilon\to 0+}\frac{c_{\rm EHZ,\tau}(D+\varepsilon K)-c_{\rm EHZ,\tau}(D)}{\varepsilon}
\end{equation}
exists, denoted by $d_{K,\tau}(D)$, and it holds that
\begin{eqnarray}\label{e:BrunD}
2(c_{\rm EHZ,\tau}(D))^{1/2}(c_{\rm EHZ,\tau}(K))^{1/2}\le d_{K,\tau}(D)\le
\inf_{z_D}\int_0^1h_K(-J\dot{z}_D(t)),
\end{eqnarray}
where $z_D:[0,1]\to\partial D$ takes over all  $c_{\rm EHZ,\tau}$-carriers  for $D$.
%\end{description}
\end{enumerate}
\end{corollary}

For the sake of clarity we shall give the proof of this corollary in Section~\ref{sec:Brunn.2}.

In view of \cite{AAO08, AAO14} we call ${\rm length}_{JK^\circ}(z_D)=\int_0^1j_{JK^\circ}(\dot{z}_D(t))$
 the length of $z_D$ with respect to the convex body $JK^\circ$.
In the case $0\in{\rm int}(K)$, since $H^\ast_K(-Jv)=(j_{JK^\circ}(v))^2/4$,  (\ref{e:BrunD}) implies
$$
d_{K,\tau}(D)\le 2\inf_{z_D}\int_0^1(H_K^{\ast}(-J\dot{z}_D(t)))^{\frac{1}{2}}=
\inf_{z_D}\int_0^1j_{JK^\circ}(\dot{z}_D(t))
$$
and hence
\begin{eqnarray}\label{e:BrunE}
4c_{\rm EHZ,\tau}(D,\omega_0)c_{\rm EHZ,\tau}(K,\omega_0)\le
\inf_{z_D}({\rm length}_{JK^\circ}(z_D))^2.
\end{eqnarray}
In particular, if the involution $\tau$  is symmetric with respect to $\langle\cdot,\cdot\rangle_{\mathbb{R}^{2n}}$,
taking $K=B^{2n}(1)$ we derive from (\ref{e:BrunE}) and (\ref{e:ball}) that
\begin{equation}\label{e:BrunF}
4\pi c_{\rm EHZ,\tau}(D,\omega_0)\le
\left(\int^T_0|\dot{\gamma}(t)|dt\right)^2.
   \end{equation}
where $\gamma:\mathbb{R}/T\mathbb{Z}\to \partial D$ is a $c_{\rm EHZ,\tau}$-carrier for $D$.
A similar result to \cite[(1.4.6)]{AAO08} can also be obtained.

Let $D\subset\mathbb{R}^{2n}$ be a $\tau_0$-invariant  convex body.
Lemma~\ref{lem:alb} yields a Lie group homomorphism
\begin{equation}\label{e:BrunF.1}
O(n)\to {\rm Sp}(2n,\mathbb{R})\cap O(2n)\equiv U(n),\;A\mapsto\Psi_A=\left(
  \begin{array}{cc}
    A & 0 \\
    0 & A \\
  \end{array}
\right).
 \end{equation}
 Since $\Psi_A\tau_0=\tau_0\Psi_A$, $c_{\rm EHZ,\tau_0}(D)=c_{\rm EHZ,\tau_0}(\Psi_AD)\;\forall A\in O(n)$.
 Thus for $A_i\in O(n)$, $i=1,\cdots,m$, it follows from (\ref{e:BrunA+}) that
 \begin{eqnarray*}
c_{\rm EHZ,\tau_0}(D)=\left(\sum^m_{i=1}\left(c_{\rm EHZ,\tau_0}(\frac{1}{m}\Psi_{A_i}D)\right)^{\frac{1}{2}}\right)^{2}
\le c_{\rm EHZ,\tau_0}\left(\frac{1}{m}\sum^m_{i=1}\Psi_{A_i}D\right)
\end{eqnarray*}
and so
  \begin{eqnarray}\label{e:BrunF+}
c_{\rm EHZ,\tau_0}(D)\le c_{\rm EHZ,\tau_0}\left(\int_{O(n)}\Psi_{A}D d\mu_n(A)\right),
\end{eqnarray}
where the integral is with respect to the  Haar measure $\mu_n$ on $O(n)$.
 By properties of the support function (cf. \S1.7 and Theorem~1.8.11 in \cite{Sch93}) it is easy to show that
 the integral in (\ref{e:BrunF+}) as a convex body in $\mathbb{R}^{2n}$ has the support function
  $$
\mathbb{R}^{2n}\ni v\mapsto F(v):=\int_{O(n)}h_{\Psi_{A}D}(v) d\mu=\int_{O(n)}h_{D}(\Psi_{A^t}v) d\mu_n(A),
$$
where $h_{\Psi_{A}D}=h_{D}\circ\Psi_{A^t}$ is the support function of $\Psi_{A}D$.
As claimed above \cite[Th.1.7.1]{Sch93}, this support function $F$ is a sublinear convex function.
By  \cite[Th.1.7.1]{Sch93} and the second proof of it there  is a unique convex
body in $\mathbb{R}^{2n}$ with support function $F$, and
this convex body is given by $\{u\in\mathbb{R}^{2n}\,|\, \langle u, v\rangle_{\mathbb{R}^n}\le F(v)\;\forall v\in \mathbb{R}^{2n}\}$.
Hence we obtain
%By the proof of \cite[Th.1.7.1]{Sch93} we see
 \begin{eqnarray}\label{e:BrunG}
\int_{O(n)}\Psi_{A}D d\mu=\left\{u\in\mathbb{R}^{2n}\,\Bigm|\, \langle u, v\rangle_{\mathbb{R}^n}\le
\int_{O(n)}h_{D}(\Psi_{A^t}v) d\mu_n(A)\;\forall v\in \mathbb{R}^{2n}\right\}.
\end{eqnarray}
Clearly, it is invariant under subgroup $\{\Psi_A\,|\,A\in O(n)\}\subset {\rm Sp}(2n,\mathbb{R})\cap O(2n)$.
(We guess that the convex body in (\ref{e:BrunG}) is the product $B^n(r_1)\times B^n(r_2)$ for some $r_1>0,r_2>0$.)
%Equation
Inequality (\ref{e:BrunF+}) is an analogue of \cite[Corollary~1.7]{AAO14}, which claimed
\begin{eqnarray}\label{e:BrunG+}
c_{\rm EHZ}(D)\le\pi(M^\ast(D))^2
\end{eqnarray}
for any centrally symmetric convex body $D\subset \mathbb{R}^{2n}$.

If the above $D\subset\mathbb{R}^{2n}$ has some special form we can improve (\ref{e:BrunF+}).
To this end let us state a beautiful result recently proved
 by  Artstein-Avidan,  Karasev, and Ostrover in \cite[Theorem~1.7]{AAKO14}:
  \begin{equation}\label{e:BrunH}
 c_{\rm HZ}(\Delta\times \Delta^\circ)=4
 \end{equation}
 for any centrally symmetric convex body $\Delta\subset\mathbb{R}^n_q$.
 Note that $\Delta=-\Delta$ implies $\Delta^\circ=-\Delta^\circ$.
 By \cite[Prop.4]{HoZe90} and Theorem~\ref{th:convex}
it is easily seen that $c_{\rm EHZ,\tau_0}(\Delta\times \Delta^\circ)\ge 4$
and $c_{\rm EHZ,\hat{\tau}_0}(\Delta\times \Delta^\circ)\ge 4$.
The proof of \cite[Theorem~1.7]{AAKO14} can still yield converse inequalities.
 Hence we obtain
\begin{equation}\label{e:BrunI}
 c_{\rm EHZ,\tau_0}(\Delta\times \Delta^\circ)=c_{\rm EHZ,\hat{\tau}_0}(\Delta\times \Delta^\circ)=4.
 \end{equation}

  Let $\Delta\subset\mathbb{R}^n_q$ and  $\Lambda\subset\mathbb{R}^n_p$  be two  convex bodies containing
  the origin in their interiors, and let $D=\Delta\times\Lambda$ be their Lagrangian product.
  Define
 %$r_\Delta:=r(\widehat{\Delta})$ with $\widehat{\Delta}:=\frac{1}{2}(\Delta+(-\Delta))$, $r_\Lambda:=r(\widehat{\Lambda})$ with $\widehat{\Lambda}:=\frac{1}{2}(\Lambda+(-\Lambda))$ and
\begin{eqnarray}\label{e:BrunJ}
&&r_\Delta=M^\ast(\widehat{\Delta})\quad\hbox{with}\quad\widehat{\Delta}:=\frac{1}{2}(\Delta+(-\Delta)),\\
&&r_\Lambda=M^\ast(\widehat{\Lambda})\quad\hbox{with}\quad\widehat{\Lambda}:=\frac{1}{2}(\Lambda+(-\Lambda)).
\label{e:BrunK}
\end{eqnarray}
Clearly, if $\Delta$ (resp. $\Lambda$) is centrally symmetric then $r_\Delta=M^\ast(\Delta)$
(resp. $r_\Lambda=M^\ast(\Lambda)$). Note that
$h_{\widehat{\Delta}}=(h_{\Delta}+ h_{(-\Delta)})/2$ and
$h_{\widehat{\Lambda}}=(h_{\Lambda}+ h_{(-\Lambda)})/2$
by Theorems~1.7.5 and 1.7.6 in \cite{Sch93}).
We have the following corollary, which will be proved in Section~\ref{sec:Brunn.2}.

\begin{corollary}\label{cor:Brun.4}
 Let $\Delta$ and $\Lambda$ be as above. Then
 \begin{eqnarray}\label{e:BrunL}
c_{\rm EHZ}(\Delta\times\Lambda)\le 4r_\Delta r_\Lambda.
\end{eqnarray}
Moreover, if $\Lambda$ (resp. $\Delta$) is centrally symmetric, then
 \begin{eqnarray}\label{e:BrunM}
c_{\rm EHZ,\tau_0}(\Delta\times\Lambda)\le 4r_\Delta r_\Lambda
\quad\hbox{\rm (resp. $c_{\rm EHZ, \hat{\tau}_0}(\Delta\times\Lambda)\le 4r_\Delta r_\Lambda$)}.
\end{eqnarray}
  \end{corollary}

\subsection{Applications to billiards dynamics}\label{sec:billard}

Recently,  Artstein-Avidan,  Karasev and Ostrover \cite{AAO14, AAKO14}
used the Ekeland-Hofer-Zehnder symplectic capacity to obtain  several
very interesting bounds and inequalities for the length of the shortest periodic billiard trajectory in a
smooth convex body in $\mathbb{R}^n$. Our above generalizations for
the Ekeland-Hofer-Zehnder symplectic capacity allow us to prove
some corresponding results for a special class of billiard trajectories in a
smooth convex body in $\mathbb{R}^n$.

 For a connected bounded open subset $\Omega\subset\mathbb{R}^n$ with boundary $\partial\Omega$
of class $C^2$, recall that a nonconstant, continuous, and piecewise $C^\infty$ path $\sigma:\mathbb{R}/(T\mathbb{Z})\to\overline{\Omega}$
is called a {\bf closed (or periodic) billiard trajectory} in  $\Omega$ if there exists a finite set
$\mathscr{B}_\sigma\subset \mathbb{R}/(T\mathbb{Z})$ with the following properties:
%\begin{description}
\begin{enumerate}
\item[\bf (B-i)] On $\mathbb{R}/(T\mathbb{Z})\setminus\mathscr{B}_\sigma$,
$\sigma$ is smooth and satisfies $\ddot{\sigma}=0$.
\item[\bf (B-ii)] For each $t\in \mathscr{B}_\sigma$, $\sigma(t)\in\partial\Omega$ and $\dot{\sigma}^\pm(t):=\lim_{\tau\to t\pm}\dot{\sigma}(\tau)$
fulfils the equation
\begin{equation}\label{e:bill1}
\dot{\sigma}^+(t)+\dot{\sigma}^-(t)\in T_{\sigma(t)}\partial\Omega,\quad
\dot{\sigma}^+(t)-\dot{\sigma}^-(t)\in (T_{\sigma(t)}\partial\Omega)^\bot\setminus\{0\}.
\end{equation}
\end{enumerate}
%\end{description}
(So $|\dot\sigma|$ is a nonzero constant on $\mathbb{R}/(T\mathbb{Z})\setminus\mathscr{B}_\sigma$.)
 $\mathscr{B}_\sigma$  is called the set of {\bf bounce times}.
 %If it is empty, the periodic billiard trajectory $\sigma$ becomes a closed geodesic.

We call the above  closed  billiard trajectory $\sigma:\mathbb{R}/(T\mathbb{Z})\to\overline{\Omega}$
  a {\bf brake  billiard trajectory} if it satisfies $\sigma(-t)=\sigma(t)$
 for all $t$. When $\Omega$ is also centrally symmetric, a closed  billiard trajectory $\sigma$
 in $\Omega$ is called a {\bf brake  billiard trajectory of second class} if it satisfies $\sigma(-t)=-\sigma(t)$
 for all $t$.

\begin{remark}\label{rm:brake}
{\rm  In \cite{Ir15} K. Irie called a nonconstant, continuous, and piecewise $C^\infty$ map
$\sigma:[0, T]\to\Omega$  a {\bf brake billiard trajectory} if it satisfies the following conditions:\\
$\bullet$ There exists a finite set $\mathscr{B}_\sigma\subset (0, T)$ such that
$\ddot{\sigma}=0$ on $(0,T)\setminus \mathscr{B}_\sigma$, and every $t\in \mathscr{B}_\sigma$
satisfies (B-i) and (B-ii).\\
$\bullet$ $\sigma(0),\sigma(T)\in\partial\Omega$,  $\dot\sigma^+(0)$ and $\dot\sigma^-(T)$
  are perpendicular to $\partial\Omega$.

 Define $\gamma:[0,T]/\{0,T\}\to \Omega$  by
$$
\gamma(t)=\left\{\begin{array}{ll}
\sigma(2t)\;\hbox{for}\;0\le t\le T/2,\\
\sigma(2T-2t)\;\hbox{for}\;T/2\le t\le T.
\end{array}\right.
$$
It satisfies $\gamma(T-t)=\gamma(t)\;\forall t$, and thus is a
 brake billiard trajectory in our sense above.
 Conversely, our brake billiard trajectory $\sigma:\mathbb{R}/(T\mathbb{Z})\to\overline{\Omega}$
 must satisfy $\dot{\sigma}^-(0)=-\dot{\sigma}^+(0)$ and $\dot{\sigma}^-(T/2)=-\dot{\sigma}^+(T/2)$.
 It follows that $\{0, T/2\}\subset\mathscr{B}_\sigma$ and
$\{\dot{\sigma}^+(0), \dot{\sigma}^-(T/2)\}\subset (T_{\sigma(t)}\partial\Omega)^\bot\setminus\{0\}$.
 These mean that the restriction of $\sigma$ to $[0, T/2]$ is
  a brake billiard trajectory  in the Irie's sense.}
\end{remark}

Ghomi \cite{Gh04} defined a {\bf generalized closed (or periodic) billiard trajectory} $\mathcal{T}$
in a  convex body $\Delta\subset\mathbb{R}^n$ as  a sequence of distinct
boundary points $q_i\in\partial K$, $i\in\mathbb{Z}/N\mathbb{Z}$, $N\ge 2$,  such that
for every $i$,
$$
\nu_i:=\frac{q_i-q_{i-1}}{|q_i-q_{i-1}|}+ \frac{q_{i}-q_{i+1}}{|q_{i}-q_{i+1}|}
$$
is an outward support vector of $\Delta$ at $q_i$.  In particular,
if $N = 2$ it called a {\bf bouncing ball orbit}.
The length of $\mathcal{T}$ is defined by
$$
{\rm length}(\mathcal{T}):=\sum^N_{i=1}|q_i-q_{i+1}|.
$$

\begin{remark}\label{rm:billT.0}
{\rm If $\Delta$ is smooth the definitions of the  generalized closed  billiard trajectory
and the closed  billiard trajectory above  coincide.
In fact, for a smooth convex body in $\Delta\subset\mathbb{R}^n$, a nonconstant, continuous,
and piecewise $C^\infty$ path $\sigma:\mathbb{R}/(T\mathbb{Z})\to \Delta$
is a  periodic billiard trajectory in  $\Delta$ with $\mathscr{B}_\sigma=\{t_1<\cdots<t_{m-1}\}$
if and only if the sequence
$$
q_0=\sigma(0), q_1=\sigma(t_1),\cdots,
q_{m-1}=\sigma(t_{m-1}), q_m=\sigma(T)
$$
is a generalized periodic billiard trajectory in  $\Delta$.
}
\end{remark}

%Without special statements, we always suppose below
Suppose that
 \textsf{$\Delta\subset\mathbb{R}^n_q$ and $\Lambda\subset\mathbb{R}^n_p$
 are two smooth convex bodies containing the origin in their interiors}. Then
 $\Delta\times \Lambda$ is only a smooth manifold with corners $\partial \Delta\times\partial \Lambda$
 in the standard symplectic space
 $(\mathbb{R}^{2n},\omega_0)=(\mathbb{R}^n_q\times \mathbb{R}^n_p, dq\wedge dp)$.
 Note that $\partial(\Delta\times\Lambda)=(\partial\Delta\times\partial\Lambda)\cup
({\rm Int}(\Delta)\times\partial\Lambda)\cup
(\partial\Delta\times{\rm Int}(\Lambda))$.
 Since
$j_{\Delta\times \Lambda}(q,p)=\max\{j_\Delta(q), j_\Lambda(p)\}$, we have
$$
\nabla j_{\Delta\times \Lambda}(q,p)=\left\{
\begin{array}{cc}
 (0,\nabla j_\Lambda(p))  &\forall (q,p)\in{\rm Int}(\Delta)\times\partial \Lambda, \\
  (\nabla j_\Delta(q),0) &\forall (q,p)\in \partial \Delta\times{\rm Int}(\Lambda)
\end{array}
\right.
$$
and thus %(noting \cite[(1.23)]{HoZe94} and our $J$ is $-J$ there)
$$
\mathfrak{X}(q,p):=J_0\nabla j_{\Delta\times \Lambda}(q,p)=\left\{
\begin{array}{cc}
 (-\nabla j_\Lambda(p),0)  &\forall (q,p)\in{\rm Int}(\Delta)\times\partial \Lambda, \\
  (0, \nabla j_\Delta(q)) &\forall (q,p)\in \partial \Delta\times{\rm Int}(\Lambda).
\end{array}
\right.
$$

\begin{definition}[\hbox{\cite[Definition~2.9]{AAO14}}]\label{def:billT.3}
{\rm A {\bf closed $(\Delta, \Lambda)$-billiard trajectory} is  a continuous and piecewise smooth map
$\gamma:\mathbb{R}/(T\mathbb{Z})\to \partial(\Delta\times\Lambda)$ satisfying the following two conditions:
%\begin{description}
\begin{enumerate}
\item[\bf (BT1)] For every $t\notin\mathscr{B}_\gamma:=\{t\in \mathbb{R}/(T\mathbb{Z})\,|\,\gamma(t)\in
\partial\Delta\times\partial\Lambda\}$ one has
$$
\dot{\gamma}(t)=\kappa\mathfrak{X}(\gamma(t))
$$
for some positive constant $\kappa$.
 %it holds that  on the complementary of  in $\mathbb{R}/(T\mathbb{Z})$;

\item[\bf (BT2)] For every $t\in\mathscr{B}_\gamma$
the left and right derivative of ${\gamma}$ at it exists, and
    \begin{equation}\label{e:linesymp3}
   \dot{\gamma}^\pm(t)\in\{-\lambda(\nabla j_\Lambda(\gamma_p(t)),0)+\mu(0,\nabla j_\Delta(\gamma_q(t)))\,|\,
    \lambda\ge 0, \;\mu\ge 0,\;(\lambda,\mu)\ne (0,0)\}
    \end{equation}
    with $\gamma(t)=(\gamma_q(t),\gamma_p(t))$.
%\end{description}
\end{enumerate}
In general, for any  two smooth convex bodies   $\Delta\subset\mathbb{R}^n_q$ and $\Lambda\subset\mathbb{R}^n_p$,
a continuous and piecewise smooth map $\gamma:\mathbb{R}/(T\mathbb{Z})\to
\partial(\Delta\times\Lambda)$ is said to be a {\bf closed $(\Delta, \Lambda)$-billiard trajectory} if there exist
$\bar{q}\in{\rm Int}(\Delta)$ and $\bar{p}\in {\rm Int}(\Lambda)$
such that $\gamma-(\bar{q},\bar{p})$ is a closed $(\Delta-\bar{q}, \Lambda-\bar{p})$
 -billiard trajectory in the above sense.
  Such a closed $(\Delta, \Lambda)$-billiard trajectory is called {\bf proper} (resp. {\bf gliding}) if $\mathscr{B}_\gamma$ is a finite set (resp. $\mathbb{R}/(T\mathbb{Z})$,
i.e., $\gamma(\mathbb{R}/(T\mathbb{Z}))\subset\partial\Delta\times\partial\Lambda$ completely).
If $\tau$ is a linear anti-symplectic involution on $\mathbb{R}^{2n}$, and
$\Delta\times\Lambda\subset\mathbb{R}^n_q\times\mathbb{R}^n_p\equiv\mathbb{R}^{2n}$
is $\tau$-invariant, we call a closed $(\Delta, \Lambda)$-billiard trajectory
$\gamma:\mathbb{R}/(T\mathbb{Z})\to\partial(\Delta\times\Lambda)$ a {\bf $\tau$-brake
$(\Delta, \Lambda)$-billiard trajectory} if $\gamma(T-t)=\tau(\gamma(t))\;\forall t$.
}
\end{definition}

So if $\Lambda$ (resp. $\Delta$) is  centrally symmetric,
we have the notion of a $\tau_0$-brake (resp. $\hat{\tau_0}$) $(\Delta, \Lambda)$-billiard trajectory
on $\partial(\Delta\times\Lambda)$, where  $\tau_0$ and $\hat{\tau_0}$
are the anti-symplectic involution on $(\mathbb{R}^{2n},\omega_0)=(\mathbb{R}^n_q\times \mathbb{R}^n_p, dq\wedge dp)$
given by (\ref{e:1.can-inv}) and  by $\hat{\tau_0}(q,p)=(-q,p)$, respectively.
(Clearly, both involutions are symmetric with respect to $\langle\cdot,\cdot\rangle_{\mathbb{R}^{2n}}$.)

The projection curve $\pi_q(\gamma)$ of a closed $(\Delta, \Lambda)$-billiard
trajectory $\gamma$ is called a {\bf $\Lambda$-billiard trajectory in $\Delta$} (\cite{AAO14}).
If $\Lambda$ (resp. $\Delta$) is  centrally symmetric,  the projection curve $\pi_q(\gamma)$ of a $\tau_0$-brake (resp. $\hat{\tau_0}$-brake)  $(\Delta, \Lambda)$-billiard trajectory $\gamma$
is called {\bf a brake $\Lambda$-billiard trajectory in $\Delta$} (resp.
{\bf a brake $\Lambda$-billiard trajectory of second class in $\Delta$}).

\begin{proposition}[\hbox{\cite{AAO14, JinLu19}}]\label{prop:brake}
 For any  two smooth convex bodies   $\Delta\subset\mathbb{R}^n_q$ and $\Lambda\subset\mathbb{R}^n_p$,
 we have:
\begin{enumerate}
 \item[{\bf (i)}] If both $\Delta$ and $\Lambda$ are also strictly convex (i.e.,
they have strictly positive Gauss curvatures at every point of their boundaries),
 every closed $(\Delta, \Lambda)$-billiard trajectory is either proper or  gliding (\cite[Proposition~2.12]{AAO14}), and $c_{\rm HZ}(\Delta\times\Lambda)$ equals the length, with respect to the support function $h_{\Lambda}$, of the shortest periodic $\Lambda$-billiard trajectory in $\Delta$.
 (\cite[\S2]{AAO14}).
 \item[{\bf (ii)}] Every proper closed $(\Delta, \Lambda)$-billiard trajectory
$\gamma:\mathbb{R}/(T\mathbb{Z})\to \partial(\Delta\times\Lambda)$
cannot be contained in $\Delta\times\partial \Lambda$ or $\partial\Delta\times\Lambda$
(\cite[Claim~1.35]{JinLu19}). Consequently, $\gamma^{-1}(\partial\Delta\times\partial\Lambda)$ is nonempty.
\item[{\bf (iii)}] $\gamma:\mathbb{R}/(T\mathbb{Z})\to \partial(\Delta\times\Lambda)$ is a closed
 $(\Delta, \Lambda)$-billiard trajectory  if and only if it is a generalized closed characteristic on $\partial(\Delta\times\Lambda)$ (\cite{AAO14}, see also \cite[Claim~1.38]{JinLu19} for a proof.)
 \item[{\bf (iv)}] If $\Delta\times\Lambda\subset\mathbb{R}^n_q\times\mathbb{R}^n_p\equiv\mathbb{R}^{2n}$
is $\tau$-invariant for a linear anti-symplectic involution $\tau$ on $\mathbb{R}^{2n}$,
$\gamma:\mathbb{R}/(T\mathbb{Z})\to \partial(\Delta\times\Lambda)$ is a $\tau$-brake
 $(\Delta, \Lambda)$-billiard trajectory  if and only if it is a generalized $\tau$-brake closed characteristic on $\partial(\Delta\times\Lambda)$.
  \end{enumerate}
 \end{proposition}

 (iv) easily follows from (iii) and the above definitions.

 The action of a proper closed $(\Delta, \Lambda)$-billiard trajectory $\gamma$
with $m$ bouncing points is given by
\begin{equation}\label{e:linesymp4}
A(\gamma)=\sum^{m}_{j=0}h_\Lambda(q_j-q_{j+1})
\end{equation}
where $q_j=\pi_q(\gamma(t_i))$ are the projections to $\mathbb{R}^n_q$
of the bouncing points $\{\gamma(t_i)\,|\,t_i\in\mathscr{B}_\gamma\}$,  and $q_{m+1}=q_0$
(\cite[(7)]{AAO14}). In particular, if $\Lambda=B^n(1)$
the projection curve $\pi_q(\gamma)$ has the length
\begin{equation}\label{e:linesymp4.1}
L(\pi_q(\gamma))=\sum^{m}_{j=0}|q_{j+1}-q_j|=A(\gamma).
\end{equation}
This also holds for any gliding closed $(\Delta, B^n(1))$-billiard trajectory
$\gamma$ if $\Delta$ is strictly convex (\cite[Claim~1.37]{JinLu19}).

\begin{claim}\label{cl:billT.1}
A $B^n(1)$-billiard trajectory in $\Delta$ coming from
a proper closed $(\Delta,\break B^n(1))$-billiard trajectory
is a closed billiard trajectory in $\Delta$ as above.
Conversely, every closed billiard trajectory in  $\Delta$,  $\sigma:\mathbb{R}/(T\mathbb{Z})\to \Delta$,
is the projection to  $\Delta$ of a proper closed $(\Delta, B^n(1))$-billiard trajectory
whose action is equal to the length of $\sigma$.
\end{claim}

The second claim is a direct consequence of \cite[Claim~1.39]{JinLu19}.
From Claim~\ref{cl:billT.1} and Definition~\ref{def:billT.3} we deduce

\begin{claim}\label{cl:billT.2}
A brake $B^n(1)$-billiard trajectory in $\Delta$ coming from
a proper $\tau_0$-brake $(\Delta, B^n(1))$-billiard trajectory
is a brake billiard trajectory in $\Delta$ as above.
Conversely, every brake billiard trajectory $\sigma$ in  $\Delta$
is the projection to  $\Delta$ of a proper $\tau_0$-brake $(\Delta, B^n(1))$-billiard trajectory
whose action is equal to the length of $\sigma$.
Moreover, if $\Delta$ is  centrally symmetric, these also hold after
``brake $B^n(1)$-billiard trajectory in $\Delta$",
``brake billiard trajectory in $\Delta$" and ``$\tau_0$-brake"
are replaced by ``brake $B^n(1)$-billiard trajectory of second class in $\Delta$",
``brake billiard trajectory of second class in $\Delta$" and ``$\hat{\tau}_0$-brake",
respectively.
\end{claim}

%Using Proposition~\ref{prop:brake} and  Theorems~\ref{th:convex} and \ref{th:EHconvex}
%we prove the following theorem in Section~\ref{sec:BillT}.
Proposition~\ref{prop:brake} and  Theorems~\ref{th:convex} and \ref{th:EHconvex} immediately lead to next result.

\begin{thm}\label{th:billT.1}
For  a linear anti-symplectic involution $\tau$ on $(\mathbb{R}^{2n},\omega_0)$,
if the Lagrangian product of smooth convex bodies   $\Delta\subset\mathbb{R}^n_q$ and $\Lambda\subset\mathbb{R}^n_p$
are $\tau$-invariant, then
there is a $\tau$-brake $(\Delta, \Lambda)$  billiard trajectory $\gamma^{\ast}$ on $\partial(\Delta\times\Lambda)$ such that
\begin{eqnarray*}
&&c_{\rm EHZ,\tau}(\Delta\times \Lambda,\omega_0)=A(\gamma^{\ast})\\
&=&\min\{A(\gamma)>0\,|\,\gamma\;\text{is
a $\tau$-brake $(\Delta, \Lambda)$  billiard trajectory $\gamma^{\ast}$ on $\partial(\Delta\times\Lambda)$}\}.
\end{eqnarray*}
\end{thm}

%let us call the image of the projection curve $\gamma_q=\pi_q(\gamma)$
%a $a-\lambda$ {\bf billiard trajectory} in $\delta$, where $\pi_q:\mathbb{r}^{2n}\to\mathbb{r}^n_q$
%denote the projection to the configuration space.

For  a linear anti-symplectic involution $\tau$ on $(\mathbb{R}^{2n},\omega_0)$,
suppose that the Lagrangian product of convex bodies   $\Delta\subset\mathbb{R}^n_q$ and $\Lambda\subset\mathbb{R}^n_p$
are $\tau$-invariant. We define
\begin{equation}\label{e:linesymp5}
\left.\begin{array}{ll}
&\zeta_\Lambda^\tau(\Delta)=c_{\rm EHZ,\tau}(\Delta\times \Lambda)\quad\hbox{and}\\
&\zeta^\tau(\Delta):=\zeta_\Lambda^\tau(\Delta)\quad\hbox{if $\Lambda=B^n(1)$}.
\end{array}\right\}
\end{equation}
Clearly, $\zeta_{\Lambda_1}^\tau(\Delta_1)\le \zeta_{\Lambda_2}^\tau(\Delta_2)$ if
both are well-defined and $\Lambda_1\subset\Lambda_2$ and $\Delta_1\subset \Delta_2$.
 Claim~\ref{cl:billT.2} and Theorem~\ref{th:billT.1} yield next result.

\begin{claim}\label{cl:billT.6}
For a smooth convex body   $\Delta\subset\mathbb{R}^n_q$  it holds that
\begin{eqnarray*}
&&\zeta^{\tau_0}(\Delta)\le\inf\{L(\sigma)\,|\,\hbox{$\sigma$ is a brake billiard trajectory  in $\Delta$}\},\\
&&\zeta^{\hat{\tau}_0}(\Delta)\le\inf\{L(\sigma)\,|\,\hbox{$\sigma$ is a brake billiard trajectory of
 second class in $\Delta$}\}
\end{eqnarray*}
if $\Delta$ is  centrally symmetric. Moreover%, if $\Delta$ is also strictly convex then
\begin{eqnarray*}
&\zeta^{\tau_0}(\Delta)=\inf\{L(\sigma)\,|\,\hbox{$\sigma$ is a brake $B^n(1)$-billiard trajectory  in $\Delta$}\},\\
&\zeta^{\hat{\tau}_0}(\Delta)=\inf\{L(\sigma)\,|\,\hbox{$\sigma$ is a brake $B^n(1)$-billiard trajectory of
 second class in $\Delta$}\}
\end{eqnarray*}
if $\Delta$ is  centrally symmetric in addition.
\end{claim}

As in the proof of \cite[Theorem~1.1]{AAO14} using Corollary~\ref{cor:Brun.1} we
prove next theorem in Section~\ref{sec:BillT}.

\begin{thm}\label{th:billT.2}
Let $\tau$ be  a linear anti-symplectic involution  on $(\mathbb{R}^n_q\times \mathbb{R}^n_p,\,\omega_0)$,
and let convex bodies  $\Delta_1, \Delta_2\subset\mathbb{R}^n_q$ and $\Lambda\subset\mathbb{R}^n_p$
be such that  $\Delta_1\times\Lambda$, $\Delta_2\times\Lambda$
and $(\Delta_1+\Delta_2)\times\Lambda$
are $\tau$-invariant.
%Suppose that  ${\rm Int}(\Delta_1)\cap{\rm Int}(\Delta_2)\times {\rm Int}(\Lambda)$
%contains a fixed point $(\bar{q},\bar{p})$ of $\tau$.
 Then
\begin{equation}\label{e:linesymp6}
 \zeta^\tau_\Lambda(\Delta_1+\Delta_2)\ge \zeta^\tau_\Lambda(\Delta_1)+ \zeta^\tau_\Lambda(\Delta_2)
   \end{equation}
   and the equality holds if  there exist
    $c_{\rm EHZ,\tau}$-carriers for $\Delta_1\times \Lambda$ and $\Delta_2\times \Lambda$ which coincide up to dilation and translation in ${\rm Fix}(\tau)$.
\end{thm}

Recall that the inradius and the circumradius of a convex body $\Delta\subset\mathbb{R}^n$
are
\begin{eqnarray*}
&&r(\Delta)=\max\{r>0\,|\, B^n(q,r)\subset \Delta\;\hbox{for some}\;q\in \Delta\}\quad\hbox{and}\\
&&R(\Delta)=\min\{R>0\,|\, B^n(q,R)\supseteq \Delta\;\hbox{for some}\;q\in\mathbb{R}^n\}
\end{eqnarray*}
respectively, i.e., the former is the radius
of the largest ball contained in $\Delta$, and the latter
the radius of the smallest ball containing $\Delta$ (\cite[page 129]{Sch93}).
In Section~\ref{sec:BillT}, we prove next result.

\begin{thm}\label{th:billT.3}
For  a convex body $\Delta\subset\mathbb{R}^n_q$
there holds
\begin{equation}\label{e:linesymp8}
4 r(\Delta)\le\zeta^{\tau_0}(\Delta)\le 4 R(\Delta).
\end{equation}
Moreover, if $\Delta$ is  centrally symmetric, then $B^{n}(\bar{q},r)\subset\Delta$ (resp.
$\Delta\subset \overline{B^{n}(\bar{q},R)}$) implies that
$B^{n}(0,r)\subset\Delta$ (resp. $\Delta\subset \overline{B^{n}(0,R)}$), and
the conclusions in (\ref{e:linesymp8}) also hold true if ``$\tau_0$" and
``brake billiard trajectory  in $\Delta$" are replaced by
``$\hat\tau_0$" and ``brake billiard trajectory of second class in $\Delta$", respectively.
\end{thm}

%For a smooth convex body $\Delta\subset\mathbb{R}^n_q$, (\ref{e:linesymp8}) and
% Claim~\ref{cl:billT.6} imply
%\begin{equation}\label{e:linesymp8.1}
% 4 r(\Delta)\le \inf\{L(\sigma)\,|\,\hbox{$\sigma$ is a brake billiard trajectory  in $\Delta$}\}.
%\end{equation}

%{\rm (ii)} If $\Delta$ is contained in the closure of the ball $B^{n}(\bar{q},R)$, then
%\begin{equation}\label{e:linesymp8.2}
%\zeta^{\tau_0}(\Delta)\le 4R.
%\end{equation}

 The {\bf width} of a convex body $\Delta\subset\mathbb{R}^n_q$ is the thickness of the
narrowest slab which contains $\Delta$, i.e.,
${\rm width}(\Delta)=\min\{h_\Delta(u)+ h_\Delta(-u)\,|\, u\in S^{n-1}\}$,
where $S^{n-1}=\{u\in\mathbb{R}^n\;\&\;|u|=1\}$. Let
\begin{eqnarray}\label{e:linesymp12}
&&S^{n-1}_\Delta:=\{u\in S^{n-1}\,|\, {\rm width}(\Delta)=h_\Delta(u)+ h_\Delta(-u)\},\\
&&H_u:=\{x\in\mathbb{R}^n\,|\,\langle x, u\rangle_{\mathbb{R}^n}=(h_\Delta(u)- h_\Delta(-u))/2\},\label{e:linesymp13}\\
&&Z^{2n}_\Delta:=([-{\rm width}(\Delta)/2, {\rm width}(\Delta)/2]\times\mathbb{R}^{n-1})\times ([-1,1]\times\mathbb{R}^{n-1}).\label{e:linesymp14}
\end{eqnarray}

\begin{thm}\label{th:billT.4}
 For $u\in S^{n-1}_\Delta$ and $\bar{q}\in H_u$  there holds
\begin{eqnarray}\label{e:linesymp16}
\zeta^{\tau_0}(\Delta)\le  c_{\rm EHZ,\tau_0}(Z^{2n}_\Delta)=2{\rm width}(\Delta).
\end{eqnarray}
\end{thm}

This will be proved in Section~\ref{sec:BillT}.
Since ${\rm width}(\Delta)\le (n+1)r(\Delta)$ by \cite[(1.2)]{Sch09},
 we have
  \begin{eqnarray}\label{e:linesymp17}
\zeta^{\tau_0}(\Delta)\le 2(n+1)r(\Delta).
\end{eqnarray}
This is the corresponding form of \cite[Theorem~1.3]{AAO14}:
$\xi(\Delta)\le 2(n+1)r(\Delta)$ for a smooth convex body $\Delta\subset\mathbb{R}^n_q$.
\\

%\subsection{Furthermore research problems}\label{sec:future}

\noindent{\bf Organization of the paper}. %The paper is organized as follows.
 The arrangements of the paper is as follows. \\
 $\bullet$ Section~\ref{section:space} gives  related preparations.\\
 $\bullet$ Section~\ref{sec:EH.2} gives the variational explanation for $c_{\rm EH,\tau_0}$.\\
 $\bullet$ Section~\ref{sec:convex}  proves Theorem~\ref{th:convex}.\\
  $\bullet$ Section~\ref{sec:EH.3} proves  Theorems~\ref{th:EHconvex}, \ref{th:EHproduct}. \\
 $\bullet$ Section~\ref{sec:EH.4} proves  Theorem~\ref{th:EHcontact}.\\
 $\bullet$ Section~\ref{sec:Brunn} proves  Theorem~\ref{th:Brun.2} and Corollaries~\ref{cor:Brun.2}, \ref{cor:Brun.4}.\\
 $\bullet$ Section~\ref{sec:BillT} proves Theorems~\ref{th:billT.2},~\ref{th:billT.3}, \ref{th:billT.4}.\\

%2 Preliminary results
%2.1 Notation and conventions

%Preliminaries %(预备知识)Variational  preparations

%\section{Variational frame and related preparations}\label{section:space}

\section{Preliminaries}\label{section:space}
\setcounter{equation}{0}

For the spaces defined in (\ref{e:spacees}),
by Propositions~3,4 in \cite[pages 84-85]{HoZe94}, we immediately obtain the following
two results.

\begin{proposition}\label{prop:compact}
   Assume $t>s\geq 0$. Then the inclusion map $I_{t,s}:{E}^t_{\tau_0}\rightarrow {E}^s_{\tau_0}$ is compact.
\end{proposition}

\begin{proposition}\label{prop:e1}
  Assume $s>\frac{1}{2}$. If $x\in E^s_{\tau_0}$, then $x$ is continuous, $x(1+t)=x(1)$ and $x(-t)=\tau_0x(t)$
  for all $t\in\mathbb{R}$. Moreover, there exists a constant $c=c_s$ such that
  $$
  \sup_{t\in [0,1]} |x(t)|\leq c\|x\|_{s}.
  $$
\end{proposition}

Define  $\mathfrak{a}:\mathbb{E}\rightarrow\mathbb{R}$ by
\begin{equation}\label{e:aaction}
\mathfrak{a}(x)=\frac{1}{2}(\|x^+\|^2_{1/2}-\|x^-\|^2_{1/2}).
\end{equation}
It is of class $C^\infty$ and has the gradient
$\nabla \mathfrak{a}(x)=x^+-x^-\in \mathbb{E}$.

\begin{remark}\label{rem:action}
 {\rm  For $x\in C^{1}(S^1,\mathbb{R}^{2n})$ satisfying $x(-t)=\tau_0x(t)\;\forall t\in\mathbb{R}$,
  there holds $x\in E^1_{\tau_0}$ and
  $$
  \mathfrak{a}(x)=\frac{1}{2}\int_0^1\langle-J_0\dot{x},x\rangle_{\mathbb{R}^{2n}}=A(x),
  $$
  where $A(x)$ is the action of $x$ defined in (\ref{e:action1}).
  In fact, let us write {\small$x=\sum_{j\in\mathbb{Z}}e^{2\pi jtJ_0}x_j$}. Then
  $$
  x(-t)=\sum_{j\in\mathbb{Z}}e^{-2\pi jtJ_0}x_j\quad\hbox{and}\quad
  \tau_0x(t)=\sum_{j\in\mathbb{Z}}\tau_0e^{2\pi jtJ_0}x_j=\sum_{j\in\mathbb{Z}}e^{-2\pi jtJ_0}\tau_0x_j.
  $$
  It follows that $x_j=\tau_0x_j$, i.e., $x_j\in L_0$,  $\forall j\in\mathbb{Z}$. Moreover,
   $$
   -J_0\dot{x}=\sum_{0\ne j\in\mathbb{Z}}2\pi je^{2\pi jtJ_0}x_j
   $$
      and $-J_0\dot{x}\in L^2$ implies that $\sum_{k\in \mathbb{Z}} |2\pi j|^2|x_j|^2<\infty$.
  It is easily computed that
  \begin{eqnarray*}
  \frac{1}{2}\int_0^1\langle -J_0\dot{x},x\rangle_{\mathbb{R}^{2n}}&=&\sum_{j\in\mathbb{Z}}2\pi |j||x_j|^2\\
                     &=&2\pi\sum_{j>0}j|x_j|^2-2\pi\sum_{j<0}|j||x_j|^2\\
                     &=&\|x^+\|^2_{1/2}-\|x^+\|^2_{1/2}.
  \end{eqnarray*}}
\end{remark}

For a $C^2$ function $H:\mathbb{R}^{2n}\rightarrow\mathbb{R}$ satisfying
(H1) and (H3) under Definition~\ref{def:deform},
%\begin{description}
%\item[(H1)] $H(z)=H(N_0z)$,
%\item[(H2)] there exist $z_0\in L_0$, real numbers $a> 2\pi$ and $b$
%such that $H(z)=a|z|^2+ \langle z, z_0\rangle+ b$ outside a compact subset of $\mathbb{R}^{2n}$,
%\end{description}
we define
$$
\hat{b}_H:L^2([0,1])\rightarrow\mathbb{R},\; x\mapsto \int_0^{1}H(x(t))dt.
$$
Clearly $\hat{b}_H$ is differentiable and has gradient $\nabla \hat{b}_H(x)=\nabla H(x)\in L^2$. Next we define
$$
\mathfrak{b}_H:\mathbb{E}\rightarrow\mathbb{R},\;x\mapsto \hat{b}_H(j(x)),
$$
 where
 $j: \mathbb{E}\rightarrow L^2$ is the inclusion map.
 Let $j^{\ast}:L^2\rightarrow \mathbb{E}$ be the adjoint operator of $j$
  defined by  $\langle j(x),y\rangle_{L^2}=\langle x, j^{\ast}(y)\rangle_{1/2}$
 for all $x\in \mathbb{E}$ and $y\in L^2$. A direct computation yields $\nabla \mathfrak{b}_H(x)=j^{\ast}\nabla H(x)$ for $x\in \mathbb{E}$.  By the arguments in \cite[pages 86-87]{HoZe94}, we obtain the following two propositions.

 \begin{proposition}\label{prop:jast}
   $j^{\ast}(L^2)\subset E^1_{\tau_0}$ and $\|j^\ast y\|_1\le \|y\|_{L^2}\;\forall y\in L^2$, so $j^{\ast}$ is a compact operator.
 \end{proposition}

 \begin{proposition}\label{prop:Lip}
   The  gradient $\nabla \mathfrak{b}_H:\mathbb{E}\rightarrow \mathbb{E}$ is compact and
    for some constant $C_1>0$ it holds that
   $$
   \|\nabla \mathfrak{b}_H(x)-\nabla \mathfrak{b}_H(y)\|_{1/2}\leq C_1\|x-y\|_{1/2},\,\forall x, y\in \mathbb{E}.
   $$
     Moreover, there exist positive numbers $C_2$ and $C_3$ such that $|\mathfrak{b}_H (x)|\leq C_2 \|x\|^2_{L^2}+C_3 $,
     $\forall x\in \mathbb{E}$.
 \end{proposition}

 Then for the above $H$ Proposition~\ref{prop:Lip} implies that the functional defined by (\ref{e:EH.1.1}),
% Define functional
\begin{equation}\label{e:Phi}
\Phi_H:\mathbb{E}\rightarrow\mathbb{R}^{2n},\;x\mapsto \mathfrak{a}(x)-\mathfrak{b}_H(x),
\end{equation}
% Then $\Phi_H$
 is $C^{1,1}$ and the gradient $\nabla \Phi_H=x^+-x^--\nabla \mathfrak{b}_H$ satisfying global Lipschitz condition.
Hence the negative gradient flow of $\Phi_{H}$ defined by
\begin{eqnarray*}
\frac{d\phi^t(x)}{dt}=-\nabla\Phi_{H}(\phi^t(x))\quad\hbox{and}\quad
\phi^0(x)=x
\end{eqnarray*}
exists for all $t\in\mathbb{R}$ and $x\in \mathbb{E}$.

 \begin{proposition}\label{prop:flow}
  The flow of $\dot{x}=-\nabla\Phi_{{H}}(x)$ admits the representation
  $$
  x\cdot t =e^tx^-+x^0+e^{-t}x^{+}+K(t,x),
  $$
  where $K:\mathbb{R}\times \mathbb{E}\rightarrow \mathbb{E}$ is continuous and maps bounded sets into precompact sets.
\end{proposition}

Note that $\Phi_{{H}}$ is also a $C^1$ functional on the Hilbert space
$E^{1/2}$ (denoted by $\widehat{\Phi}_H$ when it is consider as a functional on $E^{1/2}$)
and that for $x\in \mathbb{E}$ there holds $\nabla\widehat{\Phi}_H(x)=\check{\tau}_0\nabla\widehat{\Phi}_H(x)$ and so $\nabla \Phi_H(x)=\nabla\widehat{\Phi}_H(x)$.
Using \cite[page 88, Lemma~5]{HoZe94}  and the principle of symmetric criticality
due to Palais \cite{Palais79} we can obtain next result.

\begin{proposition}\label{solution}
  If $x\in \mathbb{E}$ is a critical point of $\Phi_{H}$, then
$x$ is in $C^2(S^1,\mathbb{R}^{2n})$ and satisfies
$$
\dot{x}=J_0\nabla H (x)\quad\hbox{and}\quad
x(-t)=\tau_0x(t).
$$
\end{proposition}

Recall that a $C^1$ functional $f$ on a Banach space $X$ is said to satisfy
the {\bf Palais-Smale condition} ((PS) {\bf condition}, for short)  if
every sequence $(u_n)\subset X$ with $f'(u_n)\to 0$ and $(f(u_n))$  bounded
 has a subsequence converging to $u\in X$.

We say $H\in C^2(\mathbb{R}^{2n},\mathbb{R}_+)$ to be {\bf  nonresonant} if it satisfies
(H3) under Definition~\ref{def:deform}  with $a\notin \mathbb{Z}\pi$.
By slightly modifying the proof  of \cite[page 89, Lemma~6]{HoZe94}(see \cite[Proposition~2.10]{JinLu19} )
 we get

\begin{proposition}\label{prop:PSmale}
Suppose that $H\in C^2(\mathbb{R}^{2n},\mathbb{R})$ satisfies  (H1) and (H3) under Definition~\ref{def:deform},
and is also  nonresonant. Then each sequence $(x_k)\subset \mathbb{E}$ with $\nabla\Phi_{H}(x_k)\to 0$
 has a convergent subsequence. In particular,  $\Phi_H$ satisfies the (PS) condition.
\end{proposition}

\section{Proof of Theorem~\ref{th:EH.1.6}}\label{sec:EH.2}
\setcounter{equation}{0}

%\section{The variational explanation for $c_{\rm EH,\tau_0}$}\label{sec:EH.2}
%\setcounter{equation}{0}

%\begin{thm}\label{th:EH.1.6}%[\hbox{\cite[Prop.3.4.1]{Sik90}}]
%If $H\in C^\infty(\mathbb{R}^{2n},\mathbb{R})$ satisfies (H1)-(H3)
%in Section~\ref{sec:1.EH}
%and is also nonresonant, then $c_{\rm EH,\tau_0}(H)$ is a positive critical value of $\Phi_H$ on $\mathbb{E}$.
%\end{thm}

The main aim of this section  is to prove Theorem~\ref{th:EH.1.6}.
As in \cite{JinLu19} our arguments are closely related to  Sikorav's approach \cite{Sik90}.
The proof will be completed by serval propositions.
Let $\Gamma$ be the set of all admissible deformations from Definition~\ref{def:deform}

\begin{proposition}\label{prop:EH.1.1}
  %\begin{description}
  \begin{enumerate}
   \item[\bf (i)] For any  $\gamma \in\Gamma$ and $\tilde{\gamma} \in\Gamma$, then there holds $\gamma \circ\tilde{\gamma} \in\Gamma$.
   \item[\bf (ii)] %(Intersection property.)
   Denote by $S^+$ the unit sphere in $\mathbb{E}^+$. For any $e\in \mathbb{E}^+\setminus\{0\}$ and $\gamma\in\Gamma$, there holds
       \begin{equation}\label{nonempty}
       \gamma(S^+)\cap (\mathbb{E}^-\oplus \mathbb{E}^0\oplus \mathbb{R}_{>0}e)\neq \emptyset.
       \end{equation}
  %\end{description}
  \end{enumerate}
\end{proposition}

See Section~3.1 in \cite{Sik90} for the proof of the above proposition.
\begin{proposition}\label{prop:EH.1.3}%[\hbox{\cite[Prop.3.3.1]{Sik90}}]
If $H\in C^0(\mathbb{R}^{2n},\mathbb{R}_+)$, then
\begin{equation}\label{e:EH.1.3}
c_{\rm EH,\tau_0}(H)\le\sup_{z\in\mathbb{R}^{2n}}\left(\pi|z_1|^2-H(z)\right),
\end{equation}
where $z=(x_1,\cdots,x_n,y_1,\cdots,y_n)^{T}$ and $z_1=(x_1,y_1)^{T}$.
\end{proposition}

\begin{proof}
   Let $e_1(t)=e^{2\pi tJ_0}X$, where $X=(1,0,\cdots,0)^{T}\in\mathbb{R}^{2n}$.
  Then $e_1\in \mathbb{E}^+$.
   For any $x\in \mathbb{E}^-\oplus \mathbb{E}^0\oplus\mathbb{R}_{>0}e_1$, we write $x=y+\lambda e_1$, where
$y\in \mathbb{E}^-\oplus \mathbb{E}^0$ and $\lambda>0$. Then
\begin{equation}\label{e:EH.2.0}
\mathfrak{a}(x)\le\frac{1}{2}\|\lambda e_1\|^2_{1/2}= \pi\lambda^2
\end{equation}
by (\ref{e:aaction}). Since $y\in \mathbb{E}^-\oplus \mathbb{E}^0$ and $e_1\in \mathbb{E}^+$, we deduce
\begin{eqnarray}\label{e:EH.2.1}
\int_0^1\langle x_1(t),e^{2\pi tJ_0}X_1\rangle_{\mathbb{R}^{2n}}&=&\int_0^1\langle x(t),e^{2\pi tJ_0}X\rangle_{\mathbb{R}^{2n}}\nonumber\\
&=&\int_0^1\langle \lambda e^{2\pi tJ_0}X,e^{2\pi tJ_0}X\rangle_{\mathbb{R}^{2n}}=\lambda,
\end{eqnarray}
where $x(t)=(x_1(t),x_2(t))\in\mathbb{R}^{2}\times\mathbb{R}^{2n-2}$ for each $t$ and $X=(X_1,X_2)\in\mathbb{R}^{2}\times\mathbb{R}^{2n-2}$.
It follows this and (\ref{e:EH.2.0})  that
$$
\mathfrak{a}(x)\le\pi\left(\int_0^1\langle x(t),e^{2\pi tJ_0}X\rangle_{\mathbb{R}^{2n}}\right)^2\le\pi\int_0^1|x_1(t)|^2.
$$
By  Proposition~\ref{prop:EH.1.1}(ii), we get
$$
\inf_{x\in\gamma(S^+)}\Phi_H(x) \le\sup_{x\in \mathbb{E}^-\oplus \mathbb{E}^0\oplus\mathbb{R}_{>0}e_1}\Phi_H(x)\le \sup_{z\in\mathbb{R}^{2n}}(\pi|z_1|^2-H(z)),\,\forall\,\gamma\in\Gamma.
$$
Then (\ref{e:EH.1.3}) follows.
\end{proof}

\begin{proposition}\label{prop:EH.1.4}
If $H\in C^\infty(\mathbb{R}^{2n},\mathbb{R}_+)$ satisfies the conditions (H2) and (H3)
under Definition~\ref{def:deform}, then
 $c_{\rm EH,\tau_0}(H)>0$.
\end{proposition}
\begin{proof}
By the assumption (H2), we can take $\hat{z}\in{\rm int}\,H^{-1}(0)\cap L_0$.
Define
$$
\gamma:\mathbb{E}\rightarrow \mathbb{E},\; x\mapsto \gamma(x)=\hat{z}+\varepsilon x,
$$
where $\varepsilon>0$ is a constant.  Clearly, $\gamma\in\Gamma$. Let us prove
$$
\inf_{y\in\gamma(S^+)}\Phi_H(y)>0
$$
for sufficiently small $\varepsilon$. Since
\begin{equation}\label{e:EH.2.2}
\Phi_H(\hat{z}+x)=1/2\|x\|^2_{1/2}-\int_0^1H(\hat{z}+x)\quad\forall x\in \mathbb{E}^+,
\end{equation}
it suffices to prove that
\begin{equation}\label{e:EH.2.3}
\lim_{\|x\|_{1/2}\rightarrow 0}\frac{\int_0^1 H(\hat{z}+x)}{\|x\|^2_{1/2}}=0.
\end{equation}
%%%%%%%%%%%%%%%%%%%%%%%%%%%%%%%%%%%%%%%%%%%%%%%%%%%%%%%%%%%%%%%%%%%%%%%%%%%%%%%%%%%%%%%%%%%%%%%%
%%%%%%%%%%%%%%%%%%%%%%%%%%%%%%%%%%%%%%%%%%%%%%%%%%%%%%%%%%%%%%%%%%%%%%%%%%%%%%%%%%%%%%%%%%%%%%%%
%%%In fact, it follows from the above equation that there exists a positive number
%%$\epsilon>0$ such that for all $x$ satisfying $\|x\|_E<\epsilon$, it holds that
%%$$
%%\frac{\int_0^1 H(z_0+x)}{\|x\|^2_E}<\frac{1}{4},
%%$$
%which is equivalent to
%$$
%\int_0^1 H(z_0+x)<\frac{1}{4}\|x\|^2_E.
%$$
%Then by (\ref{e:EH.2.2}), we have
%$$
%\Phi_H(z_0+x)\ge\frac{1}{4}\|x\|_E^2,\, \forall x\in E^+\quad\hbox{such that}\quad \|x\|_E<\epsilon.
%$$
%Especially we get that for $\varepsilon<\epsilon$ , there holds
%$$
%\inf_{y\in\gamma(S^+)}\Phi_H(y)\ge\frac{1}{4}\varepsilon^2>0.
%$$
%%%%%%%%%%%%%%%%%%%%%%%%%%%%%%%%%%%%%%%%%%%%%%%%%%%%%%%%%%%%%%%%%%%%%%%%%%%%%%%%%%%%%%%%%%
%%%%%%%%%%%%%%%%%%%%%%%%%%%%%%%%%%%%%%%%%%%%%%%%%%%%%%%%%%%%%%%%%%%%%%%%%%%%%%%%%%%%%%%%%
 Otherwise, suppose there exists a sequence $(x_j)\subset \mathbb{E}$ and $d>0$ satisfying
  \begin{equation}\label{e:EH.2.4}
  \|x_j\|_{1/2}\rightarrow 0 \quad \hbox{and} \quad \frac{\int_0^1 H(\hat{z}+x_j)}{\|x_j\|^2_{1/2}}
  \geq d>0\quad\forall j.
  \end{equation}
  Let $y_j=\frac{x_j}{\|x_j\|_{1/2}}$. Then $\|y_j\|_{1/2}=1$. By Proposition
  \ref{prop:compact}, $(y_j)$ has a convergent subsequence in $L^2$.
  By a standard result in $L^p$ theory, (see \cite[Th.4.9]{Br11}),
  we have $w\in L^2$ and a subsequence of $(y_j)$, still denoted by $(y_j)$, such that
   $y_j(t)\rightarrow y(t)$   a.e. on $(0,1)$ and that
   $|y_j(t)|\leq w(t)$   a.e. on $ (0,1)$  for each $j$.  Since (H2) implies that $H$ vanishes near $\hat{z}$,
  by (H3) and the Taylor expansion of $H$ at $\hat{z}\in\mathbb{R}^{2n}$,
  we have constants $C_1>0$ and $C_2>0$ such that $H(\hat{z}+z)\le C_1|z|^2$
  and  $H(\hat{z}+z)\le C_2|z|^3$ for all $z\in\mathbb{R}^{2n}$.
  It follows that
  \begin{eqnarray*}
  &&\frac{H(\hat{z}+x_j(t))}{\|x_j\|_{1/2}^2}\leq C_1\frac{|x_j(t)|^2}{\|x_j\|_{1/2}^2}
  =C_1|y_j(t)|^2\le C_1w(t)^2,\quad \hbox{a.e. on}\quad  (0,1),\;\forall j,\\
 &&  \frac{H(\hat{z}+x_j(t))}{\|x_j\|_{1/2}^2}\leq C_2\frac{|x_j(t)|^3}{\|x_j\|_{1/2}^2}
  =C_2|x_j(t)|\cdot |y_j(t)|^2\le C_2|x_j(t)|w(t)^2,\\
  &&\hspace{60mm} \hbox{a.e.  on }\quad  (0,1),\;\forall j.
   \end{eqnarray*}
  The first claim in (\ref{e:EH.2.4}) implies that $(x_j)$
  has a subsequence such that
  $$
  x_{j_l}(t)\rightarrow 0, \quad \text{a.e. on}\quad   (0,1).
  $$
 Hence the Lebesgue dominated convergence theorem leads to
  $$
  \int_0^1 \frac{H(\hat{z}+x_{j_l}(t))}{\|x_{j_l}\|_{1/2}^2}\rightarrow 0.
  $$
  This contradicts the second claim in (\ref{e:EH.2.4}).
  \end{proof}

\noindent{\it Proof of Theorem~\ref{th:EH.1.6}}.\quad
Let us define
   $$
   \mathcal{F}=\{\gamma(S^+)\,|\,\gamma\in\Gamma
   \;\text{and}\;\inf(\Phi_H|\gamma(S^+))>0\}.
   $$
   Then  $c_{\rm EH,\tau_0}(H)=\sup_{F\in\mathcal{F}}\inf_{x\in F}\Phi_H(x)$
   since $c_{\rm EH,\tau_0}(H)>0$. Note that the flow $\phi^u$ of $\nabla \Phi_H$  has the form
  \begin{equation}\label{e:EH2.5}
  \phi^u(x)=e^{-u}x^-+x^0+e^ux^++\widetilde{K}(u,x),
  \end{equation}
  where $\widetilde{K}:\mathbb{R}\times \mathbb{E}\rightarrow \mathbb{E}$ is compact.
  Fix a set $F\in\mathcal{F}$, where $F=\gamma(S^+)$ for some $\gamma\in\Gamma$. Then $\alpha:=\inf(\Phi_H|\gamma(S^+))>0$ by the definition of $\mathcal{F}$. Let $\rho:\mathbb{R}\rightarrow [0,1]$ be a smooth function such that $\rho(s)=0$ for $s\le 0$ and $\rho(s)=1$ for $s\ge \alpha$. Define a vector field $V$ on $\mathbb{E}$ by
  $$
  V(x)=x^+-x^--\rho(\Phi_H(x))\nabla\mathfrak{b}_H(x).
  $$
  Since $0\le \rho(\Phi_H(x))\le 1$, the flow of $V$,  denoted by $\gamma^u$,
  has also the expression as in (\ref{e:EH2.5}).
  % has the same property as $\phi^u$.
   Clearly, for $x\in\mathbb{E}^-\oplus\mathbb{E}^0$, we have $\Phi_H(x)\le 0$ and
    thus  $V(x)=-x^-$. Hence $\gamma^u(\mathbb{E}^-\oplus \mathbb{E}^0)=\mathbb{E}^-\oplus \mathbb{E}^0$ and $\gamma^u(\mathbb{E}\setminus\mathbb{E}^-\oplus\mathbb{E}^0)=
    \mathbb{E}\setminus\mathbb{E}^-\oplus\mathbb{E}^0$ since $\gamma^u$ is a homeomorphism for each $u\in\mathbb{R}$. Note that $V|\Phi_H^{-1}([\alpha,\infty])=\nabla \Phi_H(x)$. Therefore $\gamma^u(F)=\phi^u(F)$ for $u\ge 0$. Moreover $\Gamma$ is closed
    for composition operation. It follows that $\mathcal{F}$ is positively invariant under the flow $\phi^u $ of $\nabla\Phi_H$. Using Proposition~\ref{prop:PSmale} we can prove Theorem~\ref{th:EH.1.6} by a standard minimax argument.
\qed
\vspace*{4pt}

\noindent{\it Proof of Proposition~\ref{prop:EH.1.8}}.\quad
By (\ref{e:EH.1.7}) we can assume that $B$ is  bounded, i.e., it is contained in ball $B^{2n}(K)$
of sufficiently large radius $K>0$.
Since the symplectic matrix $\Psi$ commutes with $\tau_0$,
 Lemma~\ref{lem:alb} tells us that
 $$
 \Psi=\left(
                                \begin{array}{cc}
                                  A & 0 \\
                                  0 & (A^T)^{-1} \\
                                \end{array}
                              \right)
$$
 for some $A\in GL(n,\mathbb{R})$. Since every symplectic matrix
can be uniquely decomposed as a product of an orthogonal  symplectic matrix and a positive definite
symplectic matrix, $\Psi=UP$, where
$$
P=\left(\begin{array}{cc}
 \sqrt{AA^T} & 0 \\
  0 & (\sqrt{(AA^T)})^{-1} \\
   \end{array}
 \right)\quad\hbox{and}\quad U=\left(\begin{array}{cc}
 Q& 0 \\
  0 & (Q^T)^{-1} \\
   \end{array}
 \right)
$$
with $Q=A(\sqrt{AA^T})^{-1}$, and thus commute with $\tau_0$. Moreover, $P=\exp(M)$, where $M\in\mathfrak{sp}(2n)$, i.e.,
it satisfies $M^TJ_0+J_0M=0$ ($\Leftrightarrow MJ_0+J_0M^T=0$). In fact, we have
$$
M=\left(\begin{array}{cc}
 N & 0 \\
  0 & -N \\
   \end{array}
 \right)\quad\hbox{where  $N\in \mathbb{R}^{n\times n}$ satisfies $N^T=N$ and $\exp N=\sqrt{AA^T}$}.
$$
Put $S=J_0M$, which is symmetric. Actually, $S=\left(\begin{array}{cc}
 0 & N \\
  N & 0 \\
   \end{array}
 \right)$ and thus
 $$
 P_t:=\exp(tM)=\exp(-J_0St)
 $$
is a Hamiltonian flow with Hamiltonian
$$
H_S(z)=-\frac{1}{2}\langle Sz, z\rangle_{\mathbb{R}^{2n}},\quad\forall z\in\mathbb{R}^{2n}.
$$
Fix $1/2>\epsilon>0$. Let $\|S\|:=\sup\{|Su|\,|\,u\in\mathbb{R}^{2n}\,\&\, |u|=1\}$.
Since $X_{H_S}(z)=-J_0Sz$,  for a positive number $K_1>3\|S\|K$ we have
 $$
 P_t(z)\in B^{2n}(K_1),\quad\forall z\in B^{2n}(K),\;\forall t\in [-2\epsilon, 1+2\epsilon].
 $$
 Choose a positive $K_2>K_1$ and a $\tau_0$-invariant smooth cut-off function $\rho:\mathbb{R}^{2n}\to [0,1]$
  such that $\sup_z|\nabla\rho(z)|\le 1$ and
 \begin{eqnarray*}
 \rho(z)=1\;\forall z\in B^{2n}(K_1),\qquad \rho(z)=0\;\forall z\in\mathbb{R}^{2n}\setminus B^{2n}(K_2).
 \end{eqnarray*}
 Define
 $$
 H^\rho_S(z)= -\frac{1}{2} \rho(z)\langle Sz, z\rangle_{\mathbb{R}^{2n}}\quad\forall z\in\mathbb{R}^{2n}.
 $$
 Then $X_{H^\rho_S}(z)=\rho(z)X_{H_S}(z)+ H_S(z)X_{\rho}(z)$  has a support in the closure of $B^{2n}(K_2)$,
  and     $|X_{H^\rho_S}(z)|=|X_{H_S}(z)|\le \|S\| |z|$ for all $z\in B^{2n}(K_1)$. Hence
  through any $z\in\mathbb{R}^{2n}$ the differential equation initial value problem
  $$
  \dot z(t)=X_{H^\rho_S}(z(t))\quad\hbox{and}\quad z(0)=z
  $$
  has a unique solution $(-\epsilon, 1+\epsilon)\ni t\mapsto z(t)\in\mathbb{R}^{2n}$.
    It follows that
  for each $t\in (-\epsilon, 1+\epsilon)$,
  $$
  \widetilde{P}_t:\mathbb{R}^{2n}\to\mathbb{R}^{2n},\;z=z(0)\mapsto z(t)
  $$
 is a Hamiltonian diffeomorphism with support in the closure of $B^{2n}(K_2)$ and satisfying
 $$
 \widetilde{P}_t(z)=P_t(z)\;\forall z\in B^{2n}(K)\quad\hbox{and}\quad \widetilde{P}_t\circ\tau_0=\tau_0\circ\widetilde{P}_t.
 $$
 %Write $\widetilde{P}:=\widetilde{P}_1$.  Then $\widetilde{P}\circ\tau_0=\tau_0\circ\widetilde{P}$, $\widetilde{P}(z)=P(z)$ for any $z\in B^{2n}(K)$, and $\widetilde{P}(z)=z$
%for all $z\in\mathbb{R}^{2n}\setminus B^{2n}(K_2)$.
%
%
%
%For $i\le j$ let $f_{ij}:\mathbb{R}\to\mathbb{R}$ be a smooth bijective map such that
%\begin{eqnarray*}
%&&f_{ij}(s)=r_{ij}s\quad\hbox{for $|s|\le K$},\\
%&&f_{ii}(s)=s\quad\hbox{for $|s|$ sufficiently large},\\
%&&f_{ij}(s)=0\quad\hbox{for $i<j$ and for $|s|$ sufficiently large},\\
%&&f'_{ii}(s)>0\quad\hbox{for all $s\in\mathbb{R}$}.
%\end{eqnarray*}
%Then we define $\widetilde{R}:\mathbb{R}^{2n}\to\mathbb{R}^{2n}$ by
%$$
%\widetilde{R}(x_1, x_2,\cdots,x_n)^T=\left(\sum^n_{i=1}f_{1i}(x_i),\sum^n_{i=2}f_{2i}(x_i),\cdots,f_{nn}(x_n)\right)^T.
%$$
%By the recurrence method we can directly check that $\widetilde{R}$ is
%a smooth bijective map. Moreover it is not hard to compute that
%$\widetilde{R}$ has the Jacobian $\prod^n_{i=1}f_{ii}'(x_i)>0$ at $(x_1,\cdots,x_n)\in\mathbb{R}^n$.
%Hence $\widetilde{R}$ is a diffeomorphism on $\mathbb{R}^n$.
Let $\widetilde{\Psi}_t=U\circ\widetilde{P}_t$  for each $t\in (-\epsilon, 1+\epsilon)$. It commutes with $\tau_0$ and  $\widetilde{\Psi}_t(z)=U(z)$
for all $z\in\mathbb{R}^{2n}\setminus B^{2n}(K_2)$ and $t$.
Write $\widetilde{\Psi}:=\widetilde{\Psi}_1$. Then
 $\widetilde{\Psi}(z)=\Psi(z)$ for any $z\in B^{2n}(K)$.

For any $H\in\mathscr{F}(\mathbb{R}^{2n},\tau_0,\Psi(B))$ which is  $\tau_0$-nonresonant, $H\circ\widetilde{\Psi}\in \mathscr{F}(\mathbb{R}^{2n},\tau_0,B)$ and is  $\tau_0$-nonresonant. Moreover $H\circ\widetilde{\Psi}_t\in \mathscr{F}(\mathbb{R}^{2n},\tau_0)$ and is also $\tau_0$-nonresonant for all $t$.
In fact, we only need to verify (H3) under Definition~\ref{def:deform}. By the assumptions we have
 $$
 H(z)=a|z|^2+ \langle z, z_0\rangle+ b\quad\forall z\in\mathbb{R}^{2n}\setminus B^{2n}(K_3),
 $$
  where $K_3>K_2$, $a>\pi$, $a\notin \mathbb{Z}\pi$, $z_0\in L_0$ and $b\in\mathbb{R}^{2n}$. It follows
  that
$$
H(\widetilde{\Psi}_t(z))=H(Uz)=a|Uz|^2+ \langle Uz, z_0\rangle+ b=a|z|^2+\langle z, U^Tz_0\rangle+ b
$$
for all $z\in\mathbb{R}^{2n}\setminus B^{2n}(K_3)$ and $t$. Clearly,  $U^Tz_0$ is also
a fixed point of $\tau_0$.

By (ii) of Proposition~\ref{prop:EH.1.2}, the map $t\to c_{\rm EH,\tau_0}(H\circ\widetilde{\Psi}_t)$ is continuous. Moreover, each
$c_{\rm EH,\tau_0}(H\circ\widetilde{\Psi}_t)$ is a critical value of $\Phi_{H\circ\widetilde{\Psi}_t}$ on
 the Hilbert space $\mathbb{E}$ by Theorem~\ref{th:EH.1.6}. Hence there exists a $C^2$ $1$-periodic map $\mathbb{R}\ni s\mapsto x_t(s)\in\mathbb{R}^{2n}$ such that
 $$
 \frac{d}{ds}x_t(s)=X_{H\circ\widetilde{\Psi}_t}(x_t(s)),\quad x_t(-s)=\tau_0x_t(s),\quad
 c_{\rm EH,\tau_0}(H\circ\widetilde{\Psi}_t)=\Phi_{H\circ\widetilde{\Psi}_t}(x_t).
 $$
  Clearly $\widetilde{\Psi}_t(x_t)$ is also a critical value of $\Phi_{H}$ on $\mathbb{E}$ and $\Phi_{H\circ\widetilde{\Psi}_t}(x_t)=\Phi_{H}(\widetilde{\Psi}_t(x_t))$. Hence for any $t\in [0,1]$,
  $c_{\rm EH,\tau_0}(H\circ\widetilde{\Psi}_t)$ sits in the critical value set of $\Phi_{H}$ which has empty interior.  These imply that $c_{\rm EH,\tau_0}(H\circ\widetilde{\Psi}_t)\equiv c_{\rm EH,\tau_0}(H\circ\widetilde{\Psi}_0)=c_{\rm EH,\tau_0}(H\circ U)$ for all $t$. In particular, we have
  \begin{equation}\label{e:EH.1.8.1}
  c_{\rm EH,\tau_0}(H\circ\widetilde{\Psi})=c_{\rm EH,\tau_0}(H\circ U).
  \end{equation}
   We claim that
  \begin{equation}\label{e:EH.1.8.2}
  c_{\rm EH,\tau_0}(H\circ U)=c_{\rm EH,\tau_0}(H).
  \end{equation}
  To prove this let us choose  a continuous path $Q_t$ ($0\le t\le 1$) in ${\rm O}(n,\mathbb{R})$ such that $Q_1=Q$ and $Q_0=I_n$ or ${\rm diag}(-1,1,\cdots,1)$ (depending on the determinant of $Q$). Define
  $$
  \Psi_{Q_t}=\left(
                \begin{array}{cc}
                  Q_t & 0 \\
                  0 & Q_t \\
                \end{array}
              \right).
              $$
  Clearly $U=\Psi_Q=\Psi_{Q_1}$.    Let $\Omega_t$ be the translation given by $z\mapsto z-\frac{1}{2a}\Psi_{Q_t}^Tz_0$.   Then for $z\in\mathbb{R}^{2n}$ with $|z|>K_3+\frac{|z_0|}{2a}$
  and each $t\in [0,1]$ there holds
  $$
  (H\circ\Psi_{Q_t}\circ \Omega_t)(z)=a|z|^2+b-\frac{|z_0|^2}{4a}.
  $$
  Using the same arguments as above and the proof of (i) of Proposition~\ref{prop:EH.1.7} we arrive at
   \begin{eqnarray*}
   c_{\rm EH,\tau_0}(H\circ U)&=&c_{\rm EH,\tau_0}(H\circ\Psi_{Q_1})=c_{\rm EH,\tau_0}(H\circ\Psi_{Q_1}\circ \Omega_1)\\
   &=&c_{\rm EH,\tau_0}(H\circ\Psi_{Q_0}\circ \Omega_0)=c_{\rm EH,\tau_0}(H\circ\Psi_{Q_0}).
   \end{eqnarray*}
   Hence (\ref{e:EH.1.8.2}) is proved if $\det Q>0$.
  Next we  show that
  \begin{equation}\label{e:EH.1.8.3}
   c_{\rm EH,\tau_0}(H)=c_{\rm EH,\tau_0}(H\circ \Psi_{Q_0})\quad
  \hbox{with}\quad Q_0={\rm diag}(-1,1,\cdots,1).
  \end{equation}
  Notice that $Q_0^2=I_n$ and so $(\Psi_{Q_0})^2=I_{2n}$. We can define a bijection
  $$
  \Lambda:\Gamma\to\Gamma,\;\gamma\mapsto\hat{\gamma},
  $$
  where  $\gamma(x)=a(x,1)x^++b(x,1)x^0+c(x,1)x^-+K(x,1)$ and
  \begin{eqnarray*}
 \hat{\gamma}(x)&=&\hat{a}(x,1)x^++\hat{b}(x,1)x^0+\hat{c}(x,1)x^-
  +\widehat{K}(x,1)\\
  &=&a(\Psi_{Q_0}x,1)x^++b(\Psi_{Q_0}x,1)x^0+c(\Psi_{Q_0}x,1)x^-+
  \Psi_{Q_0}K(\Psi_{Q_0}x,1)
  \end{eqnarray*}
  with $(\Psi_{Q_0}x)(t)=\sum _{k\in\mathbb{Z}}e^{2k\pi tJ_0}\Psi_{Q_0}x_k$ if $x\in \mathbb{E}$ has the expansion
  $$
  x(t)=\sum _{k\in\mathbb{Z}}e^{2k\pi tJ_0}x_k.
  $$
   Since $\Psi_{Q_0}$ commutes with $J_0$,
  $\Psi_{Q_0}x$ has the equivalent definition
  $$
  (\Psi_{Q_0}x)(t):=\Psi_{Q_0}x(t).
  $$
Hence $\Psi_{Q_0}x\in\gamma(S^+)$ if and only if $x\in \hat{\gamma}(S^+)$. It follows that
  \begin{eqnarray*}
  c_{\rm EH,\tau_0}(H\circ \Psi_{Q_0})&=&\sup_{\gamma\in\Gamma}\inf_{x\in \gamma(S^+)}\Phi_{H\circ \Psi_{Q_0}}(x)\\
  &=&\sup_{\gamma\in\Gamma}\inf_{x\in \gamma(S^+)}\Phi_{H}(\Psi_{Q_0}x)\\
  &=&\sup_{\gamma\in\Gamma}\inf_{\Psi_{Q_0}x\in \gamma(S^+)}\Phi_{H}(x)\\
  &=&\sup_{\hat{\gamma}\in\Gamma}\inf_{x\in \hat{\gamma}(S^+)}\Phi_{H}(x)\\
  &=& c_{\rm EH,\tau_0}(H).
  \end{eqnarray*}
(\ref{e:EH.1.8.3}) is proved, and so (\ref{e:EH.1.8.2}) holds.

  By (\ref{e:EH.1.8.1}) and (\ref{e:EH.1.8.2})  we obtain $c_{\rm EH,\tau_0}(H\circ\widetilde{\Psi})=c_{\rm EH,\tau_0}(H)$. It follows from this that $c_{\rm EH,\tau_0}(B)\le c_{\rm EH,\tau_0}(\Psi B)$.
  Since $\Psi^{-1}$ is also a symplectic matrix commuting with $\tau_0$, we get
   $c_{\rm EH,\tau_0}(B)\le c_{\rm EH,\tau_0}(\Psi^{-1}B)$.
   Replacing $B$ by $\Psi B$ in the latter inequality we get $c_{\rm EH,\tau_0}(\Psi B)\le c_{\rm EH,\tau_0}(B)$.
   Proposition~\ref{prop:EH.1.8} is proved.
 \qed

\section{Proof of Theorem~\ref{th:convex}}\label{sec:convex}
\setcounter{equation}{0}

By Lemma 2.29 in \cite{R12} there exists a linear real symplectic isomorphism $\Psi$
from $(\mathbb{R}^{2n},\omega_0,\tau)$ to $(\mathbb{R}^{2n},\omega_0,\tau_0)$.
Note that $D$ is $\tau$-invariant if and only if $\Psi(D)$  is $\tau_0$-invariant,
and that $x$ is a (generalized) $\tau$-brake closed characteristic on $\partial D$
if and only if $\Psi(x)$ a (generalized) $\tau_0$-brake closed characteristic on $\partial \Psi(D)=\Psi(\partial D)$.
 Moreover, it is clear that $A(\Psi(x))=A(x)$ by (\ref{e:action1}).
\textsf{Because of these and Proposition~\ref{MonComf}(ii) we can assume $\tau=\tau_0$ below.}

Note that $D$ always contains a fixed point $p$ of $\tau_0$. (Indeed, choose any $\hat{p}\in D$. Then $\tau_0\hat{p}\in D$ since $\tau_0D=D$. It follows that $p:=(\hat{p}+\tau_0\hat{p})/2\in D$ due to convexity of $D$ and there holds $\tau_0p=p$.)  Consider the symplectomorphism
\begin{equation}\label{e:5convex}
\psi:(\mathbb{R}^{2n},\omega_0)\to (\mathbb{R}^{2n},\omega_0),\;x\mapsto x-p.
\end{equation}
Since $\psi\circ \tau_0=\tau_0 \circ\psi$, we have
$c_{\rm HZ,\tau_0}(D,\omega_0)=c_{\rm HZ,\tau_0}(\psi(D),\omega_0)$. Moreover,
for a (generalized) $\tau_0$-brake characteristic $z:[0,T]\to \partial D$, it is easily checked
that  $y=\psi\circ z$ is  a (generalized) $\tau_0$-brake characteristic on $\partial\psi(D)=\psi(\partial D)$
and satisfies $A(y)=A(z)$.
Hence \textsf{from now on we can assume that our $D$ contains  $0$ and denote by $j_D$ the Minkowski functional of $D$ in this section.}

\begin{lemma}\label{nointerior}
 % Let $D\subset \mathbb{R}^{2n}$ be a compact  convex domain with
%boundary $\mathcal{S}=\partial D$ satisfying $N_0D=D$ and $0\in D$.
If $\mathcal{S}=\partial D$ is of class $C^{2n+2}$, then the set
   \begin{equation}\label{e:brakespec}
    \Sigma_{\mathcal{S}}^{\tau_0}:=\{ A(x)\,|\, x\;\hbox{is a }\;\hbox{$\tau_0$-brake closed characteristic on}
    \;\mathcal{S}\;\hbox{and}\;
   A(x)>0\}
    \end{equation}
   has no interior point  in $\mathbb{R}$.
 \end{lemma}

 This is a direct consequence of the well known fact that
  \begin{equation}\label{spec}
 \Sigma_{\mathcal{S}}:=\{ A(x)\,|\, x\;\hbox{is a }\;\hbox{closed characteristic on}\;\mathcal{S}\;\hbox{and}\;A(x)>0\}
 \end{equation}
  has no interior point in $\mathbb{R}$. (See \cite{Sik90}, and \cite[Lemma~4.1]{JinLu19} for a more general result).

\begin{lemma}\label{lem:genChar}
 For every generalized $\tau_0$-brake closed characteristic on $\mathcal{S}=\partial D$,
$x:\mathbb{R}/T\mathbb{Z}\to \mathbb{R}^{2n}$ (for some $T>0$),
which may be equivalently viewed as an absolutely continuous curve
$x:[0,T]\to \mathbb{R}^{2n}$ satisfying
\begin{equation}\label{e:repara0}
\left\{
   \begin{array}{l}
     -J_0\dot{x}(t)\in N_{\mathcal{S}}(x(t)), \;{\rm a.e.},\\
     x(T-t)=\tau_0x(t),\;x(0)=x(T),\;x([0,T])\subset\mathcal{S},
   \end{array}
   \right.
\end{equation}
we  can reparameterize it
as a $W^{1,\infty}$-map $x^\ast:[0, T']\to \mathbb{R}^{2n}$ satisfying
 \begin{equation}\label{e:repara}
\left\{
   \begin{array}{l}
     -J_0\dot{x}^\ast(t)\in\partial {H}(x^\ast(t)), \;{\rm a.e.},\\
     x^\ast(T'-t)=\tau_0x^\ast(t),\;x^\ast(0)=x^\ast(T'),\;x^\ast([0,T'])\subset\mathcal{S},
   \end{array}
   \right.
\end{equation}
where $H=j_D^2$. (Precisely, there is a differentiable homeomorphism
$\varphi:[0, T']\to [0, T]$ with an absolutely continuous inverse
$\psi:[0, T]\to [0, T']$ such that $x^\ast=x\circ\varphi$
is a $W^{1,\infty}$-map satisfying (\ref{e:repara}).)
Moreover,
\begin{eqnarray}\label{e:action3}
A(z)&=&\frac{1}{2}\int^{T}_0\langle-J_0\dot{x}(t), x(t)\rangle_{\mathbb{R}^{2n}} dt\nonumber\\
&=&\frac{1}{2}\int^{T'}_0\langle-J_0\dot{x}^\ast(t), x^\ast(t)\rangle_{\mathbb{R}^{2n}} dt=
A(x^\ast)=T'.
\end{eqnarray}
 \end{lemma}

 \begin{proof}
By Lemma~2 in \cite[Chap.V,\S1]{Ek90} (or its proof)
there exists $L^1$ map $\lambda:\mathbb{R}/T\mathbb{Z}\rightarrow(0,\infty)$ such that $-J_0\dot{x}(s)\in\lambda(s)\partial H(x(s))$ a.e. on $[0,T]$. Since it is obvious that $H(\tau_0z)=H(z)$, there holds $\partial H(\tau_0z)= \tau_0\partial H(z)$. Moreover, by definitions of $\tau_0$-brake closed characteristic , we have
$x(T-s)=\tau_0x(s)$. It follows
\begin{eqnarray*}
-\tau_0\dot{x}(T-s)=\dot{x}(s)&\in&\lambda(s)J_0\partial H(x(s))\\
&=& \lambda(s)J_0\partial H(\tau_0x(T-s))\\
&=&\lambda(s)J_0\tau_0\partial H(x(T-s))\\
&=&-\lambda(s)\tau_0J_0\partial H(x(T-s)),\;\forall s.
\end{eqnarray*}
On the other hand, we also have
$$
\dot{x}(T-s)\in\lambda(T-s) J_0\partial H(x(T-s)).
$$
Therefore we can assume the above $\lambda$ to satisfy $\lambda(s)=\lambda (T-s)$ in addition. Define  $\psi:[0,T]\rightarrow [0,T']$ by $\psi(s):=\int_0^s\lambda(u)du$, where  $T'=\int_0^T\lambda(s)ds$ . Clearly we have
\begin{eqnarray*}
T'-\psi(s)=\psi(T)-\psi(s)&=&\int_0^T\lambda(u)du-\int_0^s\lambda(u)du\\
&=&\int_0^T\lambda(u)du-\int_{T-s}^T\lambda(u)du\\
&=&\int_0^{T-s}\lambda(u)du\\
&=&\psi(T-s).
\end{eqnarray*}
Let $\varphi:[0,T']\rightarrow [0,T]$ be the inverse of $\psi$. It is differentiable and satisfies
$$
T-\varphi(t)=\varphi(T'-t).
$$
Then it is easy to check that $x^\ast:=x\circ\varphi$ satisfies (\ref{e:repara}) and (\ref{e:action3}). Arguing as in the proof of Lemma 4.2 in \cite{JinLu19} we deduce that there exists $r>0$ such that
\begin{equation}\label{e:homog3}
\partial H(z)\subset B^{2n}(r),\;\forall z\in\mathcal{S}
\end{equation}
and hence $x^\ast$ is a  $W^{1,\infty}$-map.
\end{proof}

\subsection{Proof of (\ref{e:action2})}\label{sec:convex1}

We shall complete the proof via the Clarke dual variational  principle in \cite{Cl79}
(see also \cite{MoZe05, HoZe94} (smooth case) and \cite{Ek90, AAO14} (nonsmooth case)
for detailed arguments).
%Let  $j_D: \mathbb{R}^{2n}\rightarrow\mathbb{R}$ be the Minkowski (or gauge) functional associated to $D$.
Notice that the Hamiltonian function  $H:\mathbb{R}^{2n}\to \mathbb{R}$ defined by $H(z)=j_D^2(z)$
is convex (and so continuous by \cite[Cor.10.1.1]{Roc70} or \cite[Prop.2.31]{Kr15}), and satisfies $H(\tau_0z)=H(z)$
 and $\mathcal{S}=\partial D=H^{-1}(1)$.
It is also $C^{1,1}$ with uniformly Lipschitz constant if $\mathcal{S}$
 is $C^{1,1}$ in $\mathbb{R}^{2n}$.
Moreover, there exists some constant $R_1\geq 1$ such that
\begin{equation}\label{e:dualestimate0}
 \frac{|z|^2}{R_1}\leq H(z)\leq R_1|z|^2\quad
\forall z\in\mathbb{R}^{2n}.
\end{equation}
This implies that the Legendre transformation of $H$ defined by
$$
H^{\ast}(w)=\max_{\xi\in \mathbb{R}^{2n}}(\langle w ,\xi \rangle_{\mathbb{R}^{2n}}-H(\xi))
$$
is a convex function from $\mathbb{R}^{2n}$ to $\mathbb{R}$ (and thus continuous).
From this and (\ref{e:dualestimate0}) it follows that there exists a constant $R_2\geq 1$ such that
\begin{equation}\label{e:dualestimate}
 \frac{|z|^2}{R_2}\leq H^{\ast}(z)\leq R_2|z|^2\quad
\forall z\in\mathbb{R}^{2n}.
\end{equation}
Note that $H^\ast$ is also $C^{1,1}$ in $\mathbb{R}^{2n}$
with uniformly Lipschitz constant if  $\mathcal{S}$
 is $C^{1,1}$ and strictly convex. See \cite[Cor.10.1.1]{Roc70}.

Consider the subspace of $W^{1,2}([0,1],\mathbb{R}^{2n})$,
\begin{equation}\label{e:constrant1}
\mathcal{F}=\{x\in W^{1,2}([0,1],\mathbb{R}^{2n})\,|\,x(1)=x(0),\;x(1-t)=\tau_0x(t),\,\int_0^1x(t)dt=0\},
\end{equation}
and its subset
\begin{equation}\label{e:constrant2}
\mathcal{A} =\{x\in \mathcal{F}\,|\,A(x)=1 \}.
\end{equation}
\begin{remark}\label{re:mean}
  {\rm Note that $\mathcal{F}\ne\emptyset$ and so $\mathcal{A}\ne\emptyset$.
  In fact, fix some $x\in W^{1,2}([0,1],\mathbb{R}^{2n})$ such that $x(1-t)=\tau_0x(t)$
  and $x(1)=x(0)$. Let $c:=\int_0^1 x(t)dt$. Then
  $$
  \tau_0c=\int_0^1\tau_0x(t)dt=\int_0^1 x(1-t)dt=\int_{0}^1x(t)dt=c.
  $$
  Let $y(t):=x(t)-c$. Then $y\in W^{1,2}([0,1],\mathbb{R}^{2n})$  satisfies $y(1)=y(0)$ and
  $$
  \int_0^1y(t)dt=0,\quad y(1-t)=x(1-t)-c=\tau_0(x(t)-c)=\tau_0y(t), \quad A(y)=A(x).
  $$}
\end{remark}
Also notice that $\mathcal{A}$ is a regular submanifold of $\mathcal{F}$. In fact, for any $x\in \mathcal{F}$ and $\zeta\in T_x\mathcal{F}=\mathcal{F}$,
$$
dA(x)[\zeta]= \int_0^{1}\langle -J_0\dot{\zeta},x\rangle_{\mathbb{R}^{2n}}+\frac{1}{2}\langle -J_0x,\zeta\rangle_{\mathbb{R}^{2n}}|_0^1=\int_0^{1}\langle -J_0\dot{\zeta},x\rangle_{\mathbb{R}^{2n}}
$$
since $x(1)=x(0)$ and $\zeta(1)=\zeta(0)$. Thus $dA\neq 0$ on $\mathcal{A}$ because
$$
dA(x)[x]=\int_0^{1}\langle -J_0\dot{x}, x\rangle_{\mathbb{R}^{2n}}=2,\quad\forall x\in \mathcal{A}\subset A^{-1}(1).
$$

\vspace*{4pt}\noindent{\bf Step 1.} {The functional $I:\mathcal{F}\to\mathbb{R}$ defined by
$$
I(x)=\int_0^1H^{\ast}(-J_0\dot{x})
$$
has the positive infimum  on $\mathcal{A}$, $\mu:=\inf _{x\in\mathcal{A}}I(x)$.}
For any  $x\in\mathcal{F}$, since
 $\int_0^1 x(t)dt=0$   we have
\begin{equation}\label{poincare}
\|x\|_{L^2}\le \|\dot{x}\|_{L^2}.
\end{equation}
Moreover,  if $x\in\mathcal{A}$, then there holds
\begin{equation}\label{actioncons}
2=2A(x)\leq \|x\|_{L^2}\|\dot{x}\|_{L^2}\leq \|\dot{x}\|_{L^2}^2.
\end{equation}
It follows from these and (\ref{e:dualestimate}) that for any $x\in\mathcal{A}$,
\begin{equation}\label{relation}
I(x)=\int_0^1 H^{\ast}(-J_0\dot{x})\geq \frac{1}{R_2}\|\dot{x}\|_{L^2}^2\geq \frac{2}{R_2}.
\end{equation}

\noindent\noindent{\bf Step 2.} {There exists $u\in  \mathcal{A}$ such  that
$I(u)=\mu$.} %In particular, $u$ is a critical point of  $I|_{\mathcal{A}}$.}
Let $(x_n)\subset\mathcal{A}$ be a sequence satisfying
 $$
 \lim_{n\rightarrow+\infty}I(x_n)=\mu.
 $$
  By (\ref{relation}), for some $C_2>0$ we have
 $$
 2\le\|\dot{x}_n\|_{L^2}^2\leq R_2I(x_n)\leq C_2.
 $$
 It follows from  (\ref{poincare}) and (\ref{actioncons}) that
$$
\frac{2}{C_2}\leq \frac{2}{\|\dot{x}_n\|_{L^2}^2}\leq\|x_n\|_{L^2}^2\leq \|\dot{x}_n\|_{L^2}^2\leq  C_2.
$$
Hence $(x_n)$ is a bounded sequence in $W^{1,2}(\mathbb{R}/\mathbb{Z},\mathbb{R}^{2n})$.
After passing to a subsequence
if necessary, we can assume that  $(x_n)$ converges weakly to some $u$ in  $W^{1,2}(\mathbb{R}/\mathbb{Z}, \mathbb{R}^{2n})$.
 By Arzel\'{a}-Ascoli theorem, there also exists $\hat{u}\in C^{0}(\mathbb{R}/\mathbb{Z},\mathbb{R}^{2n})$ such that
$$
\lim_{n\rightarrow+\infty}\sup_{t\in [0,1]}|x_n(t)-\hat{u}(t)|=0.
$$
A standard argument gives that $\hat{u}(t)=u(t)$ almost everywhere.
As usual we can assume $\hat{u}=u$. Then  we get that
$$
u(1-t)=\tau_0u(t),\quad u(1)=u(0),\quad \int_0^1 u(t)dt=0
$$
and thus $u\in\mathcal{A}$ because
\begin{eqnarray*}
A(u)&=&\frac{1}{2}\int_0^1 \langle J_0u, \dot{u}\rangle_{\mathbb{R}^{2n}}\\
&=&\lim_{n\rightarrow+\infty}\frac{1}{2}\int_0^1 \langle J_0u, \dot{x}_{n}\rangle_{\mathbb{R}^{2n}}\\
&=&\lim_{n\rightarrow+\infty} \frac{1}{2}\int_0^1 (\langle J_0x_n, \dot{x}_{n}\rangle_{\mathbb{R}^{2n}}+\langle J_0(u-x_n), \dot{x}_{n}\rangle_{\mathbb{R}^{2n}})\\
&=&1.
\end{eqnarray*}
Consider the functional
$$
\hat{I}: L^2([0,1],\mathbb{R}^{2n})\to\mathbb{R},\;u\mapsto\int^1_0H^\ast(u(t))dt.
$$
Then $I(x)=\hat{I}(-J_0\dot{x})$ for any $x\in\mathcal{F}$.
Since $H^{\ast}$ is  convex, so is $\hat{I}$.
(\ref{e:dualestimate}) also implies that $\hat{I}$ is continuous and thus
has nonempty subdifferential $\partial\hat{I}(v)$ at each point $v\in L^2([0,1],\mathbb{R}^{2n})$.
%Hence
%\begin{equation}\label{e:convex}
%\hat{I}(v_2)-\hat{I}(v_1)\le (w, v_2-v_1)_{L^2}\quad\forall w\in \partial\hat{I}(v_2)\subset L^2([0,1],\mathbb{R}^{2n})
%\end{equation}
%for any $v_2,v_1\in L^2([0,1],\mathbb{R}^{2n})$.
By Corollary~3 in \cite[Chap. II, \S3]{Ek90} we know
$$
\partial\hat{I}(v)=\{w\in L^2([0,1],\mathbb{R}^{2n})\,|\, w(t)\in \partial H^\ast(v(t))\;\hbox{a.e. on}\;[0,1] \}.
$$
%and so (\ref{e:convex}) becomes
%$$
%\hat{I}(v_2)-\hat{I}(v_1)\le \int^1_0\langle w(t), v_2(t)-v_1(t)\rangle_{\mathbb{R}^{2n}} dy,\quad \forall w\in \partial\hat{I}(v_2).
%$$
It follows that
\begin{eqnarray}\label{mini}
I(u)-I(x_n)&=&\hat{I}(-J_0\dot{u})-\hat{I}(-J_0\dot{x}_n)\nonumber\\
&\leq& \int_0^1 \langle w(t),-J_0( \dot{u}(t)-\dot{x}_n(t))\rangle_{\mathbb{R}^{2n}} dt
\end{eqnarray}
for any $w\in \partial\hat{I}(-J_0\dot{u})=\{w\in L^2([0,1],\mathbb{R}^{2n})\,|\, w(t)\in \partial H^\ast(-J_0\dot{u}(t))\;\hbox{a.e. on}\;[0,1]\}$. Recall  that $(x_n)$ converges weakly to some $u$ in
$W^{1,2}(\mathbb{R}/\mathbb{Z}, \mathbb{R}^{2n})$. This implies that
$(\dot{x}_n)$ converges weakly to some $\dot{u}$ in  $L^2([0,1], \mathbb{R}^{2n})$.
It follows  that the left hand of (\ref{mini}) converges to $0$. Therefore
 $$
\mu\le I(u)\le\lim_{n\rightarrow+\infty}I(x_n)=\mu.
$$
The desired claim is proved.

\vspace*{4pt}\noindent{\bf Step 3.} {There exists a generalized $\tau_0$-brake closed characteristic on $\mathcal{S}$,
$\hat{x}:\mathbb{R}/\mu\mathbb{Z}\rightarrow \mathcal{S}$,  such that $A(\hat{x})=\mu$.} Since $\mu$ is the minimum  of  $I|_{\mathcal{A}}$,
  applying Lagrangian multiplier theorem (cf. \cite[Theorem~6.1.1]{Cl83}) we get some $\lambda\in\mathbb{R}$ such that
$0\in\partial (I+\lambda A)(u)=\partial I(u)+\lambda A'(u)$. Define
$$
\Lambda_{\mathcal{F}}:\mathcal{F}\rightarrow L^2([0,1]),\,x\mapsto -J_0\dot{x}.
$$
Clearly $R(\Lambda_{\mathcal{F}})$ is closed in $L^2([0,1])$ because $\mathcal{F}$ is closed in $W^{1,2}([0,1],\mathbb{R}^{2n})$. Using
$I=\hat{I}\circ \Lambda_{\mathcal{F}}$ and  Corollary~6 in \cite[Chap. II, \S2]{Ek90}
we arrive at
\begin{eqnarray*}
\partial I(u)&=& (\Lambda_{\mathcal{F}})^\ast\partial \hat{I}(\Lambda_{\mathcal{F}}u)\\
&=&
\{(\Lambda_{\mathcal{F}})^\ast w\,|\,
w\in L^2([0,1],\mathbb{R}^{2n}),\; w(t)\in \partial H^\ast(-J_0\dot{u}(t))\;\hbox{a.e. on}\;[0,1]\}
\end{eqnarray*}
Hence there exists a function $w\in L^2([0,1],\mathbb{R}^{2n})$ with $w(t)\in \partial H^\ast(-J_0\dot{u}(t))$ a.e. on $[0,1]$, such that $(\Lambda_{\mathcal{F}})^\ast w+\lambda A'(u)=0$, i.e.,
\begin{eqnarray}\label{e:elag}
0&=&\langle (\Lambda_{\mathcal{F}})^\ast w+\lambda A'(u),\zeta\rangle_{L^2}\nonumber\\
&=&\langle w, \Lambda_{\mathcal{F}}\zeta\rangle_{L^2}+\langle\lambda A'(u),\zeta\rangle_{L^2}\nonumber\\
&=&\int_0^1 \langle w(t),-J_0\dot{\zeta}(t)\rangle_{\mathbb{R}^{2n}}+\lambda
\int_0^1 \langle  u(t),-J_0\dot{\zeta}(t)\rangle_{\mathbb{R}^{2n}},\quad \forall\zeta\in \mathcal{F}.
\end{eqnarray}
This $w$ can be chosen to satisfy $w(1-t)=\tau_0 w(t)$ in addition. [In fact, since $u(1-t)=\tau_0u(t)$ and $\zeta(1-t)=\tau_0\zeta (t)$, (\ref{e:elag})
implies that
\begin{eqnarray*}
&&\int_0^1 \langle \tau_0w(1-t)+\lambda u(t),-J_0\dot{\zeta}(t)\rangle_{\mathbb{R}^{2n}} dt\\
&=&\int_0^1 \langle \tau_0w(1-t)+\lambda \tau_0u(1-t),-\tau_0J_0\dot{\zeta}(1-t)\rangle_{\mathbb{R}^{2n}} dt\\
&=&\int_0^1 \langle w(1-t)+\lambda u(1-t),-J_0\dot{\zeta}(1-t)\rangle_{\mathbb{R}^{2n}} dt\\
&=&\int_0^1 \langle w(t)+\lambda u(t),-J_0\dot{\zeta}(t)\rangle_{\mathbb{R}^{2n}} dt\\
&=&0,\quad\forall \zeta\in\mathcal{F}.
\end{eqnarray*}
From this and (\ref{e:elag}) it follows that  $\hat{w}(t):=(w(t)+\tau_0w(1-t))/2$ satisfies
$$
\int_0^1 \langle \hat{w}(t)+\lambda u(t),-J_0\dot{\zeta}(t)\rangle_{\mathbb{R}^{2n}} dt= 0,\quad\forall \zeta\in\mathcal{F}.
$$
 Thus the desired $w$ can be obtained.]
%Since $w(t)\in \partial H^\ast(-J_0\dot{u}(t))$, there holds
%$$
%\tau_0 w(1-t)\in \tau_0\partial H^\ast(-J_0\dot{u}(1-t))=\partial H^\ast(-\tau_0J_0\dot{u}(1-t))=\partial H^\ast(-J_0\dot{u}(t)).
%$$
% It follows that $\hat{w}(t)\in \partial H(-J_0\dot{u})$ a.e. on $[0,1]$. Moreover
%$\hat{w}(1-t)=\tau_0\hat{w}(t)$ and satisfies
%$$
%\int_0^1 \langle \hat{w}(t)+\lambda u(t),-J_0\dot{\zeta}(t)\rangle_{\mathbb{R}^{2n}} dt
%=0,\;\forall \zeta\in\mathcal{F}.
%$$
This implies that for some $a_0\in L_0$,
\begin{equation}\label{mini.1}
{w}(t)+\lambda u(t)=a_0\quad\hbox{a.e. on}\quad [0,1].
\end{equation}
%In the following we denote $\hat{w}$ by $w$ for simplicity.
Since
\begin{eqnarray}\label{mini.2}
\langle w, -J_0\dot{u}\rangle_{L^2}&=&\int_0^1 \langle w(t),-J_0\dot{u}(t)\rangle_{\mathbb{R}^{2n}} dt\nonumber\\
&=&\int_0^1\langle a_0-\lambda u(t),-J_0\dot{u}(t)\rangle_{\mathbb{R}^{2n}}=-2\lambda,
\end{eqnarray}
and the convex functional $\hat{I}$ is a $2$-positively homogeneous, using
the Euler formula (cf. \cite[Theorem~3.1]{YangWei08}) we may obtain
$$
\langle w, -J_0\dot{u}\rangle_{L^2}=2\hat{I}(-J_0\dot{u})=2I(u)=2\mu
$$
and thus $-\lambda=\mu$ by (\ref{mini.2}).

%Let $\hat{I}^\ast$ be the Fenchel conjugate of $\hat{I}$.
%Then Theorem~2 in \cite[Chap. II,\S3]{Ek90} yields
%$$
%\hat{I}^\ast(x)=\int^1_0H(x(t))dt.
%$$
%Using Corollary~3 in \cite[Chap. II,\S3]{Ek90} again we get
%$$
%\partial\hat{I}^\ast(x)=\{y\in L^2([0,1],\mathbb{R}^{2n})\,|\, y(t)\in \partial H(x(t))\;a.e.\}.
%$$
By (\ref{mini.1}), $a_0-\lambda u(t)=w(t)\in\partial H^\ast(-J_0\dot{u}(t))$ a.e. on $[0,1]$.
So
 \begin{equation}\label{mini.3}
-J_0\dot{u}(t)\in \partial H(w(t))= \partial H(-\lambda u(t)+a_0)\quad\hbox{a.e. on}\quad [0,1]
\end{equation}
and hence
\begin{equation}\label{mini.4}
v:[0,\mu]\rightarrow \mathbb{R}^{2n},\;
t\mapsto v(t):=-\lambda u(-t/\lambda)+a_0=\mu u(t/\mu)+a_0
\end{equation}
satisfies
 \begin{equation}\label{mini.5}
-J_0\dot{v}(t)\in  \partial H(v(t))\quad\hbox{a.e.}
\end{equation}
and
$$
v(\mu)=v(0),\quad v(\mu-t)=\tau_0v(t).
$$
The Legendre reciprocity formula (cf. \cite[Proposition~II.1.15]{Ek90}) in convex analysis yields
\begin{eqnarray}\label{e:vconstant}
\int_0^\mu H(v(t))dt&=&-\int_0^\mu H^\ast(-J_0\dot{v}(t))dt+\int_0^\mu\langle v(t),-J_0\dot{v}(t)\rangle_{\mathbb{R}^{2n}} dt\nonumber\\
&=&-\mu\int_0^1H^\ast(-J_0\dot{u}(s))ds+\mu^2\int_0^1\langle-J_0\dot{u}(s),u(s)\rangle_{\mathbb{R}^{2n}}
ds\nonumber\\
&=&-\mu^2+2\mu^2=\mu^2.
\end{eqnarray}
By \cite[Theorem~2]{Ku96}, (\ref{mini.5})
implies that $H(v(t))$ is constant and hence by (\ref{e:vconstant}) $H(v(t))\equiv\mu$ for all $t\in [0, \mu]$.
The latter shows that  $v$ is nonconstant.
It follows that
 \begin{equation}\label{e:desiredChar}
x^\ast:[0,\mu]\rightarrow \mathcal{S},\; t\mapsto\frac{v(t)}{\sqrt{\mu}}=\sqrt{\mu}u(t/\mu)+ a_0/\sqrt{\mu}
\end{equation}
is a $\tau_0$-brake closed characteristic on $\mathcal{S}$ with action $A(x^\ast)=\mu$.
Moreover, (\ref{mini.5}) and the $2$-homogeneity of $H$ imply $-J_0\dot{x}^\ast(t)\in  \partial H(x^\ast(t))$ a.e. on $[0,1]$.
The following claim shows that $x^\ast$ satisfies (\ref{e:action2}).

\vspace*{4pt}\noindent\noindent{\bf Step 4.} {For any generalized $\tau_0$-brake closed characteristic
on $\mathcal{S}$ with positive action, $y:[0,T]\rightarrow \mathcal{S}$, there holds $A(y)\ge \mu$.
} By Lemma~\ref{lem:genChar}, by reparameterizing it we can assume that
$y\in W^{1,2}([0,T],\mathbb{R}^{2n})$ and satisfies (\ref{e:repara})
with $T'$ being replaced by $T$. Then
\begin{equation}\label{e:desiredChar1}
A(y)=T\quad\hbox{and}\quad H(y(t))\equiv 1.
\end{equation}
 Define
$$
y^{\ast}:[0,1]\rightarrow \mathbb{R}^{2n},\;   t\mapsto y^{\ast}(t)=a y(tT)+b,
$$
 where $a:=1/\sqrt{T}$ and $b:=-1/\sqrt{T}\int_0^1 y(tT)dt\in L_0$. Then $y^{\ast}\in\mathcal{A}$ and
$$
-J_0\dot{y}^\ast(t)=-aTJ_0\dot{y}(tT)\in aT\partial H({y}(Tt))=
\partial H(aT{y}(Tt))\quad\hbox{a.e. on}\quad [0,1]
$$
and thus  $aT{y}(Tt)\in\partial H^\ast(-J_0\dot{y}^\ast(t))$ a.e. on $[0,1]$.
The Legendre reciprocity formula (cf. \cite[Proposition~II.1.15]{Ek90}) in convex analysis yields
\begin{eqnarray*}
H(aT{y}(Tt))+ H^\ast(-J_0\dot{y}^\ast(t))&=&\langle-J_0\dot{y}^\ast(t), aT{y}(Tt)\rangle_{\mathbb{R}^{2n}}\\
&=&\langle-aTJ_0\dot{y}(Tt), aT{y}(Tt)\rangle_{\mathbb{R}^{2n}}\\
&=&(aT)^2\langle-J_0\dot{y}(Tt), {y}(Tt)\rangle_{\mathbb{R}^{2n}}\\
&=&2(aT)^2H({y}(Tt))\\
&=&2(aT)^2=2T,\quad\hbox{a.e. on}\quad [0,1],
\end{eqnarray*}
where the fourth equality comes from
the Euler formula \cite[Theorem~3.1]{YangWei08}.
Moreover, (\ref{e:desiredChar1}) implies $H(aT{y}(Tt))=(aT)^2=T$. Then
$H^\ast(-J_0\dot{y}^\ast(t))=T$ a.e. on $[0,1]$ and so
 \begin{eqnarray*}
\int_0^1H^{\ast}(-J_0\dot{y}^{\ast}(t))dt=T.
\end{eqnarray*}
By the definition of $\mu$, we have $T\ge \mu$, and so $A(y)\ge\mu$ by (\ref{e:desiredChar1}).

\subsection{Proof of (\ref{e:action-capacity1}) for smooth and strictly convex $D$}\label{sec:convex3}

The proof is similar to that of  \cite[Propposition~4]{HoZe90}.
For the sake of completeness we prove it in details.

\vspace*{4pt}\noindent{\bf Step 1.} {Prove}
 \begin{equation}\label{e:action-capacity5}
   c_{\rm HZ,\tau_0}(D,\omega_0)\ge A(x^{\ast}).
   \end{equation}
For small $0<\epsilon, \delta<1/2$, pick  a smooth
 function $f: [0,1]\rightarrow\mathbb{R}$ such that
   \begin{eqnarray*}
     &f(t)=0,& \,\,\,   t\le\delta, \\
     & f(t)=A(x^{\ast})-\varepsilon  ,& \,\,\, 1-\delta\le t, \\
    & 0\leq f'(t)<A(x^{\ast}),& \,\,\, \delta<t<1-\delta.
   \end{eqnarray*}
Define $H(x)=f(j_D^2(x))$ for $x\in D$. Then
   $H\in \mathcal{H}(D,\omega_0,\tau_0)$ because $j_D$ is $C^\infty$
   in $\mathbb{R}^{2n}\setminus\{0\}$ and satisfies $j_D(\tau_0z)=j_D(z)$. Let us prove that every
   $x:\mathbb{R}/T\mathbb{Z}\to D$  satisfying
   \begin{equation}\label{e:action-capacity6}
        \dot{x}=J_0\nabla H(x)=f'(j^2_D(x))J_0\nabla j^2_D(x)\quad\hbox{and}\quad
     x(-t)=\tau_0 x(t)
     \end{equation}
with $0<T\le 1$ is constant. By contradiction we assume that $x=x(t)$
is nonconstant. Then $j_D(x(t))$ is equal to a nonzero constant and thus $x(t)\ne 0$
for each $t$. Moreover, $f'(j^2_D(x(t)))\equiv a\in (0, A(x^\ast))$.
Since $\nabla j^2_D(\lambda z)=\lambda\nabla j^2_D(z)$ for all $(\lambda,z)\in\mathbb{R}_+\times\mathbb{R}^{2n}$, multiplying $x(t)$
by a suitable positive number we can assume that $x(\mathbb{R}/T\mathbb{Z})\subset\mathcal{S}=\partial D$
and
\begin{equation}\label{e:action-capacity7}
        \dot{x}=aJ_0\nabla j^2_D(x)\quad\hbox{and}\quad
     x(-t)=\tau_0 x(t).
     \end{equation}
     Note that $\langle \nabla j^2_D(z), z\rangle_{\mathbb{R}^{2n}}=2j^2_D(z)=2$ for any $z\in\mathcal{S}$.
We deduce
from (\ref{e:action-capacity7}) that
$$
A(x)=aT\le a<A(x^\ast),
$$
which contradicts (\ref{e:action2}).
This shows that $H\in \mathcal{H}_{ad}(D,\omega_0,\tau_0)$ is admissible and hence
$$
c_{\rm HZ,\tau_0}(D,\omega_0)\ge m(H)=A(x^{\ast})-\epsilon.
$$
Letting $\epsilon\to 0$ we get (\ref{e:action-capacity5}).

\vspace*{4pt}\noindent{\bf Step 2.} {Prove}
 \begin{equation}\label{e:action-capacity8}
   c_{\rm HZ,\tau_0}(D,\omega_0,\tau_0)\le A(x^{\ast}).
   \end{equation}
Let $H\in \mathcal{H}(D,\omega_0,\tau_0)$ satisfy $m(H)>A(x^\ast)$.
We wish to prove that the system
   \begin{equation}\label{e:action-capacity9}
        \dot{x}=J_0\nabla H(x),\quad x(t+1)=x(t)\quad\hbox{and}\quad
     x(-t)=\tau_0x(t)
     \end{equation}
has a nonconstant solution $x:\mathbb{R}/\mathbb{Z}\to D$.
By Lemma~\ref{nointerior} we have a small number $\epsilon>0$ such that $m(H)>A(x^{\ast})+\epsilon$ and  $A(x^\ast)+\epsilon\notin \Sigma_{\mathcal{S}}^{\tau_0}$. This means that the following system
   \begin{equation}\label{e:action-capacity10}
        \dot{x}=(A(x^\ast)+\epsilon)J_0\nabla j^2_D(x),\quad x(1+t)=x(t) \quad\hbox{and}\quad x(-t)=\tau_0x(t)
     \end{equation}
admits only the trivial solution $x\equiv 0$. (Otherwise, we have
$x(t)\ne 0\;\forall t$ as above. Thus after
 multiplying $x(t)$ by a suitable positive number we can assume that
  $x(t)\in \mathcal{S}=\partial D$, which leads to
  $\Sigma_{\mathcal{S}}^{\tau_0}\ni A(x)=A(x^\ast)+\epsilon$.)
For a fixed number $\delta>0$ we take a smooth
 function $f: [1, \infty)\rightarrow\mathbb{R}$ such that
   \begin{eqnarray*}
     &f(t)\ge (A(x^\ast)+\epsilon)t,& \,\,\,   t\ge 1, \\
     &f(t)=(A(x^\ast)+\epsilon)t,& \,\,\,   t\;\hbox{large}, \\
     & f(t)=m(H),& \,\,\, 1\le t\le 1+\delta, \\
    & 0\leq f'(t)\le A(x^{\ast})+\epsilon,& \,\,\, t>1+\delta.
   \end{eqnarray*}
With this $f$ we get an extension of $H$ as follows
$$
 \overline{H}(z)=\left\{
   \begin{array}{l}
     H(z),\quad\hbox{for}\;z\in D, \\
     f(j_D^2(z)), \quad\hbox{for}\;z\notin D.
   \end{array}
   \right.
$$
Notice that $\overline{H}\in C^\infty(\mathbb{R}^{2n},\mathbb{R})$ satisfies $\overline{H}(\tau_0z)=\overline{H}(z)$.
By (\ref{e:Phi}),
$$
\Phi_{\overline{H}}(x)=
\mathfrak{a}(x)-\mathfrak{b}_{\overline{H}}(x)
=1/2(\|x^+\|_{1/2}^2-\|x^-\|_{1/2}^2)-\int^1_0\overline{H}(x(t))dt.
$$

\begin{lemma}\label{lem:positive}
   Assume that $x$ is a $\tau_0$-brake orbit of $X_{\overline{H}}$ with
   period $1$ and with  $\Phi_{\overline{H}}(x)>0$.
   Then it is nonconstant, sits in $D$ completely, and thus
   is a brake orbit of $X_H$ on $D$.
\end{lemma}

\begin{proof}
   Since $\overline{H}\geq 0$
 and $\Phi_{\overline{H}}(x)>0$,  $x$ cannot be constant.
 Moreover, by contradiction, suppose that $x(\bar{t})\notin D$ for some $\bar{t}$.
Then $x(t)\notin D$ for all $t$. By the construction of $\overline{H}$, the Hamiltonian equation $\dot{x}=X_{\overline{H}}(x)$  becomes
  $$
  \dot{x}=J_0f'(j_D^2(x))\nabla j_D^2(x).
  $$
  This implies that $j_D(x(t))$ is constant. A direct computation leads to
  \begin{eqnarray*}
  \Phi_{\overline{H}}(x)&=&\int_0^1\left\{\frac{1}{2}\langle f'(j_D^2(x))\nabla j_D^2(x),x\rangle_{\mathbb{R}^{2n}}-f(j_D^2(x))\right\}\\
  &=&f'(j_D^2(x))j_D^2(x)-f(j_D^2(x))\\
  &\leq &\left(A(x^\ast)+\varepsilon\right)j_D^2(x)- \left(A(x^\ast)+\varepsilon\right)j_D^2(x)\\
  &=&0,
  \end{eqnarray*}
  which contradicts the assumption $\Phi_{\overline{H}}(x)>0$.
    Hence $x(t)\in D$ for all $t$, and thus $x$ is a nonconstant $\tau_0$-brake orbit of $X_H$.
\end{proof}

\begin{lemma}\label{lem:PS}
   If  a sequence $(x_k)\subset \mathbb{E}$ such that $\nabla \Phi_{\overline{H}}(x_k)\rightarrow 0$ in
   $\mathbb{E}$, then it
   has a convergent subsequence in $\mathbb{E}$.
\end{lemma}

\begin{proof}
If $(x_k)$ is bounded in $\mathbb{E}$, as in the proof \cite[page 89, Lemma~6]{HoZe94} we deduce that
$(x_k)$ has a convergent subsequence.
If $(x_k)$ is unbounded in $\mathbb{E}$,
we can assume
   $\lim_{k\rightarrow +\infty }\|x_k\|_{1/2}=+\infty$ without loss of generality.
   Let $y_k=\frac{x_k}{\|x_k\|_{1/2}}$. Then $\|y_k\|_{1/2}=1$ and satisfies
   $$
   y_k^+-y_k^--j^{\ast}\left(\frac{ \nabla \overline{H}(x_k)}{\|x_k\|_{1/2}}\right)\rightarrow 0.
   $$
By the construction of $\overline{H}$ and the fact that
$\nabla j_D^2(\lambda z)=\lambda \nabla j_D^2(z)$ for $\lambda>0$,
 arguing as in the proof \cite[page 89, Lemma~6]{HoZe94}
 we  can assume that  $y_{k}\rightarrow  y$ in $\mathbb{E}$.  Moreover, since
 there exists a constant $C>0$ such that $|(j_D^2)''(z)u|\le C|u|$  for all $z,u\in\mathbb{R}^{2n}$,
  we get that
  $$
  j^\ast\left(\frac{\nabla\overline{H}(x_k)}{\|x_k\|_{1/2}}\right)\to (A(x^{\ast})+\epsilon)
  j^{\ast}\nabla j_D^2(y)\quad\hbox{in $\mathbb{E}$}.
  $$
  By Proposition~\ref{solution}, $y$ satisfies
  $$
  \dot{y}=(A(x^\ast)+\epsilon)J_0\nabla j_D^2(y),\quad y(1+t)=y(t)\quad\hbox{and}\quad
  y(-t)=\tau_0y(t)
  $$
  and thus $y=0$ because
   $A(x^\ast)+\epsilon\notin \Sigma_{\mathcal{S}}^{\tau_0}$.
   This contradicts to the fact $\|y\|_{1/2}=1$. That is,  $(x_k)$ must be bounded in $\mathbb{E}$.
\end{proof}

For $x^\ast$ in (\ref{e:desiredChar}) define
$x_0(t)=x^\ast(\mu t)=\sqrt{\mu}u(t)+ a_0/\sqrt{\mu}$ for $t\in [0, 1]$.
Then $x_0$ satisfies
\begin{equation}\label{e:action-capacity12}
\left\{
   \begin{array}{l}
     \dot{x_0}=A(x^\ast)J_0\nabla j_D^2(x_0),\\
      x_0(t+1)=x_0(t),\; x_0(-t)=\tau_0x(t)\\
       A(x_0)=A(x^\ast),\\
    j_D(x_0(t))\equiv 1,\;\hbox{i.e.,}\; x_0([0,1])\subset\mathcal{S}.
   \end{array}
   \right.
\end{equation}
Denote by $x_0^+$ the projections of $x_0$ onto $\mathbb{E}^+$.
 Then $x_0^+\ne 0$. (Otherwise,
 a contradiction occurs because $0<A(x^\ast)=A(x_0)=-\frac{1}{2}\|x_0^-\|^2_{1/2}$.)
Following \cite{HoZe90}
 we define for $s>0$ and $r>0$,
 \begin{eqnarray*}
 &&W_s:=\mathbb{E}^-\oplus \mathbb{E}^0\oplus sx_0^+ \subset\mathbb{E},\\
 &&\Sigma_r:=\{x^-+x^0+sx_0^+\,|\,0\le s\le r,\;\|x^-+x^0\|_{1/2}\le r\}.
 \end{eqnarray*}
Let $\partial\Sigma_r$ denote the boundary of $\Sigma_r$ in $\mathbb{E}^0\oplus \mathbb{E}^-\oplus\mathbb{R}x_0^+$. Then
 \begin{equation}\label{boundary}
  \partial\Sigma_r=\{x=x^-+x^0 +sx_0^+\in\Sigma_r\,|\,\|x^-+x^0\|_{1/2}=r \;\text{or}\; s=0 \;\text{or} \; s=r \}.
  \end{equation}
Lemmas~5,~6 in \cite{HoZe90} directly lead to the following two results.
\begin{lemma}\label{lem:HZ5}
  There exists a constant $C>0$ such that for any $s\ge 0$,
  $$
  \Phi_{\overline{H}}(x)\le-\epsilon\int^1_0(j_D(x(t)))^2dt+C,\quad\forall x\in W_s.
  $$
\end{lemma}
\begin{lemma}\label{lem:HZ6}
    $\Phi_{\overline{H}}|\partial\Sigma_r\le 0$ if $r>0$ is sufficiently large.
\end{lemma}
Arguing as in the proof of Proposition~\ref{prop:EH.1.4}, we get
 \begin{lemma}\label{lem:HZ7}
 For $z_0\in L_0\cap H^{-1}(0)$,
there exist constants $\alpha>0$ and $\beta>0$ such that
  $$
  \Phi_{\overline{H}}|\Gamma_\alpha\geq\beta>0,
  $$
  where $\Gamma_\alpha=\{z_0+x\,|\,x\in \mathbb{E}^+,\,\|x\|_{1/2}=\alpha\}$.
  \end{lemma}
Let $\phi^t$ be the negative gradient flow of $\Phi_{\overline{H}}$. Arguing as in Section~\ref{section:space}, $\phi^t$ has the property described in Proposition~\ref{prop:flow}.
\begin{lemma}\label{positive+}
$\phi^t(\Sigma_r)\cap\Gamma_\alpha\neq \emptyset, \,\forall t\geq 0$.
 \end{lemma}
\begin{proof}
Observe that $\phi^t(\Sigma_r)\cap\Gamma_\alpha\neq \emptyset$
if and only if $t\cdot x=\phi^t(x)\in\Gamma_\alpha$ for some $x\in\Sigma_r$, that is,
\begin{equation}\label{e:3.16}
\left\{
\begin{array}{l}
  x\in\Sigma_r \\
  (P^-+P^0)(t\cdot x-z_0)=0,\\
  \|t\cdot x-z_0\|_{1/2}=\alpha.
\end{array}
\right.
\end{equation}
By Proposition~\ref{prop:flow}
$$
(P^-+P^0)(t\cdot x-z_0)=e^tx^-+P^-K(t,x)+x^0+P^0K(t,x)-z_0.
$$
Hence (\ref{e:3.16}) is equivalent to
\begin{equation*}
\left\{
\begin{array}{l}
  x\in\Sigma_r \\
  x^-+e^{-t}P^-K(t,x)+x^0+P^0K(t,x)-z_0=0,\\
  (\|t\cdot x-z_0\|_{1/2}-\alpha)x_0^+=0.
\end{array}
\right.
\end{equation*}
 Define $B(t,\cdot):\mathbb{E}^-\oplus \mathbb{E}^0\oplus \mathbb{R}x_0^+\rightarrow \mathbb{E}^-\oplus \mathbb{E}^0\oplus \mathbb{R}x_0^+$ by
 $$
 B(t,x)= e^{-t}P^-K(t,x) +P^0K(t,x)-z_0
 +(\|t\cdot x-z_0\|_{1/2}-\alpha)x_0^+-P^+x.
 $$
  Then (\ref{e:3.16}) is equivalent to
 \begin{equation}\label{e:3.17}
 \left\{
\begin{array}{l}
  x\in\Sigma_r \\
  x+B(t,x)=0.
\end{array}
\right.
 \end{equation}
Since $K$ in Proposition~\ref{prop:flow} is continuous and maps bounded sets into precompact sets,
each $B(t,\cdot)$ is also such a map.
Note that for the constant  $\alpha$ in Lemma \ref{lem:HZ7} we have
   $$
   0\notin (id +B(t,\cdot))(\partial\Sigma_r), \,\forall t\geq 0
   $$
  if  $r$ in Lemma \ref{lem:HZ6} is sufficiently large. From now on, we fix a sufficiently large $r>\alpha$.
   Then $\deg(\Sigma_r, id+B(t,\cdot),0)=\deg(\Sigma_r, id+B(0,\cdot),0)$. Since $K(0, x)=0$, we have
   $$
   B(0,x)=-z_0+P^{+}\{(\|x-z_0\|_{1/2}-\alpha)x_0^{+}-x\}.
   $$
    Define the homotopy
   $$
   L_{\mu}(x)=-z_0+P^{+}\{(\mu\|x-z_0\|_{1/2}-\alpha)x_0^{+}-\mu x\} \quad\text{for}\quad 0\leq \mu\leq 1.
   $$
   Since $L_\mu$ maps $\mathbb{E}^-\oplus \mathbb{E}^+\oplus\mathbb{R}x_0^+$ into a finite dimensional space, $L_\mu$
   maps bounded sets into precompact sets for every $0\le \mu\le 1$.
  We conclude that $x+L_\mu(x)\neq 0$ for $x\in\partial\Sigma_r$ and $0\leq \mu\leq 1$.
  Otherwise, suppose that $x+L_\mu x=0$ for some $\mu\in [0, 1]$ and some
  $x=x^-+x^0 +sx_0^+\in \partial\Sigma_r$.  Then
  \begin{equation}\label{e:3.18}
 -z_0+P^{+}\{(\mu\|x-z_0\|_{1/2}-\alpha)x_0^{+}-\mu x\}=-x.
 \end{equation}
  It follows that $x^-=0$ and $x^0=z_0$. (Note that we can choose sufficiently large $r$ such that $\tau>|z_0|$). Thus (\ref{e:3.18}) becomes
  $(\mu s\|x_0^+\|_{1/2}-\alpha)x_0^{+}-\mu sx_0^+=-sx_0^+$, that is,
  \begin{equation}\label{e:3.19}
\mu s\|x_0^+\|_{1/2}-\alpha=\mu s-s.
 \end{equation}
  Moreover, by (\ref{boundary}) we have either $s=0$ or $s=\tau$.
  Because $\alpha>0$,  (\ref{e:3.19}) implies $s\ne 0$. Hence we get
  $\mu r \|x_0^+\|_{1/2}-\alpha=\mu r-r$, i.e.,
  $$
  r=\frac{\alpha}{\mu\|x_0^+\|_{1/2}+1-\mu}\le \frac{\alpha}{\min\{\|x_0^+\|_{1/2},1\}}
  $$
  because $0<\min\{\|x_0^+\|_{1/2},1\}\le\mu\|x_0^+\|_{1/2}+1-\mu\le 1$ for $0\le\mu\le 1$.
 Thus as sufficiently large $r$ satisfies $r>\alpha/\min\{1,\|x_0^+\|_{1/2}\}$,
 we get a contradiction.

    It follows from the homotopy invariance of degree that
   \begin{eqnarray*}
    \deg(\Sigma_r,id +B(t,\cdot),0)
   &=&\deg(\Sigma_r,id +B(0,\cdot),0)\\
   &=&\deg(\Sigma_r,id+L_0,0)\\
   &=&\deg(\Sigma_r,id-z_0-\alpha x_0^{+},0)=1.
      \end{eqnarray*}
     This implies that (\ref{e:3.17}) and so (\ref{e:3.16}) has solutions.
  \end{proof}
Let $\mathcal{F}=\{\phi^t(\Sigma_r)\,|\, t\geq 0\}$ and define
$$
c(\Phi_{\overline{H}}, \mathcal{F}):=\inf_{t\geq 0}\sup_{x\in \phi^t(\Sigma_r)}\Phi_{\overline{H}}(x).
$$
Lemmas~\ref{lem:HZ7},~\ref{positive+} imply
 $$
 0<\beta\leq \inf _{x\in\Gamma_\alpha}\Phi_{\overline{H}}(x)\leq \sup_{x\in \phi^t(\Sigma_r)}\Phi_{\overline{H}}(x),\quad
\forall t\geq 0,
$$
 and hence $c(\Phi_{\overline{H}}, \mathcal{F})\geq\beta>0$.
On the other hand,  since $\Sigma_r$ is bounded and Proposition~\ref{prop:Lip} implies that $\Phi_{\overline{H}}$ maps  bounded sets into bounded sets we arrive at
$$
c(\Phi_{\overline{H}}, \mathcal{F})\leq \sup_{x\in\Sigma_r}\Phi_{\overline{H}}(x)<\infty.
$$
 Using the Minimax Lemma on \cite[page 79]{HoZe94}, we get a critical point $x$ of $\Phi_{\overline{H}}$ with $\Phi_{\overline{H}}(x)>0$.

Now Lemma~\ref{lem:positive},~\ref{solution} together yield the desired result.

\subsection{Completing the proof of Theorem~\ref{th:convex} for general case}\label{sec:convex4}

By modifying parts of the proof of Proposition~1.12 in \cite{Kr15} we get the following result.

\begin{proposition}\label{prop:approximation2}
  Let $K\subset\mathbb{R}^{2n}$ be a $\tau_0$-invariant bounded convex domain containing $0$,
  and let $U$ be an open neighborhood of $K$. Then there exists a bounded $\tau_0$-invariant
 and strictly convex body $K'$ with smooth boundary such that
   $K\subset K'\subset U$.
\end{proposition}

\begin{proof}
  For simplicity let $p=j_K$ the Minkowski functional of $K$.
  Then there holds $p(\tau_0z)=p(z)$. Let $\chi$ be a nonnegative $C_0^{\infty}$
  function supported in $B^{2n}(1)$ such that $\chi (\tau_0z)=\chi(z)$ and $\int_{\mathbb{R}^{2n}} \chi=1$.
  For $\epsilon>0$ and $z\in\mathbb{R}^{2n}$, define
  $$
  p_\epsilon(z)=\int_{\mathbb{R}^{2n}} p(z-\epsilon u)\chi (u)du+\epsilon |z|^2.
  $$
  Then it is easy to check that $p_\epsilon$ is smooth and strictly convex. Moreover,
 \begin{eqnarray*}
  p_\epsilon(\tau_0z)&=&\int_{\mathbb{R}^{2n}} p(\tau_0z-\epsilon u)\chi (u)du+\epsilon |\tau_0z|^2\\
  &=&\int_{\mathbb{R}^{2n}} p(z-\epsilon \tau_0u)\chi (\tau_0u)du+\epsilon |z|^2\\
  &=&\int_{\mathbb{R}^{2n}} p(z-\epsilon u)\chi (u)du+\epsilon |z|^2\\
  &=&p_\epsilon(z).
  \end{eqnarray*}
  In addition,
  $$
  p_\epsilon(z)\le \int_{\mathbb{R}^{2n}} (p(z)+\epsilon p(-u))\chi(u)du+\epsilon|z|^2=p(z)+\epsilon\left(\int_{\mathbb{R}^{2n}} p(-u)\chi(u)du+|z|^2\right).
  $$
  Let $C_1=\int_{\mathbb{R}^{2n}} p(-u)\chi(u)du+\sup_{z\in K}|z|^2 $. Then $C_1$ is a positive constant such that
  $$
  p_\epsilon(z)\le p(z)+\epsilon C_1=1+\epsilon C_1,\quad\forall z\in K .
  $$
  Clearly there exists a positive number $\delta>0$ so that
  $z-\epsilon u\notin (1+\delta)K$
  for every $z\notin U$ and for all $u\in B^{2n}(1)$ if the above $\epsilon$
  is chosen sufficiently small. We can also require $\epsilon<\delta/C_1$.
  It follows that
  $$
  p_\epsilon(z)>1+\delta+\epsilon |z|^2\quad\hbox{for}\quad z\notin U.
  $$
  Let $C_2:=\min_{z\notin U}|z|$. Then $C_2$ is a positive number such that
  $$
  p_\epsilon(z)>1+\delta+\epsilon C_2\quad\hbox{for}\quad z\notin U.
  $$
  Define
  $$
  K'=\{z\in\mathbb{R}^{2n}\,|\,p_\epsilon(z)\le 1+2C_1\epsilon\}.
  $$
  Since $1+\delta+\epsilon C_2>1+2\epsilon C_1$ , we have
  $$
  K\subset K'\subset U.
  $$
  It is obvious that $K'$ is strictly convex with smooth boundary and satisfies $\tau_0K'=K'$.
\end{proof}

By Proposition~\ref{prop:approximation2},
we can choose two sequences of $C^\infty$ strictly convex domains with smooth boundaries,
 $(D^+_k)$ and $(D^-_k)$, such that
 \begin{enumerate}
 %\begin{description}
\vspace*{3pt} \item[\bf (i)] $\tau_0D^+_k=D^+_k$ and $\tau_0D^-_k=D^-_k$ for each $k$.
\vspace*{3pt} \item[\bf (ii)] $D^-_1\subset D^-_2\subset\cdots\subset D$ and $\lim_{k\to\infty}d_{\rm Hausd}(D^-_k, D)=0$,
\vspace*{3pt} \item[\bf (iii)] $D^+_1\supseteq D^+_2\supseteq\cdots\supseteq D$ and $\lim_{k\to\infty}d_{\rm Hausd}(D^+_k,D)=0$.
 %\end{description}
 \end{enumerate}
% where $d$ is the Hausdorff distance defined in (\ref{e:Hausdis}).
 Denote by $j_D, j_{D^+_k}$ and $j_{D^-_k}$ the Minkowski  functionals of
 $D, D^+_k$ and $D^-_k$, respectively. Let $H=j_D^2, H^+_k=j_{D^+_k}^2$ and $H^{-}_k=j_{D^-_k}^2$
for each $k\in\mathbb{N}$. Their  Legendre transformations are
$H^\ast, H_k^{+\ast}$ and $H_k^{-\ast}$, $k=1,2,\cdots$. Denote by
$$
I(u)=\int^1_0H^\ast(-J\dot{u}),\quad I^+_k(u)=\int^1_0H_k^{+\ast}(-J\dot{u}),\quad I^-_k(u)=\int^1_0H_k^{-\ast}(-J\dot{u})
$$
for $u\in\mathcal{A}$, $k=1,2,\cdots$. Note that (ii) and (iii) imply
 %\begin{description}
 \begin{enumerate}
\vspace*{3pt} \item[\bf (iv)] $j_{D^-_1}\ge j_{D^-_2}\ge\cdots\ge j_D$ and so $H_1^{-\ast}\le H_2^{-\ast}\le\cdots\le H^\ast$,
\vspace*{3pt} \item[\bf (v)] $j_{D^+_1}\le j_{D^+_2}\le\cdots\le j_D$ and so $H_1^{+\ast}\ge H_2^{+\ast}\ge\cdots\ge H^\ast$.
 %\end{description}
 \end{enumerate}
These lead to
\begin{equation}\label{e:action-capacity13}
I^+_1(u)\ge I^+_2(u)\ge\cdots\ge I(u)\ge\cdots\ge I^-_2(u)\ge I^-_1(u),\quad\forall u\in  \mathcal{A}.
\end{equation}
By the first three steps in Section~\ref{sec:convex1} these functional attain their
minimums on $\mathcal{A}$. It easily follows from (\ref{e:action-capacity13}) that
\begin{equation}\label{e:action-capacity14}
\min_{\mathcal{A}}I^+_1\ge \min_{\mathcal{A}}I^+_2\ge\cdots\ge \min_{\mathcal{A}}I\ge\cdots\ge
\min_{\mathcal{A}}I^-_2\ge \min_{\mathcal{A}}I^-_1.
\end{equation}
Now (\ref{e:action2}) gives rise to
\begin{equation}\label{e:action-capacity15}
\min_{\mathcal{A}}I=\min\{A(x)>0\,|\,x\;\text{is a generalized $\tau_0$-brake closed characteristic on}\;\mathcal{S}\},
\end{equation}
and results in Section~\ref{sec:convex3} yield
\begin{eqnarray}\label{e:action-capacity16}
{c}_{\rm HZ,\tau_0}(D^+_k,\omega_0)=\min_{\mathcal{A}}I^+_k\quad\hbox{and}\quad
{c}_{\rm HZ,\tau_0}(D^-_k,\omega_0)=\min_{\mathcal{A}}I^-_k
\end{eqnarray}
for each $k\in\mathbb{N}$. By this, (\ref{e:action-capacity14})
and the monotonicity of ${c}_{\rm HZ,\tau_0}$ we get
$$
\begin{array}{ccccc}
 {c}_{\rm HZ,\tau_0}(D^+_k,\omega_0)& \ge& {c}_{\rm HZ,\tau_0}(D,\omega_0)&\ge & {c}_{\rm HZ,\tau_0}(D^-_k,\omega_0)\\[2mm]
 \parallel&  &&  &\parallel\\[2mm]
\min_{\mathcal{A}}I^+_k& \ge&  \min_{\mathcal{A}}I&\ge &\min_{\mathcal{A}}I^-_k
\end{array}
$$
Moreover $\lim_{k\to\infty}{c}_{\rm HZ,\tau_0}(D^+_k,\omega_0)=
 {c}_{\rm HZ,\tau_0}(D,\omega_0)$ and $\lim_{k\to\infty}{c}_{\rm HZ,\tau_0}(D^-_k,\omega_0)
={c}_{\rm HZ,\tau_0}(D,\omega_0)$ by Proposition~\ref{MonComf}(iii).
The squeeze theorem leads to
$$
{c}_{\rm HZ,\tau_0}(D,\omega_0)=\min_{\mathcal{A}}I.
$$
The desired result follows from this and (\ref{e:action-capacity15}).

\section{Proofs of Theorems~\ref{th:EHconvex}, \ref{th:EHproduct}}\label{sec:EH.3}
\setcounter{equation}{0}

Our proofs  closely follow those of Theorems~6.5, 6.6 in \cite{Sik90}.

\subsection{Proof of Theorem~\ref{th:EHconvex}}\label{sec:EH.3.1}

In the arguments  at the beginning of Section~\ref{sec:convex},
we can use the definition (\ref{e:EH.1.8}), instead of Proposition~\ref{MonComf}(ii),
to boil down to the case $\tau=\tau_0$.

%Proof of Theorem~\ref{th:convex}}

We first assume that $\partial D$ is smooth and strictly convex.
  As in the arguments at the beginning of Section~\ref{sec:convex}
 we can also assume $0\in D$ below since  $c_{\rm EH,\tau_0}(D)=c_{\rm EH,\tau_0}(\psi(D))$
 under the translation (\ref{e:5convex}) by Proposition~\ref{prop:EH.1.7} (i).

Let  $j_D: \mathbb{R}^{2n}\rightarrow\mathbb{R} $ be the Minkowski functional of $D$. Consider the Hamiltonian function  $H(z)=j^2_D(z)$ and its Legendre transformation $H^{\ast}$. Let $I$, $\mathcal{F}$ and $\mathcal{A}$
be as in the proof of (\ref{e:action2}) in Section \ref{sec:convex1}. Then there exists $w\in\mathcal{A}$  such that
$$
a:=\min\{I(u)\,|\,u\in\mathcal{A}\}=I(w)=A(x^{\ast})\quad\hbox{and}\quad A(w)=1.
$$
 Denote by $w^\ast$
 the projections of $w$ onto $\mathbb{E}^\ast$ (according to the decomposition $\mathbb{E}
=\mathbb{E}^+\oplus \mathbb{E}^-\oplus \mathbb{E}^0$), $\ast=0,-,+$.  Then $w^+\ne 0$.
 (Otherwise, a contradiction occurs because
$1=A(w) = A(w^0\oplus w^-) =-\frac{1}{2}\|w^-\|^2_{1/2}$.)
Put $y=w/\sqrt{a}$ so that
$$
 I(y)=1\quad\hbox{and}\quad A(y)=\frac{1}{a}.
$$
Now for any $\lambda\in\mathbb{R}$ and $x\in \mathbb{E}$ it holds that
\begin{eqnarray*}
\lambda^2=I(\lambda y)&=&\int^1_0H^\ast(-\lambda J_0\dot{y}(t))dt\\
&=&\int^1_0\sup_{\zeta\in\mathbb{R}^{2n}}\{\langle\zeta, -\lambda J_0\dot{y}(t)\rangle_{\mathbb{R}^{2n}}- H(\zeta)\}dt\\
&\ge&\int^1_0\{\langle x(t), -\lambda J_0\dot{y}(t)\rangle_{\mathbb{R}^{2n}}- H(x(t))\}dt.
\end{eqnarray*}
This leads to
\begin{eqnarray*}
\int^1_0H(x(t))dt&\ge& \int^1_0\langle x(t), -\lambda J_0\dot{y}(t)\rangle_{\mathbb{R}^{2n}} dt-\lambda^2\\
&=&\lambda \int^1_0\langle x(t), - J_0\dot{y}(t)\rangle_{\mathbb{R}^{2n}} dt-\lambda^2.
\end{eqnarray*}
Taking
$$
\lambda=\frac{1}{2} \int^1_0\langle x(t), -J_0\dot{y}(t)\rangle_{\mathbb{R}^{2n}} dt
$$
we arrive at
\begin{eqnarray}\label{e:EH.3.1}
\int^1_0H(x(t))dt\ge\left(\frac{1}{2} \int^1_0\langle x(t), - J_0\dot{y}(t)\rangle_{\mathbb{R}^{2n}} dt\right)^2\quad
\forall x\in \mathbb{E}.
\end{eqnarray}
Note that $y^+\ne 0$ and $\mathbb{E}^-\oplus \mathbb{E}^0\oplus\mathbb{R}_{>0}y=\mathbb{E}^-\oplus \mathbb{E}^0\oplus\mathbb{R}_{>0}y^+$.
From (ii) in Proposition~\ref{prop:EH.1.1} we derive the following result.

\begin{proposition}\label{prop:EH2}
For any $h\in\Gamma$ it holds that
$$
h(S^+)\cap(\mathbb{E}^-\oplus \mathbb{E}^0\oplus\mathbb{R}_{>0}y)\ne\emptyset.
$$
\end{proposition}
Fix an $h\in\Gamma$. Let $x\in h(S^+)\cap(\mathbb{E}^-\oplus \mathbb{E}^0\oplus\mathbb{R}_{>0}y)$.
Consider the polynomial
$$
P(t)=\mathfrak{a}(x+ty)=\mathfrak{a}(x)+ t\int^1_0\langle x, - J_0\dot{y}\rangle_{\mathbb{R}^{2n}} + \mathfrak{a}(y)t^2.
$$
  Writing $x=x^{-0}+ sy=x^{-0}+ sy^{-0}+ sy^+$, then
$$
P(t)=\mathfrak{a}(x^{-0}+ (t+s)y)
$$
and by $\mathfrak{a}|\mathbb{E}^-\oplus\mathbb{E}^0\le 0$
we deduce that $P(-s)\le 0$. Moreover,by $\mathfrak{a}(y)=1/a>0$ we get
$$
P(t)\to+\infty\quad\hbox{as}\quad |t|\to+\infty.
$$
These imply that there exists $t_0\in\mathbb{R}$ such that $P(t_0)=0$. It follows that
$$
\left(\int^1_0\langle x, - J_0\dot{y}\rangle_{\mathbb{R}^{2n}}\right)^2-4\mathfrak{a}(y)\mathfrak{a}(x)\ge 0
$$
and so
\begin{eqnarray}\label{e:EH.3.2}
\mathfrak{a}(x)&\le& (\mathfrak{a}(y))^{-1}\left(\frac{1}{2}\int^1_0\langle x, - J_0\dot{y}\rangle_{\mathbb{R}^{2n}}\right)^2\nonumber\\
&=&a\left(\frac{1}{2}\int^1_0\langle x, - J_0\dot{y}\rangle_{\mathbb{R}^{2n}}\right)^2\nonumber\\
&\le&a\int^1_0H(x(t))dt
\end{eqnarray}
by (\ref{e:EH.3.1}). Let $\Sigma^{\tau_0}_{\mathcal S}$ be as in (\ref{e:EHcontact+}).
 For $\epsilon>0$, let
\begin{eqnarray}\label{e:EH.3.7}
\mathscr{E}_\epsilon(\mathbb{R}^{2n},\tau_0,D)
\end{eqnarray}
consist of $\overline{H}=f\circ H$, where $f\in C^\infty(\mathbb{R},\mathbb{R})$ satisfies
\begin{eqnarray}\label{e:EH.3.8}
f(s)=0\;\forall s\le 1,\quad  f'(s)\ge 0,\;\forall\;s\ge 1,\quad f'(s)=\alpha\in\mathbb{R}\setminus\Sigma^{\tau_0}_{\mathcal S}
\;\hbox{if}\;f(s)\ge\epsilon,
\end{eqnarray}
and $\alpha$ is required to satisfy
$$
\alpha H(z)\ge \pi|z|^2-C,\quad\hbox{for}\, |z|\,\hbox{sufficiently large},
$$
where $C>0$ is a constant.
Arguing as in the Lemma~\ref{lem:PS}, we get
\begin{lemma}\label{lem:EH8}
For $\overline{H}\in\mathscr{E}_\epsilon(\mathbb{R}^{2n},\tau_0,D)$, $\Phi_{\overline{H}}$ satisfies the
$(PS)$ condition.
\end{lemma}
By the same method used in proving Theorem~\ref{th:EH.1.6}, we get
\begin{corollary}\label{cor:EH9}
For $\overline{H}\in\mathscr{E}_\epsilon(\mathbb{R}^{2n},\tau_0,D)$, $c_{\rm EH,\tau_0}(\overline{H})$ is a positive critical value of $\Phi_{\overline{H}}$.
\end{corollary}
\begin{lemma}\label{lem:EH10}
For $\overline{H}\in\mathscr{E}_\epsilon(\mathbb{R}^{2n},\tau_0,D)$, any positive critical value $c$ of
$\Phi_{\overline{H}}$ satisfies
 $$
 c>\min\Sigma^{\tau_0}_{\mathcal S}-\epsilon.
 $$
 In particular, $c_{\rm EH,\tau_0}(\overline{H})>\min\Sigma^{\tau_0}_{\mathcal S}-\epsilon$.
\end{lemma}
\begin{proof}
    Let $x\in \mathbb{E}$ be a critical point of $\Phi_{\overline{H}}$ with $\Phi_{\overline{H}}(x)>0$. Clearly there holds $\overline{H}(\tau_0z)=\overline{H}(z)$. Then by Proposition~\ref{solution} we have
\begin{eqnarray*}
&&-J_0\dot{x}(t)=\nabla \overline{H}(x(t))=f'(H(x(t)))\nabla H(x(t)), \\
&&\quad x(t+1)=x(t)\quad\hbox{and}\quad x(-t)=\tau_0 x(t),
\end{eqnarray*}
and $H(x(t))\equiv s_0$ (a constant). It follows that
\begin{eqnarray*}
\Phi_{\overline{H}}(x)&=&\frac{1}{2}\int^1_0\langle J_0x(t),\dot{x}(t)\rangle_{\mathbb{R}^{2n}} dt-\int^1_0\overline{H}(x(t))dt\\
&=&\frac{1}{2}\int^1_0\langle x(t),f'(s_0)\nabla H(x(t))\rangle_{\mathbb{R}^{2n}} dt-\int^1_0f(s_0)dt\\
&=&f'(s_0)s_0-f(s_0).
\end{eqnarray*}
Since $\Phi_{\overline{H}}(x)>0$, we get $\beta:=f'(s_0)>0$, and so $s_0>1$. Put
$$
y(t)=\frac{1}{\sqrt{s_0}}x(t/\beta).
$$
Then $y$ satisfies
$$
H(y(t))=1,\quad -J_0\dot{y}=\nabla H(y(t)),\quad y(\beta+t)=y(t)\quad\hbox{and}\quad y(-t)=\tau_0y(t).
$$
These show that $f'(s_0)=\beta=A(y)\in\Sigma^{\tau_0}_{\mathcal S}$.
Therefore $f(s_0)<\epsilon$ by definition of $\overline{H}$. It follows from these that
$$
\Phi_{\overline{H}}(x)=f'(s_0)s_0-f(s_0)> f'(s_0)-\epsilon\ge \min\Sigma^{\tau_0}_{\mathcal S}-\epsilon.
$$
\end{proof}
Since for any $\epsilon>0$ and $G\in \mathscr{F}(\mathbb{R}^{2n},\tau_0,D)$, there exists $\overline{H}\in\mathscr{E}_\epsilon(\mathbb{R}^{2n},\tau_0,D)$ such that $\overline{H}\ge G$.
 It follows that $c_{\rm EH,\tau_0}(G)\ge c_{\rm EH,\tau_0}(\overline{H})\ge\min\Sigma^{\tau_0}_{\mathcal S}-\epsilon$. Hence
$$
c_{\rm EH,\tau_0}(D)\ge \min\Sigma^{\tau_0}_{\mathcal S}=a.
$$
\begin{lemma}\label{lem:EH11}
$c_{\rm EH,\tau_0}(D)\le a$.
\end{lemma}

\begin{proof}
It suffices to prove: for any $\varepsilon>0$, there exists $\widetilde{H}\in\mathscr{F}(\mathbb{R}^{2n},\tau_0,D)$
such that
\begin{eqnarray}\label{e:EH.3.9}
c_{\rm EH,\tau_0}(\widetilde{H})< a+\varepsilon.
\end{eqnarray}
This may be reduced to prove: $\forall h\in\Gamma$, there exists $x\in h(S^+)$ such that
\begin{eqnarray}\label{e:EH.3.10}
\Phi_{\widetilde{H}}(x)< a+\varepsilon.
\end{eqnarray}

For $\nu>0$, there exists $H_\nu\in \mathscr{F}(\mathbb{R}^{2n},\tau_0, D)$
such that
\begin{eqnarray}\label{e:EH.3.11}
H_\nu\ge \nu\left(H-\left(1+\frac{\varepsilon}{2a}\right)\right).
\end{eqnarray}
For $h\in\Gamma$ choose $x\in h(S^+)$ satisfying (\ref{e:EH.3.2}). We shall prove that for $\nu>0$
large enough $\widetilde{H}=H_\nu$ satisfies the requirements.

\vspace*{4pt}\noindent{\bf Case I.} If $\int^1_0H(x(t))dt\le\left(1+\frac{\varepsilon}{a}\right)$, then by $H_\nu\ge 0$ and (\ref{e:EH.3.2}), we have
$$
\Phi_{H_\nu}(x)\le \mathfrak{a}(x)\le a\int^1_0H(x(t))dt\le a\left(1+\frac{\varepsilon}{a}\right)<a+\varepsilon.
$$

\vspace*{4pt}\noindent{\bf Case II.} If $\int^1_0H(x(t))dt>\left(1+\frac{\varepsilon}{a}\right)$, then (\ref{e:EH.3.11}) implies
\begin{eqnarray}\label{e:EH.3.12}
\int^1_0H_\nu(x(t))dt&\ge& \nu\left(\int^1_0H(x(t))dt-\left(1+\frac{\varepsilon}{2a}\right)\right)\nonumber\\
&\ge&\nu \frac{\varepsilon}{2a}\left(1+\frac{\varepsilon}{a}\right)^{-1}\int^1_0H(x(t))dt
\end{eqnarray}
because
$$
\left(1+\frac{\varepsilon}{2a}\right)=
\left(1+\frac{\varepsilon}{2a}\right)\left(1+\frac{\varepsilon}{a}\right)^{-1}
\left(1+\frac{\varepsilon}{a}\right)<
\left(1+\frac{\varepsilon}{2a}\right)\left(1+\frac{\varepsilon}{a}\right)^{-1}\int^1_0H(x(t))dt
$$
and
$$
1-\left(1+\frac{\varepsilon}{2a}\right)\left(1+\frac{\varepsilon}{a}\right)^{-1}=
\left(1+\frac{\varepsilon}{a}\right)^{-1}\left[\left(1+\frac{\varepsilon}{a}\right)-
\left(1+\frac{\varepsilon}{2a}\right)\right]=
\frac{\varepsilon}{2a}\left(1+\frac{\varepsilon}{a}\right)^{-1}.
$$
Let us choose $\nu>0$ so large that
$$
\nu\frac{\varepsilon}{2a}\left(1+\frac{\varepsilon}{a}\right)^{-1}>a.
$$
Then (\ref{e:EH.3.12}) leads to
$$
\int^1_0H_\nu(x(t))dt\ge a\int^1_0H(x(t))dt
$$
and hence
$$
\Phi_{H_\nu}(x)=\mathfrak{a}(x)-\int^1_0H_\nu(x(t))dt\le \mathfrak{a}(x)-a\int^1_0H(x(t))dt\le 0
$$
by (\ref{e:EH.3.2}).

In summary, in two case we have $\Phi_{H_\nu}(x)<a+\varepsilon$.
\end{proof}

Finally, let us prove (\ref{e:fixpt.1}). Note that
 both $\mathcal{S}=\partial D$ and $D$ contain fixed points of $\tau_0$.
 As above we can assume $0\in D$. Clearly
 $\mathcal{S}$ also contains fixed points of $\tau_0$ and so
 $c_{\rm EH,\tau_0}(\mathcal{S})$ is well-defined.
 Let us define
\begin{eqnarray}\label{e:EH.3.3}
\widetilde{\mathscr{E}}_\epsilon(\mathbb{R}^{2n},\tau_0, \mathcal{S})
\end{eqnarray}
consist of $\overline{H}=f\circ H$, where $f\in C^\infty(\mathbb{R},\mathbb{R})$ satisfies
\begin{eqnarray}\label{e:EH.3.4}
&&f(s)=0\;\hbox{for}\; s\;\hbox{near}\;1,\quad f'(s)\le 0\;\forall s\le 1,
\quad f'(s)\ge 0\;\forall s\ge 1,\\
&&f'(s)=\alpha\in\mathbb{R}\setminus\Sigma^{\tau_0}_{\mathcal S}
\;\hbox{if}\;s\ge 1\;\hbox{and}\;f(s)\ge\epsilon,\label{e:EH.3.5}
\end{eqnarray}
where $\alpha$ is also required to be so large that
\begin{eqnarray}\label{e:EH.3.6}
\alpha H(z)\ge \pi|z|^2-C
\end{eqnarray}
for some constant $C>0$. Similar to Lemma~\ref{lem:EH10}, there holds that
$$
c_{\rm EH,\tau_0}(\overline{H})>\min \Sigma^{\tau_0}_{\mathcal{S}}-\epsilon,\;\forall\; \overline{H}\in \widetilde{\mathscr{E}}_\epsilon(\mathbb{R}^{2n},\tau_0,\mathcal{S}).
$$
It follows that $c_{\rm EH,\tau_0}(\mathcal{S})\ge a$. By the monotonicity of
$c_{\rm EH,\tau_0}$ we get that $c_{\rm EH,\tau_0}(\mathcal{S})=a$.
\qed

\subsection{Proof of Theorem~\ref{th:EHproduct}}\label{sec:EH.3.2}

\begin{lemma}\label{lem:9.1}
$c_{\rm EH,\tau_0}(D\times\mathbb{R}^{2k})=c_{\rm EH,\tau_0}(D)$
for any $\tau_0$-invariant convex domain $D\subset\mathbb{R}^{2n}$.
\end{lemma}

\begin{proof}
It suffices to prove this lemma for a $\tau_0$-invariant bounded strictly convex domain
 $D\subset\mathbb{R}^{2n}$ with $C^2$-smooth boundary ${\mathcal S}$. Let $H=j^2_D$.
 By the definition and monotonicity we have
 $$
 c_{\rm EH,\tau_0}(D\times\mathbb{R}^{2k})=\sup_R c_{\rm EH,\tau_0}(E_R),
 $$
 where $E_R=\{(z,z')\in\mathbb{R}^{2n}\times\mathbb{R}^{2k}\,|\,H(z)+ (|z'|/R)^2<1\}$.
 Clearly, $E_R$ is invariant under the canonical involution on
  $\mathbb{R}^{2n}\times\mathbb{R}^{2k}$ (i.e., the product of
          the canonical involutions on $\mathbb{R}^{2n}$ and $\mathbb{R}^{2k}$).
  Since $E_R$ is convex  and $\mathcal{S}_R=\partial E_R$ is of class $C^{1,1}$
  (because $H$ is of class $C^{1,1}$ on  $\mathbb{R}^{2n}$), (\ref{e:fixpt}) gives rise to
 $$
 c_{\rm EH,\tau_0}(E_R)=\min\Sigma^{\tau_0}_{\mathcal{S}_R},
 $$
where  $\Sigma^{\tau_0}_{\mathcal{S}_R}$ is as in (\ref{e:EHcontact+}).
 Let   $(x,x'):[0,\lambda]\rightarrow \mathcal{S}_R$ satisfy
\begin{eqnarray*}
&&\dot{x}=X_H(x),\quad x(\lambda+t)=x(t)\quad\hbox{and}\quad\quad x(-t)=\tau_0x(t),\\
&&\dot{x}'=2J_0x'/R^2,\quad x'(\lambda+t)=x'(t)\quad\hbox{and}\quad\quad x'(-t)=\tau_0 x'(t).
\end{eqnarray*}
Then  $x\equiv 0$ or $\lambda\in \Sigma^{\tau_0}_{\mathcal S}$ by the first line,
and  $x'\equiv 0$ or $\lambda=k\pi R^2$  by the second one, where $k\in\mathbb{Z}$.
Hence for $R>0$ large enough we arrive at
$$
c_{\rm EH,\tau_0}(E_R)=\min\Sigma^{\tau_0}_{S_R}=\min\Sigma^{\tau_0}_{\mathcal S}=c_{\rm EH,\tau_0}(D)
$$
and so the desired conclusion.
\end{proof}

\begin{lemma}\label{lem:9.2}
Let $D\subset\mathbb{R}^{2n}$ be a $\tau_0$-invariant convex bounded domain
  with $C^2$-smooth boundary ${\mathcal S}$ and containing $0$. Let $\widetilde{H}\in\mathscr{F}(\mathbb{R}^{2n},\tau_0,D)$. Then for any $\epsilon>0$
  there exists $\gamma\in\Gamma$ such that
   \begin{equation}\label{e:EH.3.14}
      \Phi_{\widetilde{H}}|\gamma(B^+\setminus\epsilon B^+)\ge c_{\rm EH,\tau_0}(D)-\epsilon
      \quad\hbox{and}\quad \Phi_{\widetilde{H}}|\gamma(B^+)\ge 0,
   \end{equation}
    where $B^+$ is the closed unit ball in $\mathbb{E}^+$ and $S^+=\partial B^+$.
\end{lemma}
\begin{proof}
Let $\mathscr{E}_{\epsilon/2}(\mathbb{R}^{2n},\tau_0,D)$ be as in
 (\ref{e:EH.3.7}). Replacing $\widetilde{H}$ by a greater function we may assume
 $\widetilde{H}\in\mathscr{E}_{\epsilon/2}(\mathbb{R}^{2n},\tau_0,D)$,
 because making $\tilde{H}$ larger only decreases $\Phi_{\tilde{H}}$.
 Since $\widetilde{H}=0$ on $D$ and
$0\in D$, there exists $\alpha>0$ such that
\begin{equation}\label{e:EH.3.15}
\inf \Phi_{\widetilde{H}}|(\alpha S^+)>0\quad\hbox{and}\quad \Phi_{\widetilde{H}}|(\alpha B^+)\ge 0,
 \end{equation}
 (see the proof of Proposition~\ref{prop:EH.1.4} ).
 Let $\varphi_u$ be the flow of $\nabla \Phi_{\widetilde{H}}$. Put
\begin{eqnarray*}
S_u=\varphi_u(\alpha S^+)\quad\hbox{and}\quad
d(\widetilde{H})=\sup_{u\ge 0}\inf(\Phi_{\widetilde{H}}|S_u).
\end{eqnarray*}
Then we have
$$
0<\inf \Phi_{\widetilde{H}}|S_0\le d(\widetilde{H})\le c_{\rm EH,\tau_0}(\widetilde{H})<\infty.
$$
Since $\Phi_{\widetilde{H}}$ satisfies the (PS) condition, $d(\widetilde{H})$ is a positive critical value of $\Phi_{\widetilde{H}}$, and
  $d(\widetilde{H})\ge c_{\rm EH,\tau_0}(D)-\epsilon/2$ by Lemma~\ref{lem:EH10}. Moreover,
  by the definition of
$d(\widetilde{H)}$ there exists $r>0$ such that $\Phi_{\widetilde{H}}|S_r\ge d(\widetilde{H})-\epsilon/2$ and thus
\begin{equation}\label{e:EH.3.16}
 \Phi_{\widetilde{H}}|S_r\ge c_{\rm EH,\tau_0}(D)-\epsilon.
 \end{equation}
 Because $\Phi_{\widetilde{H}}$ is nondecreasing along flow,
\begin{equation}\label{e:EH.3.17}
\Phi_{\widetilde{H}}|S_u\ge \Phi_{\widetilde{H}}|S_0>0,\quad\forall u\ge 0.
 \end{equation}
Define $\gamma:\mathbb{E}\to \mathbb{E}$ by
$$
\gamma(x^++x^0+x^-)=\widetilde{\gamma}(x^+)+x^0+x^-,
$$
where
\begin{eqnarray*}
&&\widetilde{\gamma}(x)=2(\alpha/\epsilon)x\hspace{14mm}\quad\quad\hbox{if}\quad x\in \mathbb{E}^+\;\hbox{and}\;\|x\|_{1/2}\le\frac{1}{2}\epsilon,\\
&&\widetilde{\gamma}(x)=\varphi_{r(2\|x\|_{1/2}-\epsilon)/\epsilon}(\alpha x/\|x\|_{1/2})\quad\hbox{if}\quad x\in \mathbb{E}^+\;\hbox{and}\;
\frac{1}{2}\epsilon<\|x\|_{1/2}\le \epsilon,\\
&&\widetilde{\gamma}(x)=\varphi_{r}(\alpha x/\|x\|_{1/2})\hspace{8mm}\quad\quad\hbox{if}\quad x\in \mathbb{E}^+\;\hbox{and}\;
\|x\|_{1/2}>\epsilon.
\end{eqnarray*}
%\begin{eqnarray*}
%&&\widetilde{\gamma}(x)=2(\alpha/\epsilon)x\hspace{14mm}\quad\quad\hbox{if}\quad x\in \mathbb{E}^+\;\hbox{and}\;\|x\|\le\frac{1}{2}\epsilon,\\
%&&\widetilde{\gamma}(x)=\varphi_{r(2\|x\|-\epsilon)/\epsilon}(\alpha x/\|x\|)\quad\hbox{if}\quad x\in \mathbb{E}^+\;\hbox{and}\;
%\frac{1}{2}\epsilon<\|x\|\le \epsilon,\\
%&&\widetilde{\gamma}(x)=\varphi_{r}(\alpha x/\|x\|)\hspace{8mm}\quad\quad\hbox{if}\quad x\in \mathbb{E}^+\;\hbox{and}\;
%\|x\|>\epsilon.
%\end{eqnarray*}
Then
$$
\gamma(B^+\setminus\epsilon B^+)=S_r\quad\hbox{and}\quad
\gamma(B^+)=(\alpha B^+)\bigcup_{0\le u\le r}S_u,
$$
Thus this, (\ref{e:EH.3.15}) and (\ref{e:EH.3.16})-(\ref{e:EH.3.17}) show that $\gamma$ satisfies (\ref{e:EH.3.14}).

Finally, we get $\gamma\in\Gamma$ by considering the homotopy
\begin{eqnarray*}
\gamma_0(x)=2(\alpha/\epsilon)x^++x^0+x^-,\quad
\gamma_u(x)=u^{-1}(\gamma(ux)),\quad 0<u\le 1.
\end{eqnarray*}
\end{proof}

\begin{proof}[Proof of Theorem~\ref{th:EHproduct}]
\noindent{\bf Step 1}.
By the standard approximation arguments
 we can assume that each $D_i\subset\mathbb{R}^{2n_i}$ is a $\tau_0$-invariant
  bounded convex domain of class $C^2$.
 Moreover, we may also assume that each $D_i$ contains the origin of $\mathbb{R}^{2n_i}$, $1\le i\le k$, since $c_{\rm EH,\tau_0}$ is invariant under translation by a fixed point of $\tau_0$.
 Thus it follows from  the monotonicity and Lemma~\ref{lem:9.1} that
 \begin{equation}\label{e:EH.3.13}
 c_{\rm EH,\tau_0}(D_1\times\cdots\times D_k)\le \min_ic_{\rm EH,\tau_0}(D_i\times\mathbb{R}^{2(n-n_i)})=\min_ic_{\rm EH,\tau_0}(D_i).
 \end{equation}

In order to prove the converse inequality,
note that for each ${H}\in\mathscr{F}(\mathbb{R}^{2n},\tau_0, D_1\times\cdots\times  D_k)$
we may choose $\widehat{H}_i\in\mathscr{F}(\mathbb{R}^{2n_i},\tau_0, D_i)$, $i=1,\cdots,k$, such that
$$
\widehat{H}(z):=\sum \widehat{H}_i(z_i)\ge H(z)\quad\forall z.
$$
For each $i=1,\cdots,k$, by Lemma~\ref{lem:9.2} there exist  $\gamma_i\in\Gamma(\mathbb{R}^{2n_i})$
such that
$$
\Phi_{\widehat{H}_i}|\gamma_i(B_i^+\setminus(2k)^{-1}B^+_i)\ge c_{\rm EH,\tau_0}(D_i)-\epsilon,\qquad
\Phi_{\widehat{H}_i}|\gamma_i(B_i^+)\ge 0.
$$
Since for any $x=(x_1,\cdots,x_k)\in S^+\subset B^+_1\times\cdots\times B^+_k$ there exists some
$i_0$ such that
$$
x_{i_0}\in B_{i_0}^+\setminus(2k)^{-1}B_{i_0}^+,
$$
 putting $\gamma=\gamma_1\times\cdots\times\gamma_k$ we arrive at
$$
\Phi_{\widehat{H}}(\gamma(x))=\sum \Phi_{\widehat{H}_i}(\gamma_i(x_i))\ge\min_i(c_{\rm EH,\tau_0}(D_i)-\epsilon)
$$
and hence
$$
c_{\rm EH,\tau_0}({H})\ge
c_{\rm EH,\tau_0}(\widehat{H})=\sup_{h\in\Gamma}\inf_{x\in h(S^+)}\Phi_{\widehat{H}}(x)\ge
\min_ic_{\rm EH,\tau_0}(D_i)-\epsilon
$$
by Proposition~\ref{prop:EH.1.2}(i) and (\ref{e:EH.1.2}). This leads to
\begin{equation}\label{e:EH.3.18}
c_{\rm EH,\tau_0}(D_1\times\cdots\times D _k)\ge\min_i c_{\rm EH,\tau_0}(D_i)
 \end{equation}
and so the first equality in (\ref{e:product1}) by combining  with (\ref{e:EH.3.13}).

\vspace*{4pt}\noindent{\bf Step 2}.
As in Step 1, for each $i=1,\cdots,k$ we may assume: (i)
 $D_i\subset\mathbb{R}^{2n_i}$ is a $\tau_0$-invariant bounded
 convex  domain of class $C^2$, (ii) $D_i$ contains the origin of $\mathbb{R}^{2n_i}$.
Then Lemma~\ref{lem:9.2} holds for every $\widetilde{H}\in\mathscr{F}(\mathbb{R}^{2n},\tau_0,\partial D_i)$. Arguing as in Step 1 we get that
$$
c_{\rm EH,\tau_0}(\partial D_1\times\cdots\times \partial D_k)\ge\min_i c_{\rm EH,\tau_0}(D_i).
$$
Since $c_{\rm EH,\tau_0}(\partial D_1\times\cdots\times \partial D_k)
\le  c_{\rm EH,\tau_0}(D_1\times\cdots\times D_k)=\min_ic_{\rm EH,\tau_0}(D_i)$ by
the monotonicity property of $c_{\rm EH,\tau_0}$ and (\ref{e:EH.3.13})
we obtain the second equality in (\ref{e:product1}).
\end{proof}
%\hfill$\Box$\vspace{2mm}

\section{Proof of Theorem~\ref{th:EHcontact}}\label{sec:EH.4}
\setcounter{equation}{0}

 As explained at the beginning of Section~\ref{sec:EH.3.1}
we only need to consider the case that $\tau=\tau_0$ below.

\begin{lemma}\label{spec2}
  For any $\tau_0$-invariant contact hypersurface $\mathcal{S}$, $\Sigma_{\mathcal{S}}^{\tau_0}$ has empty interior.
\end{lemma}
It is well known that for any  contact hypersurface  $\mathcal{S}$,
$\Sigma_{\mathcal{S}}$ has empty interior. Since $\Sigma_{\mathcal{S}}^{\tau_0}$ is a subset of $\Sigma_{\mathcal{S}}$, Lemma~\ref{spec2} follows.

Using the flow $\phi^t$ of the Liouville vector field $X$
we may define a very special parameterized family of hypersurfaces modelled on ${\mathcal S}$, given by
 \begin{eqnarray}\label{e:EH.4.5}
 \psi:(-\varepsilon, \varepsilon)\times\mathcal{S}\to \mathbb{R}^{2n},\; (s, z)\mapsto \phi^s(z),
\end{eqnarray}
where $\varepsilon>0$ is so small that $\mathbb{R}^{2n}\setminus \cup_{t\in (-\varepsilon, \varepsilon)}\phi^t(\mathcal{S})$   has two components.
By (\ref{e:EHcontact}) there holds
\begin{equation}\label{e:EH.4.6}
\tau_0(\phi^t(z))=\phi^t(\tau_0 z),\,\forall (t,z)\in (-\varepsilon, \varepsilon)\times\mathcal{S}.
\end{equation}
Define $U:=\cup_{t\in (-\varepsilon, \varepsilon)}\phi^t(\mathcal{S})$ and
\begin{eqnarray}\label{e:EH.4.7}
K_\psi:U\to\mathbb{R},\;w\mapsto \nu
\end{eqnarray}
if $w=\psi(\nu,z)\in U$ where $z\in\mathcal{S}$. Let $X_{K_\psi}$ be the
Hamiltonian vector field associated to $K_\psi$ defined by $\omega_0(\cdot,X_{K_\psi})=dK_\psi$. Then for  $w=\psi(\nu,z)\in U$ it holds that $\omega_0(w)(X(w),X_{K_\psi}(w))=1$.
Moreover, by a direct computation we prove in Appendix~\ref{app:B},
  \begin{eqnarray}\label{e:EH.4.9}
 X_{K_\psi}(\psi(\nu,z))=e^{-\nu}
  d\phi^\nu(z)[X_{K_\psi}(z)],  \quad\forall (\nu,z)\in(-\varepsilon, \varepsilon)\times{\mathcal S}.
\end{eqnarray}
Clearly (\ref{e:EH.4.6}) and (\ref{e:EH.4.9}) show that $y:\mathbb{R}/T\mathbb{Z}\to {\mathcal S}_\nu$
   satisfies
   $$
   \dot{y}(t)=X_{K_\psi}(y(t)),\quad y(0)=y(T)\quad\hbox{and}\quad y(-t)=\tau_0y(t)
   $$
   if and only if $y(t)=\psi(\nu, x(e^{-\nu} t))$, where $x:\mathbb{R}/e^{-\nu} T\mathbb{Z}\to\mathcal{S}$
    satisfies
    $$
    \dot{x}(t)=X_{K_\psi}(x(t)),\quad x(e^{-\nu} T)= x(0)\quad\hbox{and}\quad x(-t)=\tau_0x(t).
    $$
    In addition
    $A(y)=e^\nu A(x) $.
Fix $0<\delta<\varepsilon$. Let ${\bf A}_\delta$ and ${\bf B}_\delta$
 denote the unbounded and bounded components of
  $\mathbb{R}^{2n}\setminus \cup_{t\in (-\delta, \delta)}\phi^t(\mathcal{S})$,
  respectively.   Then
$$
\psi(\{\nu\}\times \mathcal{S})\subset {\bf B}_\delta\quad\hbox{for}\quad -\varepsilon<\nu<-\delta.
$$
Let $\mathscr{F}(\mathbb{R}^{2n},\tau_0)$ be given by (\ref{e:EH.1.5.1}).
We call $H\in\mathscr{F}(\mathbb{R}^{2n},\tau_0)$ {\bf adapted to} $\psi$ if
\begin{equation}\label{e:EH.4.10}
H(x)=\left\{\begin{array}{ll}
C_0\ge 0 &{\rm if}\;x\in {\bf B}_\delta,\\
f(\nu) &{\rm if}\;x=\psi(\nu,y),\;y\in{\mathcal S},\;\nu\in [-\delta,\delta],\\
C_1\ge 0 &{\rm if}\;x\in {\bf A}_\delta\cap B^{2n}(0,R),\\
h(|x|^2) &{\rm if}\;x\in {\bf A}_\delta\setminus B^{2n}(0,R)
\end{array}\right.
\end{equation}
where $f:(-\varepsilon, \varepsilon)\to\mathbb{R}$ and $h:[0, \infty)\to\mathbb{R}$ are smooth functions satisfying
\begin{eqnarray}
&&f|(-\varepsilon,-\delta]=C_0,\quad f|[\delta,\varepsilon)=C_1,\label{e:EH.4.11}\\
&&sh'(s)-h(s)\le 0\quad\forall s.\label{e:EH.4.12}
\end{eqnarray}
\begin{lemma}\label{lem:EH.4.2}
%\begin{description}
\begin{enumerate}
\item[\bf (i)] If $x$ is a nonconstant critical point of $\Phi_H$ on $X$ such that
$x(0)\in \psi(\{\nu\}\times\mathcal{S})$ for some $\nu\in (-\delta,\delta)$ satisfying $f'(\nu)>0$,
then
$$
e^{-\nu}f'(\nu)\in \Sigma^{\tau_0}_{\mathcal{S}}\quad\hbox{and}\quad \Phi_H(x)=f'(\nu)-f(\nu).
$$
\item[\bf (ii)] If some $\nu\in (-\delta,\delta)$ satisfies $f'(\nu)>0$ and $ e^{-\nu}f'(\nu)\in \Sigma^{\tau_0}_{\mathcal{S}}$,
then there is a nonconstant critical point $x$ of $\Phi_H$ on $X$ such that $x(0)\in \psi(\{\nu\}\times\mathcal{S})$.
%\end{description}
\end{enumerate}
\end{lemma}
\begin{proof}
{\bf (i)} By Proposition~\ref{solution}, $x$ satisfies
\begin{equation}\label{e:criticalpt}
\dot{x}=X_H(x),\quad x(1)=x(0)\quad\hbox{and}\quad x(-t)=\tau_0x(t).
\end{equation}
Since $x(0)\in \psi(\{\nu\}\times\mathcal{S})$, it follows that $x(t)\in\psi(\{\nu\}\times\mathcal{S})$, $\forall t\in [0,1]$ and the first equality in (\ref{e:criticalpt}) becomes
$$
\dot{x}=f'(\nu)X_{K_\psi}(x).
$$
Let $y(t):=\phi^{-\nu}(x(t))$, for $0\le t\le 1$. Then $y:\mathbb{R}/\mathbb{Z}\to \mathcal{S}$ is a $\tau_0$-brake closed characteristic on $\mathcal{S}$. Let $\lambda:=\imath_{X}\omega_0$. Then the action of $y$ is given by
$$
A(y)=\int y^\ast\lambda=\int x^\ast(\phi^{-\nu})^\ast\lambda=e^{-\nu}\int x^\ast\lambda=e^{-\nu}f'(\nu)>0,
$$
i.e. $e^{-\nu}f'(\nu)\in\Sigma^{\tau_0}_{\mathcal{S}}$. Moreover, (since $x$ is $C^2$) we have
$$
\Phi_H(x)=A(x)-\int_0^1 H(x(t))dt=f'(\nu)-f(\nu).
$$

\noindent{\bf (ii)} By the assumption there exists $y:\mathbb{R}/\mathbb{Z}\to\mathcal{S}$ such that
$$
\dot{y}=e^{-\nu}f'(\nu)X_{K_\psi}(y),\quad y(1)=y(0)\quad\hbox{and}\quad
y(-t)=\tau_0y(t).
$$
Hence $x(t)=\psi(\nu, y(t))=\phi^\nu(y(t))$ satisfies
\begin{eqnarray*}
\dot{x}(t)&=&d\phi^\nu(y(t))[\dot{y}(t)]=e^{-\nu}f'(\nu)d\phi^\nu(y(t))[X_{K_\psi}(y)]\\
&=&f'(\nu)X_{K_\psi}(\phi^\nu(y(t)))=f'(\nu)X_{K_\psi}(x(t))=X_H(x(t)),
\end{eqnarray*}
and
$$
x(1)=x(0),\quad x(-t)=\tau_0x(t),
$$
that is, $x$ is the desired critical point of $\Phi_H$ with $x(0)\in\psi(\{\nu\}\times\mathcal{S})$ and $\Phi_H(x)=f'(\nu)-f(\nu)$.
\end{proof}
\noindent{\bf Continuing the proof of Theorem~\ref{th:EHcontact}}.\quad
For $C>0$ large enough and $\delta>2\eta>0$ small enough we define
  an  $H=H_{C,\eta}\in\mathscr{F}(\mathbb{R}^{2n},\tau_0)$
   adapted to $\psi$ as follows:
\begin{equation}\label{e:EH.4.14}
H_{C,\eta}(x)=\left\{\begin{array}{ll}
C\ge 0 &{\rm if}\;x\in {\bf B}_\delta,\\
f_{C,\eta}(\nu) &{\rm if}\;x=\psi(\nu,y),\;y\in{\mathcal S},\;\nu\in [-\delta,\delta],\\
C&{\rm if}\;x\in {\bf A}_\delta\cap B^{2n}(0,R),\\
h_{C,\eta}(|x|^2) &{\rm if}\;x\in {\bf A}_\delta\setminus B^{2n}(0,R)
\end{array}\right.
\end{equation}
where $B^{2n}(0,R)\supseteq\overline{\psi((-\varepsilon,\varepsilon)\times{\mathcal S})}$
(the closure of $\psi((-\varepsilon,\varepsilon)\times{\mathcal S})$),
$f_{C,\eta}:(-\varepsilon, \varepsilon)\to\mathbb{R}$ and $h_{C,\eta}:[0, \infty)\to\mathbb{R}$ are smooth functions satisfying
\begin{eqnarray*}
&&f_{C,\eta}|[-\eta,\eta]\equiv 0,\quad f_{C,\eta}(s)=C\;\hbox{if}\;|s|\ge 2\eta,\\
&&f'_{C,\eta}(s)s>0\quad\hbox{if}\quad \eta<|s|<2\eta,\\
&&f'_{C,\eta}(s)-f_{C,\eta}(s)>c_{\rm EH,\tau_0}(\mathcal{S})+1\quad\hbox{if}\;s>0\;\hbox{and}\;\eta<f_{C,\eta}(s)<C-\eta,\\
&&h_{C,\eta}(s)=a_Hs+b\quad\hbox{for $s>0$ large enough}, a_H>\pi\; \hbox{and}\; a_H=C/R^2\notin \pi\mathbb{Z},\\
&&sh'_{C,\eta}(s)-h_{C,\eta}(s)\le 0\quad\forall s\ge 0.
\end{eqnarray*}
Notice that ${\bf B}_{\delta}$, ${\bf A}_\delta$ and $B^{2n}(0,R)$ is $\tau_0$-invariant. Moreover, we have
 \begin{eqnarray*}
 x\in \psi (\{[-\delta,\delta]\}\times\mathcal{S})\quad &\Longrightarrow &\quad \hbox{$\tau_0x\in  \psi (\{[-\delta,\delta]\}\times\mathcal{S})$ and $K_\psi(x)=K_\psi(\tau_0x)$},\\
   x\in {\bf A}_\delta\setminus B^{2n}(0,R)\quad&\Longrightarrow &\quad \hbox{$\tau_0x\in {\bf A}_\delta\setminus B^{2n}(0,R) $ and $|x|^2=|\tau_0x|^2$}.
   \end{eqnarray*}
Such a family $H_{C,\eta}$ ($C\to+\infty$, $\eta\to 0$) can be chosen to be cofinal in
the set $\mathscr{F}(\mathbb{R}^{2n},\tau_0,\mathcal{S})$ defined by (\ref{e:EH.1.5.2}) and
also to have the property that
\begin{equation}\label{e:EH.4.14.1}
C\le C'\Rightarrow H_{C,\eta}\le H_{C',\eta},\qquad \eta\le \eta'\Rightarrow H_{C,\eta}\ge H_{C,\eta'}.
\end{equation}
It follows that
$$
c_{\rm EH,\tau_0}(\mathcal{S})=\lim_{\eta\to 0\&C\to+\infty}c_{\rm EH,\tau_0}(H_{C,\eta}).
$$
By Proposition~\ref{prop:EH.1.2}(i) and (\ref{e:EH.4.14.1}),
$$
\eta\le \eta'\Rightarrow c_{\rm EH,\tau_0}(H_{C,\eta})\le c_{\rm EH,\tau_0}(H_{C,\eta'}),
$$
 and hence
\begin{equation}\label{e:EH.4.14.2}
\Upsilon(C):=\lim_{\eta\to 0}c_{\rm EH,\tau_0}(H_{C,\eta})
\end{equation}
exists, and for $C\le C'$ we have
$$
\Upsilon(C)=\lim_{\eta\to 0}c_{\rm EH,\tau_0}(H_{C,\eta})\ge \lim_{\eta\to 0}c_{\rm EH,\tau_0}(H_{C',\eta})=\Upsilon(C'),
$$
i.e.,  $C\mapsto\Upsilon(C)$ is non-increasing.
We claim
\begin{equation}\label{e:EH.4.14.3}
 c_{\rm EH,\tau_0}(\mathcal{S})=\lim_{C\to+\infty}\Upsilon(C).
\end{equation}
In fact, for any $\epsilon>0$ there exist $\eta_0>0$ and $C_0>0$ such that
$$
-\epsilon<c_{\rm EH,\tau_0}(H_{C,\eta})-c_{\rm EH,\tau_0}(\mathcal{S})<\epsilon\quad\forall \eta<\eta_0,\;\forall C>C_0.
$$
Letting $\eta\to 0$ we get
$$
-\epsilon\le\Upsilon(C)-c_{\rm EH,\tau_0}(\mathcal{S})\le\epsilon\quad\forall C>C_0,
$$
and thus the desired claim in (\ref{e:EH.4.14.3}).

By Theorem~\ref{th:EH.1.6}, $c_{\rm EH,\tau_0}(H_{C,\eta})$ is  a positive critical value of
$\Phi_{H_{C,\eta}}$ and the associated
critical point $x\in \mathbb{E}$ gives rise to  a nonconstant $\tau_0$-brake characteristic
 sitting in the interior of $U$.
Note that $c_{\rm EH,\tau_0}(H_{C,\eta})>0$ implies $f'_{C,\eta}(\nu)>f_{C,\eta}(\nu)\ge 0$ and so $\nu>0$ by the choice of $f$ below (\ref{e:EH.4.14}).
Then from Lemma~\ref{lem:EH.4.2}(i) we deduce
$$
c_{\rm EH,\tau_0}(H_{C,\eta})=\Phi_{H_{C,\eta}}(x)=f'_{C,\eta}(\nu)-f_{C,\eta}(\nu),
$$
where $f'_{C,\eta}(\nu)\in e^\nu\Sigma^{\tau_0}_{\mathcal{S}}$ and $\eta<|\nu|<2\eta$.
Choose $C>0$ so large that
$$
c_{\rm EH,\tau_0}(H_{C,\eta})<c_{\rm EH,\tau_0}(\mathcal{S})+1.
$$
Then the choice of $f$ below (\ref{e:EH.4.14}) implies:
$$
\hbox{either\quad $f_{C,\eta}(\nu)<\eta\quad$ or $\quad f_{C,\eta}(\nu)>C-\eta$.}
$$

Choose a sequence of positive numbers $\eta_n\to 0$. Passing to a subsequence we may assume two cases.

\vspace*{3pt}\noindent{\bf Case 1}. $c_{\rm EH,\tau_0}(H_{C,\eta_n})=f'_{C,\eta_n}(\nu_n)-f_{C,\eta_n}(\nu_n)=e^{\nu_n}a_n-f_{C,\eta_n}(\nu_n)$, where
$a_n\in\Sigma^{\tau_0}_{\mathcal S}$, $0\le f_{C,\eta_n}(\nu_n)<\eta_n$ and $\eta_n<\nu_n<2\eta_n$.\vspace*{3pt}

Since $c_{\rm EH,\tau_0}(H_{C,\eta_n})\to\Upsilon(C)$, the sequence $e^{-\nu_n}(c_{\rm EH,\tau_0}(H_{C,\eta_n})+f_{C,\eta_n}(\nu_n))=a_n$
is a bounded sequence. Passing to a subsequence we may assume that $(a_n)$ is convergent.
Let $a_n\to a_C\in \overline{\Sigma^{\tau_0}_{\mathcal{S}}}$ (the closure of $\Sigma^{\tau_0}_{\mathcal{S}}$).
Note that
\begin{eqnarray*}
&&\lim_{n\to\infty}\left(e^{-\nu_n}(c_{\rm EH,\tau_0}(H_{C,\eta_n})+f_{C,\eta_n}(\nu_n))\right)\\
&=&\lim_{n\to\infty}e^{-\nu_n}(\lim_{n\to\infty}c_{\rm EH,\tau_0}(H_{C,\eta_n})+\lim_{n\to\infty}f_{C,\eta_n}(\nu_n))\\
&=&\Upsilon(C).
\end{eqnarray*}
Hence
\begin{equation}\label{e:EH.4.16}
\Upsilon(C)=a_C\in \overline{\Sigma^{\tau_0}_{\mathcal{S}}}.
\end{equation}
Moreover, by standard arguments one can show that $\overline{\Sigma^{\tau_0}_{\mathcal{S}}}=\Sigma_{\mathcal{S}}^{\tau_0}\cup\{0\}$. Therefore
$\overline{\Sigma^{\tau_0}_{\mathcal{S}}}$ also has empty interior.\vspace*{3pt}

\noindent{\bf Case 2}. $c_{\rm EH,\tau_0}(H_{C,\eta_n})=f'_{C,\eta_n}(\nu_n)-f_{C,\eta_n}(\nu_n)=e^{\nu_n}a_n-f_{C,\eta_n}(\nu_n)
=e^{\nu_n}a_n-C-(f_{C,\eta_n}(\nu_n)-C)$, where
$a_n\in\Sigma^{\tau_0}_{\mathcal S}$, $C-\eta_n<f_{C,\eta_n}(\nu_n)\le C$ and $\eta_n<\nu_n<2\eta_n$.\vspace*{3pt}

%As in  Case 1 we can prove
Let $a_C$ be as above. Since $\tau_n\to 0$ and $\eta_n\to 0$ as $n\to\infty$, taking limits in both sides of
$$
c(H_{C,\eta_n})=e^{\tau_n}a_n-C-(f_{C,\eta_n}(\tau_n)-C)
$$
we obtain
\begin{equation}\label{e:EH.4.17}
\Upsilon(C)+C=a_C\in \overline{\Sigma^{\tau_0}_{\mathcal{S}}}.
\end{equation}

\vspace*{4pt}\noindent{\bf Step 1.} {Prove $c_{\rm EH,\tau_0}(\mathcal{S})\in\overline{\Sigma^{\tau_0}_{\mathcal{S}}}$.}

 Suppose that there exists an increasing sequence $C_n\to+\infty$ such that $\Upsilon(C_n)=a_{C_n}\in \Sigma^{\tau_0}_{\mathcal{S}}$ for each $n$.
Then by~(\ref{e:EH.4.14.3}) we arrive at
$$
c_{\rm EH,\tau_0}({\mathcal S})=\lim_{n\to\infty}\Upsilon(C_n)\in \overline{\Sigma^{\tau_0}_{\mathcal S}}.
$$
Otherwise, we have
\begin{equation}\label{e:EH.4.18}
\left.\begin{array}{ll}
&\hbox{there exist $\bar{C}>0$ such that
(\ref{e:EH.4.17}) holds}\\
&\hbox{for each $C\in (\bar{C}, +\infty)$
}
\end{array}\right\}
\end{equation}
Let us  prove that this case does not occur. %Note that (\ref{e:EH.4.18}) implies

\begin{claim}\label{cl:EH.4.3}
Assuming (\ref{e:EH.4.18}), if $C<C'$ belong to $(\bar{C}, +\infty)$ then
\begin{equation}\label{e:EH.4.19}
\Upsilon(C)+C\ge \Upsilon(C')+C'.
\end{equation}
\end{claim}

\begin{proof}
Assume that for some $C'>C>\bar{C}$,
\begin{equation}\label{e:EH.4.20}
\Upsilon(C)+C<\Upsilon(C')+C'.
\end{equation}
 We shall prove:
\begin{equation}\label{e:EH.4.21}
\left.\begin{array}{ll}
 &\hbox{ for any given
$d\in (\Upsilon(C)+C,\Upsilon(C')+C')$}\\
&\hbox{ there exists $C_0\in (C, C')$ such that
$\Upsilon(C_0)+C_0=d$.}
\end{array}\right\}
\end{equation}

Put $\Delta_d=\{C''\in (C, C')\,|\, C''+\Upsilon(C'')>d\}$. Since
$\Upsilon(C')+C'>d$ and
$\Upsilon(C')\le\Upsilon(C'')\le\Upsilon(C)$ for any $C''\in (C, C')$
we obtain $\Upsilon(C'')+C''>d$ if $C''\in (C, C')$ is sufficiently close to $C'$.
Hence $\Delta_d\ne\emptyset$. Set $C_0=\inf \Delta_d$. Then $C_0\in [C, C')$.

Let $(C_n'')\subset \Delta_d$ satisfy $C_n''\downarrow C_0$.
Since $\Upsilon(C_n'')\le \Upsilon(C_0)$, we have
$$
d<C_n''+\Upsilon(C_n'')\le\Upsilon(C_0)+ C_n''
$$
for each $n\in\mathbb{N}$,
and thus $d\le\Upsilon(C_0)+C_0$ by letting $n\to\infty$.

Suppose that
\begin{equation}\label{e:EH.4.22}
d<\Upsilon(C_0)+C_0.
\end{equation}
Since $d>C+\Upsilon(C)$, this implies $C\ne C_0$ and so $C_0>C$.
For $\hat{C}\in (C, C_0)$, from $\Upsilon(\hat{C})\ge\Upsilon(C_0)$ and
(\ref{e:EH.4.22}) we derive that $\Upsilon(\hat{C})+\hat{C}>d$
if $\hat{C}$ is close to $C_0$. Hence such $\hat{C}$ belongs to
$\Delta_d$, which contradicts $C_0=\inf\Delta_d$.

Hence (\ref{e:EH.4.22}) does not hold. That is, $d=\Upsilon(C_0)+C_0$.
 (\ref{e:EH.4.21}) is proved. Since (\ref{e:EH.4.21}) contradicts the fact that
 $\overline{\Sigma^{\tau_0}_{\mathcal{S}}}$ has empty interior. Hence (\ref{e:EH.4.19}) does hold for all $\overline{C}<C<C'$.
\end{proof}
Take a $C^\ast>\overline{C}$ and a sequence $C_n$ such that $C_n>C^\ast$ and $\lim_{n\to\infty}C_n=\infty$. By (\ref{e:EH.4.19}) $\Upsilon(C_n)+C_n\le \Upsilon(C^\ast)+C^\ast$. It follows that
$\lim_{n\to\infty}\Upsilon(C_n)=-\infty$, which contradicts ~(\ref{e:EH.4.14.3}). Hence ~(\ref{e:EH.4.18}) does not occur.

\vspace*{4pt}\noindent{\bf Step 2.} {Prove}
\begin{equation}\label{e:EH.4.23}
c_{\rm EH,\tau_0}(B)=c_{\rm EH,\tau_0}({\mathcal S}).
\end{equation}

Note that
\begin{equation}\label{e:EH.4.24}
c_{\rm EH,\tau_0}(B)=\inf_{\eta>0, C>0}c_{\rm EH,\tau_0}(\hat{H}_{C,\eta}),
\end{equation}
where
\begin{equation}\label{e:EH.4.25}
\hat{H}_{C,\eta}(x)=\left\{\begin{array}{ll}
0 &{\rm if}\;x\in {\bf B}_\delta,\\
\hat{f}_{C,\eta}(\nu) &{\rm if}\;x=\psi(\nu,y),\;y\in{\mathcal S},\;\nu\in [-\delta,\delta],\\
C &{\rm if}\;x\in {\bf A}_\delta\cap B^{2n}(0,R),\\
\hat{h}(|x|^2) &{\rm if}\;x\in {\bf A}_\delta\setminus B^{2n}(0,R)
\end{array}\right.
\end{equation}
where $B^{2n}(0,R)\supseteq\overline{\psi((-\varepsilon,\varepsilon)\times{\mathcal S})}$, $\hat{f}_{C,\eta}:(-\varepsilon, \varepsilon)\to\mathbb{R}$ and $\hat{h}:[0, \infty)\to\mathbb{R}$ are smooth functions satisfying
\begin{eqnarray*}
&&\hat{f}_{C,\eta}|(-\varepsilon,\eta]\equiv 0,\quad \hat{f}_{C,\eta}(s)=C\;\hbox{if}\;s\ge 2\eta,\\
&&\hat{f}'_{C,\eta}(s)s>0\quad\hbox{if}\quad \eta<s<2\eta,\\
&&\hat{f}'_{C,\eta}(s)-\hat{f}_{C,\eta}(s)>c_{\rm EH,\tau_0}(\mathcal{S})+1\quad\hbox{if}\;s>0\;\hbox{and}\;
\eta<\hat{f}_{C,\eta}(s)<C-\eta,\\
&&\hat{h}_{C,\eta}(s)=a_Hs+b\quad\hbox{for $s>0$ large enough}, a_H=C/R^2\notin \pi\mathbb{Z}, a_H>\pi,\\
&&s\hat{h}'_{C,\eta}(s)-\hat{h}_{C,\eta}(s)\le 0\quad\forall s\ge 0.
\end{eqnarray*}

For $H_{C,\eta}$ in (\ref{e:EH.4.14}) with related functions ${f}_{C,\eta}$ and ${h}_{C,\eta}$,   we choose an associated
$\hat{H}_{C,\eta}$ as in (\ref{e:EH.4.25}) such that $\hat{f}_{C,\eta}|[0,\varepsilon)={f}_{C,\eta}|[0,\varepsilon)$
and $\hat{h}_{C,\eta}={h}_{C,\eta}$. Consider $H_s=sH_{C,\eta}+(1-s)\hat{H}_{C,\eta}$, $0\le s\le 1$, and put
$$
\Phi_s(x):=\Phi_{H_s}(x)\quad\forall x\in E.
$$

It suffices to prove $c_{\rm EH,\tau_0}(H_0)=c_{\rm EH,\tau_0}(H_1)$. If $x$ is a critical point of $\Phi_s$ with $\Phi_s(x)>0$,
as in Lemma~\ref{lem:positive} we have $x([0,1])\in {\mathcal S}_\nu=\psi(\{\nu\}\times{\mathcal S})$
for some $\nu\in (\eta,2\eta)$. The choice of $\hat{H}_{C,\eta}$ shows
$H_s(x(t))\equiv {H}_{C,\eta}(x(t))$ for $t\in [0,1]$. This implies that
each $\Phi_s$ has the same positive critical value as $\Phi_{H_{C,\eta}}$.
By the continuity in Proposition~\ref{prop:EH.1.2}(ii), $s\mapsto c_{\rm EH,\tau_0}(H_s)$ is continuous
and takes values in the set of positive critical value of $\Phi_{H_{C,\eta}}$
(which has measure zero by Sard's theorem). Hence  $s\mapsto c_{\rm EH,\tau_0}(H_s)$ is constant. In particular,
$$
c_{\rm EH,\tau_0}(\hat{H}_{C,\eta})=c_{\rm EH,\tau_0}(H_0)=c_{\rm EH,\tau_0}(H_1)=c_{\rm EH,\tau_0}(H_{C,\eta}).
$$
These show that
$c_{\rm EH,\tau_0}(\mathcal{S})=c_{\rm EH,\tau_0}(B)\in \overline{\Sigma^{\tau_0}_{\mathcal{S}}}$.
Moreover,  $c_{\rm EH,\tau_0}(B)>0$. Hence $c_{\rm EH,\tau_0}(\mathcal{S})=c_{\rm EH,\tau_0}(B)\in \Sigma ^{\tau_0}_{\mathcal{S}}$.

\section{Proofs of Theorem~\ref{th:Brun.2} and Corollaries~\ref{cor:Brun.2}, \ref{cor:Brun.4}}\label{sec:Brunn}
\setcounter{equation}{0}

\subsection{Proof of Theorem~\ref{th:Brun.2}}\label{sec:Brunn.1}
Since $D, K\subset \mathbb{R}^{2n}$ are convex bodies  containing  $0$ in their interiors,
they have the Minkowski (or gauge) functionals $j_D, j_K:\mathbb{R}^{2n}\to\mathbb{R}$.
Following ideas of \cite{{AAO08}} we need to generalize
  some results in Section~\ref{sec:convex1}. For $p>1$, we identify
  $$
  W^{1,p}(S^1,\mathbb{R}^{2n}):=\left\{x\in W^{1,p}([0,1],\mathbb{R}^{2n})\,|\,x(1)=x(0)\right\}
  $$
  and put
  $$
\mathcal{F}_p:=\left\{x\in W^{1,p}(S^1,\mathbb{R}^{2n})\,\Bigm|\;x(1-t)=\tau_0x(t),\,
\int_0^1x(t)dt=0\right\},
$$
which is a closed subspace of $W^{1,p}([0,1],\mathbb{R}^{2n})$. Since the functional
$$
\mathcal{F}_p\ni x\mapsto  A(x)=\frac{1}{2}\int_0^1\langle -J_0\dot{x}(t),x(t)\rangle_{\mathbb{R}^{2n}} dt
$$
is $C^1$ and $dA(x)[x]=2$ for any $x\in\mathcal{F}_p$ with $A(x)=1$,  we deduce  that
\begin{equation}\label{e:Brun.0}
\mathcal{A}_p:=\{x\in \mathcal{F}_p\,|\,A(x)=1 \}
\end{equation}
is a regular $C^1$ submanifold of $\mathcal{F}_p$.
Define a functional
\begin{equation}\label{e:Brun.1}
I_p:\mathcal{F}_p\to\mathbb{R},\;x\mapsto \int_0^1(H_D^{\ast}(-J_0\dot{x}(t)))^{\frac{p}{2}}dt.
\end{equation}
Recall that $H_D^{\ast}$ is the Legendre transform of $j_D^2$.
If $D$ is strictly convex and has $C^1$-smooth boundary
then $I_p$ is a $C^{1}$ functional  with derivative given by
 $$
 dI_p(x)[y]=\int_0^1\langle \nabla (H_D^{\ast})^{\frac{p}{2}}(-J_0\dot{x}(t)),-J_0\dot{y}\rangle_{\mathbb{R}^{2n}} dt,\quad\forall x,y\in \mathcal{F}_p.
 $$
Notice that $H_D^{\ast}=\frac{1}{4}h_D^2$, where $h_D$ is the support function of $D$. It follows that $(H_D^{\ast})^{\frac{p}{2}}=\frac{1}{2^p}h_D^p$.

Corresponding to \cite[Proposition~2.2]{AAO08} we have next result.

\begin{proposition}\label{prop:Brun.1}
For $p>1$, there holds
$$
(c_{\rm EHZ,\tau_0}(D))^{\frac{p}{2}}=\min_{x\in\mathcal{A}_p}
\int_0^1(H_D^{\ast}(-J_0\dot{x}(t)))^{\frac{p}{2}}dt.
$$
\end{proposition}
\begin{proof}
\noindent{\bf Step 1}.\quad{\it $\mu_p:=\inf_{x\in\mathcal{A}_p}I_p(x)$ is positive. }
Note that
\begin{equation}\label{e:Brun.2}
\|x\|_{L^\infty}\le 2\|\dot{x}\|_{L^p}\quad\forall x\in\mathcal{F}_p
\end{equation}
because $\int_0^1x(t)dt=0$.
So for any $x\in\mathcal{A}_p$ we have
$$
2=2A_p(x)\le \|x\|_{L^q}\|\dot{x}\|_{L^p}\le \|x\|_{L^\infty}\|\dot{x}\|_{L^p}\le  2\|\dot{x}\|_{L^p}^2,
$$
and thus $\|\dot{x}\|_{L^p}\ge 1$, where $\frac{1}{p}+\frac{1}{q}=1$.
Let $R_2$ be as in (\ref{e:dualestimate}). These lead to
$$
I_p(x)\ge\left(\frac{1}{R_2}\right)^{p/2}\|\dot{x}\|_{L^p}^p\ge \left(\frac{1}{R_2}\right)^{p/2}.
$$

\vspace*{4pt}\noindent{\bf Step 2.} {There exists $u\in  \mathcal{A}_p$ such  that
$I_p(u)=\mu_p$.}
Let $(x_n)\subset\mathcal{A}_p$ be a sequence satisfying
 $$
 \lim_{n\rightarrow+\infty}I_p(x_n)=\mu_p.
 $$
Then there exists a constant $C>0$ such that
$$
\left(\frac{1}{R_2}\right)^{p/2}\|\dot{x}_n\|_{L^p}^p\le I_p(x_n)\le C,\quad\forall n\in\mathbb{N}.
$$
By (\ref{e:Brun.2}) and the fact that $\|x\|_{L^p}\le\|x\|_{L^\infty}$, we deduce that $(x_n)$ is bounded in $W^{1,p}(S^1,\mathbb{R}^{2n})$. Therefore $(x_n)$ has  a subsequence, also denoted by $(x_n)$, which
converges weakly to some $u\in W^{1,p}(S^1,\mathbb{R}^{2n})$.
By Arzel\'{a}-Ascoli theorem, there also exists $\hat{u}\in C^{0}(\mathbb{R}/\mathbb{Z},\mathbb{R}^{2n})$ such that
$$
\lim_{n\rightarrow+\infty}\sup_{t\in [0,1]}|x_n(t)-\hat{u}(t)|=0.
$$
Standard arguments lead to $u(t)=\hat{u}(t)$ almost everywhere. Hence $u$ also satisfies
$$
u(1-t)=\tau_0u(t)\quad\hbox{and}\quad\int^1_0u(t)dt=0.
$$
As in Step 2 of \cite[\S4.1]{JinLu19}, we also have $A_p(u)=1$, and so $u\in\mathcal{A}_p$.
As in Section~\ref{sec:convex1} we can use results in
 convex analysis to show that there exists $\omega\in L^q([0,1],\mathbb{R}^{2n})$ such that $\omega(t)\in\partial (H_D^{\ast})^{\frac{p}{2}}(-J_0\dot{u}(t))$
 almost everywhere. It follows that
$$
I_p(u)-I_p(x_n)\le \int_0^1\langle \omega(t),-J_0(\dot{u}(t)-\dot{x}_n(t))\rangle_{\mathbb{R}^{2n}} dt\rightarrow 0
$$
since $x_n$ converges weakly to $u$. Hence
$$
\mu_p\le I_p(u)\le\lim_{n\rightarrow\infty}I_p(x_n)= \mu_p.
$$
\vspace*{4pt}\noindent{\bf Step 3.} {There exists a generalized $\tau_0$-brake closed characteristic on $\partial D$,
${x}^\ast:[0, 1]\rightarrow \partial D$,  such that $A({x}^\ast)=(\mu_p)^{\frac{2}{p}}$.}
Since $u$ is the minimizer  of  $I_p|_{\mathcal{A}_p}$,
  applying Lagrangian multiplier theorem (\cite[Theorem~6.1.1]{Cl83}) we get some $\lambda_p\in\mathbb{R}$ such that
$0\in\partial I_p(u)+\lambda_p A'(u)$.
This means that there exists some $\rho\in L^q([0,1],\mathbb{R}^{2n})$ satisfying
\begin{equation}\label{e:Brun.3-}
\rho(t)\in\partial (H_D^{\ast})^{\frac{p}{2}}(-J_0\dot{u}(t))\quad\hbox{a.e. on}\quad [0,1]
\end{equation}
and
\begin{equation}\label{e:Brun.3}
\int_0^1\langle\rho(t),-J_0\dot{\zeta}(t)\rangle_{\mathbb{R}^{2n}}+\lambda_p\int_0^1\langle u(t),-J_0\dot{\zeta}(t)\rangle_{\mathbb{R}^{2n}}=0\quad\forall \zeta\in\mathcal{F}_p.
\end{equation}
As in Step 3 of Section~\ref{sec:convex1}, this $\rho(t)$ can be chosen to satisfy $\rho(1-t)=\tau_0\rho(t)$ and (\ref{e:Brun.3}) implies
\begin{equation}\label{e:Brun.4}
\rho(t)+\lambda_p u(t)={\bf a}_0,\quad\hbox{a.e. on}\quad [0,1]
\end{equation}
for  $\lambda_p=-\frac{p}{2}\mu_p$ and some ${\bf a}_0\in L_0$. Since $((H_D^{\ast})^{\frac{p}{2}})^\ast=\frac{2^q}{qp^{q-1}}j_D^q$,
by (\ref{e:Brun.3-}) there holds
$$
-J_0\dot{u}(t)\in\frac{2^q}{qp^{q-1}}\partial j_D^q(-\lambda_p u(t)+{\bf a}_0),\quad\hbox{a.e.}.
$$
Let $v(t):=-\lambda_p u(t)+{\bf a}_0$. Then it satisfies
$$
-J_0\dot{v}(t)\in-\lambda_p\frac{2^q}{qp^{q-1}}\partial j_D^q(v(t)),
\quad v(1-t)=\tau_0 v(t),\quad v(1)=v(0).
$$
These imply that $j_D^q(v(t))$ is  constant by \cite[Theorem~2]{Ku96}, and
$$
\int_0^1\frac{-2^{q-1}\lambda_p}{p^{q-1}}j_D^q(v(t))dt=\int_0^1\frac{1}{2}\langle-J_0\dot{v}(t),
v(t)\rangle_{\mathbb{R}^{2n}}dt=\lambda_p^2=\left(\frac{p\mu_p}{2}\right)^2
$$
by the Euler formula (cf. \cite[Theorem~3.1]{YangWei08}).
Therefore $j_D^q(v(t))=\left(\frac{p}{2}\right)^q\mu_p$ and
$$
A_p(v)=\frac{1}{2}\int^1_0\langle-J_0\dot{v}(t), v(t)\rangle_{\mathbb{R}^{2n}} dt=\lambda_p^2=\left(\frac{p\mu_p}{2}\right)^2.
$$
Let $x^{\ast}(t)=\frac{v(t)}{j_D(v(t))}$. Then $x^{\ast}$ is a generalized $\tau_0$-brake closed
characteristic on $\partial D$ with action
$$
A(x^{\ast})=\frac{1}{j_D^2(v(t))}A(v)=\mu_p^{\frac{2}{p}}.
$$
\vspace*{4pt}\noindent{\bf Step 4.} {For any generalized $\tau_0$-brake closed characteristic
on $\partial D$ with positive action, $y:[0,T]\rightarrow \partial D$, there holds $A(y)\ge \mu_p^{\frac{2}{p}}$.}
By Lemma~\ref{lem:genChar}, by reparameterizing it we may assume that
$y\in W^{1,\infty}([0,T],\mathbb{R}^{2n})$ and satisfies
\begin{eqnarray*}
&&j_D(y(t))\equiv 1,\quad
y(0)=y(T),\quad y(T-t)=\tau_0y(t),\\
&&-J_0\dot{y}(t)\in\partial j_D^q(y(t))\quad\hbox{a.e. on}\;[0, T].
\end{eqnarray*}
It follows that
\begin{equation}\label{e:Brun.4.1}
A(y)=\frac{qT}{2}.
\end{equation}
Similar to the case $p=2$, define $y^{\ast}:[0,1]\rightarrow \mathbb{R}^{2n}$,   $t\mapsto y^{\ast}(t)=a y(tT)+ {\bf b}$,
 where $a>0$ and ${\bf b}\in L_0$ are chosen so that $y^{\ast}\in\mathcal{A}_p$. Then (\ref{e:Brun.4.1})
 leads to
 \begin{equation}\label{e:Brun.5}
 1=A(y^{\ast})=a^2A(y)=\frac{a^2qT}{2}.
 \end{equation}
 Moreover, $(H_D^{\ast}(-J_0\dot{y}^{\ast}(t)))^{\frac{p}{2}}
 =(aT)^p\frac{q^{p}}{2^p}$.
 Now Step 1 tells us that $I_p(y^\ast)\ge \mu_p$ and so $(aT)^p\frac{q^{p}}{2^p}\ge\mu_p$.
 This, (\ref{e:Brun.4.1}) and (\ref{e:Brun.5}) lead to $A(y)\ge \mu_p^{\frac{2}{p}}$.

 Summarizing the four steps we have
\begin{eqnarray*}
&&\min\{A(x)>0\,|\,x\;\text{is a generalized}\;\tau_0\hbox{-brake closed characteristic on}\;\partial D\}\\
&&=(\min_{x\in\mathcal{A}_p}I_p)^{\frac{2}{p}}.
\end{eqnarray*}
The desired result follows from this and Theorem~\ref{th:convex}.
\end{proof}
\begin{remark}\label{rem:Brun.1.5}
{\rm  Checking the proof of Proposition~\ref{prop:Brun.1} it is easily seen that
for a minimizer $u$ of  $I_p|_{\mathcal{A}_p}$ ($p>1$) there exists ${\bf a}_0\in L_0$
such that
$$
x^\ast(t)=\left(c_{\rm EHZ,\tau_0}(D)\right)^{1/2}u(t)+ \frac{2}{p}\left(c_{\rm EHZ,\tau_0}(D)\right)^{(1-p)/2}{\bf a}_0
$$
gives a generalized $\tau_0$-brake closed characteristic on $\partial D$
with  $A(x^{\ast})=c_{\rm EHZ,\tau_0}(D)$, namely, $x^\ast$ is a $c_{\rm EHZ,\tau_0}$-carrier for $\partial D$.}
\end{remark}
The following is a corresponding result with \cite[Proposition~2.1]{AAO08}.

\begin{proposition}\label{prop:Brun.2}
  For $p_1>1$ and $p_2\ge 1$, there holds
  $$
  (c_{\rm EHZ,\tau_0}(D))^{\frac{p_2}{2}}=\min_{x\in\mathcal{A}_{p_1}}
\int_0^1(H_D^{\ast}(-J_0\dot{x}(t)))^{\frac{p_2}{2}}dt=\min_{x\in\mathcal{A}_{p_1}}\frac{1}{2^{p_2}}\int_0^1
  (h_{D}(-J_0\dot{x}))^{p_2}dt.
  $$
\end{proposition}

\begin{proof}
 Firstly, suppose $p_1\ge p_2>1$. Then $\mathcal{A}_{p_1}\subset \mathcal{A}_{p_2}$
and the first two steps in the proof of Proposition~\ref{prop:Brun.1} imply that
  $I_{p_1}|_{\mathcal{A}_{p_1}}$ has a minimizer $u\in \mathcal{A}_{p_1}$.
  It follows that
   \begin{eqnarray*}
  c_{\rm EHZ,\tau_0}(D)&=&\left(\int_0^1(H_D^{\ast}(-J_0\dot{u}(t)))^{\frac{p_1}{2}}dt\right)^
  {\frac{2}{p_1}}\\
  &\ge& \left(\int_0^1(H_D^{\ast}(-J_0\dot{u}(t)))^{\frac{p_2}{2}}dt\right)^
  {\frac{2}{p_2}}\\
  &\ge & \inf_{x\in\mathcal{A}_{p_1}}\left(\int_0^1(H_D^{\ast}(-J_0\dot{x}(t)))^{\frac{p_2}{2}}dt\right)^
  {\frac{2}{p_2}}\\
  &\ge & \inf_{x\in\mathcal{A}_{p_2}}\left(\int_0^1(H_D^{\ast}(-J_0\dot{x}(t)))^{\frac{p_2}{2}}dt\right)^
  {\frac{2}{p_2}}\\
  &=&c_{\rm EHZ,\tau_0}(D),
  \end{eqnarray*}
  where two equalities come from Proposition~\ref{prop:Brun.1} and the first inequality
  is because of H\"older's inequality.
  Hence the functional $\int_0^1(H_D^{\ast}(-J_0\dot{x}(t)))^{\frac{p_2}{2}}dt$ attains its minimum at $u$ on $\mathcal{A}_{p_1}$ and
  \begin{equation}\label{e:Brun.6}
  c_{\rm EHZ,\tau_0}(D)=\min_{x\in\mathcal{A}_{p_1}}
  \left(\int_0^1(H_D^{\ast}(-J_0\dot{x}(t)))^{\frac{p_2}{2}}dt\right)^{\frac{2}{p_2}}.
  \end{equation}

Next, if  $p_2\ge p_1>1$, then $\mathcal{A}_{p_2}\subset \mathcal{A}_{p_1}$
and we have $u\in \mathcal{A}_{p_2}$ minimizing
  $I_{p_2}|_{\mathcal{A}_{p_2}}$ such  that
   \begin{eqnarray*}
  c_{\rm EHZ,\tau_0}(D)&=&\left(\int_0^1(H_D^{\ast}(-J_0\dot{u}(t)))^{\frac{p_2}{2}}dt\right)^
  {\frac{2}{p_2}}\\
   &\ge & \inf_{x\in\mathcal{A}_{p_1}}\left(\int_0^1(H_D^{\ast}(-J_0\dot{x}(t)))^{\frac{p_2}{2}}dt\right)^
  {\frac{2}{p_2}}\\
  &\ge & \inf_{x\in\mathcal{A}_{p_1}}\left(\int_0^1(H_D^{\ast}(-J_0\dot{x}(t)))^{\frac{p_1}{2}}dt\right)^
  {\frac{2}{p_1}}\\
  &=&c_{\rm EHZ,\tau_0}(D).
  \end{eqnarray*}
This yields (\ref{e:Brun.6}) again.

  Finally, let $p_2=1$ and let $u\in \mathcal{A}_{p_1}$ minimize
  $I_{p_1}|_{\mathcal{A}_{p_1}}$. It is clear that
  \begin{eqnarray}\label{e:Brun.7}
  c_{\rm EHZ,\tau_0}(D)&=&\left(\int_0^1(H_D^{\ast}(-J_0\dot{u}(t)))^{\frac{p_1}{2}}dt\right)^
  {\frac{2}{p_1}}\nonumber\\
  &\ge& \left(\int_0^1(H_D^{\ast}(-J_0\dot{u}(t)))^{\frac{1}{2}}dt\right)^{2}\nonumber\\
  &\ge & \inf_{x\in\mathcal{A}_{p_1}}\left(\int_0^1(H_D^{\ast}(-J_0\dot{x}(t)))^{\frac{1}{2}}dt
  \right)^{2}.
  \end{eqnarray}
  Let $R_2$ be as in (\ref{e:dualestimate}).
 Then
 $$
 (H_D^{\ast}(-J\dot{x}(t)))^{\frac{p}{2}}\le (R_2|\dot{x}(t)|^2)^{\frac{p}{2}}\le (R_2+1)^{\frac{p_1}{2}}
 (1+|\dot{x}(t)|)^{p_1}\quad\forall t
 $$
 for any $1\le p\le p_1$. By (\ref{e:Brun.6})
 $$
  c_{\rm EHZ,\tau_0}(D)=\min_{x\in\mathcal{A}_{p_1}}
  \left(\int_0^1(H_D^{\ast}(-J_0\dot{x}(t)))^{\frac{p}{2}}dt\right)^{\frac{2}{p}},\quad 1<p\le p_1.
 $$
    Letting $p\downarrow 1$ and using Lebesgue dominated convergence theorem we get
  $$
  c_{\rm EHZ,\tau_0}(D)\le \inf_{x\in\mathcal{A}_{p_1}}\left(\int_0^1(H_D^{\ast}(-J_0\dot{x}(t)))^{\frac{1}{2}}dt\right)^
  {2}.
  $$
  This and (\ref{e:Brun.7}) show  that the functional $\mathcal{A}_{p_1}\ni x\mapsto\int_0^1(H_D^{\ast}(-J_0\dot{x}(t)))^{\frac{1}{2}}dt$ attains its minimum at $u$ and
 $$
  c_{\rm EHZ,\tau_0}(D)=\min_{x\in\mathcal{A}_{p_1}}
  \left(\int_0^1(H_D^{\ast}(-J_0\dot{x}(t)))^{\frac{1}{2}}dt\right)^{2}.
  $$
\end{proof}

%
%As in \cite[Corollary~2.3]{AAO08},
%Propositions~\ref{prop:Brun.1},~\ref{prop:Brun.2} have the following
%
%\begin{corollary}\label{cor:Brun.4}
%For fixed $p_1>1$ and $p_2\ge 1$, if  $u\in \mathcal{A}_{p_1}$ minimize
%  $I_{p_2}|_{\mathcal{A}_{p_1}}$, i.e.,
%$$
%\min_{x\in \mathcal{A}_{p_1}}\int_0^1(H_D^{\ast}(-J\dot{x}(t)))^{\frac{p_2}{2}}dt
%=\int_0^1(H_D^{\ast}(-J\dot{u}(t)))^{\frac{p_2}{2}}dt,
%$$
%then the function $H_D^{\ast}(-J\dot{u}(t))$ is equal to the constant $c_{\rm EHZ,\tau_0}(D)$.
%\end{corollary}
%\begin{proof}
%If $p_2>1$, then
%\begin{eqnarray*}
%c_{\rm EHZ,\tau_0}(D)&=&(\int_0^1(H_D^\ast(-J\dot{u}))^{p_2/2})^{2/p_2}\\
%&\ge& (\int_0^1(H_D^\ast(-J\dot{u}))^{1/2})^{2}\\
%&\ge&\min_{x\in\mathcal{A}_{p_1}}(\int_0^1(H_D^\ast(-J\dot{x}))^{1/2})^{2}\\
%&=&c_{\rm EHZ,\tau_0}(D)
%\end{eqnarray*}
%which implies that $(\int_0^1(H_D^\ast(-J\dot{u}))^{p_2/2})^{2/p_2}
%=(\int_0^1(H_D^\ast(-J\dot{u}))^{1/2})^{2}$ and hence $H_D^\ast(-J\dot{u})$ is constant.
%\end{proof}

\noindent{\it Proof of Theorem~\ref{th:Brun.2}}.
As done at the beginning of Section~\ref{sec:convex},
the original question can be boiled down to the case $\tau=\tau_0$.

Choose a real $p_1>1$. Then Proposition~\ref{prop:Brun.2} implies
\begin{eqnarray}\label{e:Brun.8}
  c_{\rm EHZ,\tau_0}(D+_pK)^{\frac{p}{2}}&=&\min_{x\in\mathcal{A}_{p_1}}\int_0^1
  \left(\frac{(h_{D+_pK}(-J_0\dot{x}))^2}{4}\right)^{\frac{p}{2}}\nonumber\\
  &=&\min_{x\in\mathcal{A}_{p_1}}\frac{1}{2^p}\int_0^1
  (h_{D+_pK}(-J_0\dot{x}))^{p}\nonumber\\
  &=&\min_{x\in\mathcal{A}_{p_1}}\frac{1}{2^p}\int_0^1
  ((h_{D}(-J_0\dot{x}))^{p}+(h_{K}(-J_0\dot{x}))^{p})\nonumber\\
  &\ge& \min_{x\in\mathcal{A}_{p_1}}\frac{1}{2^p}\int_0^1
  (h_{D}(-J_0\dot{x}))^{p}+\min_{x\in\mathcal{A}_{p_1}}\frac{1}{2^p}\int_0^1
  (h_{K}(-J_0\dot{x}))^{p}\nonumber\\
  &=&c_{\rm EHZ,\tau_0}(D)^{\frac{p}{2}}+c_{\rm EHZ,\tau_0}(K)^{\frac{p}{2}}.
\end{eqnarray}

Now suppose that $p>1$ and the equality in (\ref{e:BrunA}) holds.
We may require that the above $p_1$ satisfies $1<p_1<p$.
Since  there exists $u\in\mathcal{A}_{p_1}$ such that
\begin{eqnarray*}
  c_{\rm EHZ,\tau_0}(D+_pK)^{\frac{p}{2}}=\int_0^1
  \left(\frac{(h_{D+_pK}(-J_0\dot{u}))^2}{4}\right)^{\frac{p}{2}},
  \end{eqnarray*}
 the above computation yields
 \begin{eqnarray*}
  c_{\rm EHZ,\tau_0}(D+_pK)^{\frac{p}{2}}&=&\frac{1}{2^p}\int_0^1
  ((h_{D}(-J_0\dot{u}))^{p}+(h_{K}(-J_0\dot{u}))^{p})\\
  &\ge& \min_{x\in\mathcal{A}_{p_1}}\frac{1}{2^p}\int_0^1
  (h_{D}(-J_0\dot{x}))^{p}+\min_{x\in\mathcal{A}_{p_1}}\frac{1}{2^p}\int_0^1
  (h_{K}(-J_0\dot{x}))^{p}\\
  &=&c_{\rm EHZ,\tau_0}(D)^{\frac{p}{2}}+c_{\rm EHZ,\tau_0}(K)^{\frac{p}{2}}
\end{eqnarray*}
 and thus
\begin{eqnarray*}
&&c_{\rm EHZ,\tau_0}(D)^{\frac{p}{2}}=\min_{x\in\mathcal{A}_{p_1}}\frac{1}{2^p}\int_0^1
  (h_{D}(-J_0\dot{x}))^{p}=\frac{1}{2^p}\int_0^1
  (h_{D}(-J_0\dot{u}))^{p}\quad\hbox{and}%\\
\end{eqnarray*} \begin{eqnarray*}&&c_{\rm EHZ,\tau_0}(K)^{\frac{p}{2}}=\min_{x\in\mathcal{A}_{p_1}}\frac{1}{2^p}\int_0^1
  (h_{K}(-J_0\dot{x}))^{p}=\frac{1}{2^p}\int_0^1
  (h_{K}(-J_0\dot{u}))^{p}.
\end{eqnarray*}
These,  Propositions~\ref{prop:Brun.1}, \ref{prop:Brun.2}  and H\"older's inequality lead to
\begin{eqnarray*}
\min_{x\in\mathcal{A}_{p_1}}\left(\int_0^1
  (h_{D}(-J_0\dot{x}))^{p_1}\right)^{\frac{1}{p_1}}&=& 2(c_{\rm EHZ,\tau_0}(D))^{\frac{1}{2}} \\   &=&\min_{x\in\mathcal{A}_{p_1}}\left(\int_0^1
  (h_{D}(-J_0\dot{x}))^{p}\right)^{\frac{1}{p}}\\
  &=&\left(\int_0^1(h_{D}(-J_0\dot{u}))^{p}\right)^{\frac{1}{p}}\\
  &\ge&
  \left(\int_0^1
  (h_{D}(-J_0\dot{u}))^{p_1}\right)^{\frac{1}{p_1}},\\
 \min_{x\in\mathcal{A}_{p_1}}\left(\int_0^1
  (h_{K}(-J_0\dot{x}))^{p_1}\right)^{\frac{1}{p_1}}&=&2(c_{\rm EHZ,\tau_0}(K))^{\frac{1}{2}} \\ &=&\min_{x\in\mathcal{A}_{p_1}}\left(\int_0^1
  (h_{K}(-J_0\dot{x}))^{p}\right)^{\frac{1}{p}}\\
  &=&\left(\int_0^1(h_{K}(-J_0\dot{u}))^{p}\right)^{\frac{1}{p}}\\
  &\ge&\left(\int_0^1(h_{K}(-J_0\dot{u}))^{p_1}\right)^{\frac{1}{p_1}}.
 \end{eqnarray*}
 It follows that
 \begin{eqnarray*}
 2(c_{\rm EHZ,\tau_0}(D))^{\frac{1}{2}}&=&
  \left(\int_0^1(h_{D}(-J_0\dot{u}))^{p}\right)^{\frac{1}{p}}=
  \left(\int_0^1(h_{D}(-J_0\dot{u}))^{p_1}\right)^{\frac{1}{p_1}},\\
 2(c_{\rm EHZ,\tau_0}(K))^{\frac{1}{2}}&=&
  \left(\int_0^1(h_{K}(-J_0\dot{u}))^{p}\right)^{\frac{1}{p}}
  =\left(\int_0^1(h_{K}(-J_0\dot{u}))^{p_1}\right)^{\frac{1}{p_1}}.
 \end{eqnarray*}
% Recall that H\"older's inequality claims $\int_0^1(h_{K}(-J\dot{u}))^{p_1}\le \|(h_{K}(-J\dot{u}))^{p_1}\|_{L^r}$
% with $r=p/p_1$ and there holds equality if and only if
% there exist nonnegative real $C_1, C_2$, $C_1+C_2>0$, such that $C_1 ((h_{K}(-J\dot{u}(t)))^{p_1})^r=C_2$
% a.e. $t\in (0,1)$. Hence  $h_{D}(-J\dot{u})$ is equal
%to  $2c_{\rm EHZ,\tau_0}(D,\omega_0)^{\frac{1}{2}}$. Similarily,   $h_{K}(-J\dot{u})$
%is equal to  $2c_{\rm EHZ,\tau_0}(K,\omega_0)^{\frac{1}{2}}$.
%
%
%It follows that $u$ minimizes
%$$
%\int_0^1(H_D^{\ast}(-J\dot{x}(t)))^{\frac{p_1}{2}}dt\quad\hbox{and}\quad \int_0^1(H_K^{\ast}(-J\dot{x}(t)))^{\frac{p_1}{2}}dt
%$$
%over $\mathcal{A}_{p_1}$.
%%%%%%%%%%%%%%%%%%%%%%%%%%%%%%%%%%%%%%%%%%%%%%%%%%%%%
 By Remark~\ref{rem:Brun.1.5} there are
 ${\bf a}_D, {\bf a}_K\in L_0={\rm Fix}(\tau_0)$ such that
\begin{eqnarray*}
&&\gamma_D(t)=\left(c_{\rm EHZ,\tau_0}(D)\right)^{1/2}u(t)+ \frac{2}{p_1}\left(c_{\rm EHZ,\tau_0}(D)\right)^{(1-p_1)/2}{\bf a}_D,\\
&&\gamma_K(t)=\left(c_{\rm EHZ,\tau_0}(K)\right)^{1/2}u(t)+ \frac{2}{p_1}\left(c_{\rm EHZ,\tau_0}(D)\right)^{(1-p_1)/2}{\bf a}_K
\end{eqnarray*}
are  $c_{\rm EHZ,\tau_0}$ carriers for $\partial D$ and $\partial K$, respectively.
Clearly, they coincide up to  dilation and translation by elements  in
$L_0$.

Finally, suppose that $p\ge 1$ and there exist $c_{\rm EHZ,\tau_0}$ carriers
$\gamma_D:[0, T]\to\partial D$ and $\gamma_K:[0, T]\to\partial K$
satisfying $\gamma_D=\alpha\gamma_K+{\bf b}$
for some $\alpha\in\mathbb{R}\setminus\{0\}$ and
some ${\bf b}\in L_0$.
Then the latter and (\ref{e:action1}) imply $A(\gamma_D)=\alpha^2 A(\gamma_K)$.
Moreover by Step 4 in the proof of Proposition~\ref{prop:Brun.1} we can
construct $z_D$ and $z_K$ in $\mathcal{A}_{p_1}$ such that
\begin{eqnarray}\label{e:Brun.8.1}
&&c_{\rm EHZ,\tau_0}(D)^{\frac{p}{2}}=\min_{x\in\mathcal{A}_{p_1}}\frac{1}{2^p}\int_0^1
  (h_{D}(-J_0\dot{x}))^{p}=\frac{1}{2^p}\int_0^1
  (h_{D}(-J_0\dot{z}_D))^{p},\\
  &&c_{\rm EHZ,\tau_0}(K)^{\frac{p}{2}}=  \min_{x\in\mathcal{A}_{p_1}}\frac{1}{2^p}\int_0^1
  (h_{K}(-J_0\dot{x}))^{p}=\frac{1}{2^p}\int_0^1
  (h_{K}(-J_0\dot{z}_K))^{p}.\label{e:Brun.8.2}
\end{eqnarray}
Precisely, for suitable vectors ${\bf b}_D, {\bf b}_K\in L_0$ it holds that
\begin{eqnarray*}
z_D(t)=\frac{1}{\sqrt{A(\gamma_D)}}\gamma_D(Tt)+{\bf b}_D\quad\hbox{and}\quad
z_K(t)=\frac{1}{\sqrt{A(\gamma_K)}}\gamma_K(Tt)+{\bf b}_K.
\end{eqnarray*}
It follows that $\dot{z}_D(t)=\dot{z}_K$.
This, (\ref{e:Brun.8.1})-(\ref{e:Brun.8.2}) and (\ref{e:Brun.8}) lead to
\begin{eqnarray*}
  c_{\rm EHZ,\tau_0}(D+_pK)^{\frac{p}{2}}
  =c_{\rm EHZ,\tau_0}(D)^{\frac{p}{2}}+c_{\rm EHZ,\tau_0}(K)^{\frac{p}{2}}.
\end{eqnarray*}
\qed

%
%Let $\tau$ be a linear anti-symplectic involution on $\mathbb{R}^{2n}$,
%and let $D, K\subset \mathbb{R}^{2n}$ be two $\tau$-invariant convex bodies  containing  $0$ in their interiors. Then for any real $p\ge 1$ it holds that
%   \begin{equation}\label{e:BrunB}
%   \left(c_{\rm EHZ,\tau_0}(D+_pK,\omega_0)\right)^{\frac{p}{2}}\ge \left(c_{\rm EHZ,\tau_0}(D,\omega_0)\right)^{\frac{p}{2}}+ \left(c_{\rm EHZ,\tau_0}(K,\omega_0)\right)^{\frac{p}{2}}.
%   \end{equation}
%   Moreover, the equality holds if and only if there exist
%    $c_{\rm EHZ,\tau_0}$-carriers for $D$ and $K$,
%      $\gamma_D:[0,T]\rightarrow \partial D$ and  $\gamma_K:[0,T]\rightarrow \partial K$, such that
%      they coincide up to translation and dilation, i.e.,
%   $\gamma_D=\alpha\gamma_K+ {\bf b}$ for some
%   $\alpha\in\mathbb{R}$ and ${\bf b}\in L_0$.

\subsection{Proofs of  Corollaries~\ref{cor:Brun.2}, \ref{cor:Brun.4}}\label{sec:Brunn.2}

\begin{proof}[Proof of Corollary~\ref{cor:Brun.2}]
 As in the proof of \cite[Corollary~1.8]{AAO08} (i) follows from
Proposition~\ref{MonComf} and Corollary~\ref{cor:Brun.1} directly.
In order to show  the existence of the limit in (\ref{e:BrunC}) we take
 $p\in{\rm Fix}(\tau)\cap{\rm Int}(D)$ and
$q\in{\rm Fix}(\tau)\cap{\rm Int}(K)$.
Note that
$$
c_{\rm EHZ,\tau}(D+\varepsilon K)-c_{\rm EHZ,\tau}(D)=c_{\rm EHZ,\tau}((D-q)+\varepsilon K)-
c_{\rm EHZ,\tau}(D-q)
$$
and $K\subset R(D-q)=\{Rx\,|\,x\in D-q\}$ for some $R>0$ (since $0\in{\rm int}(D-q)$).
We may deduce that the function of $\varepsilon>0$  in (\ref{e:BrunC})
is bounded. This function is also decreasing by Corollary~\ref{cor:Brun.1}
(see reasoning \cite[pages 21-22]{AAO08}), and so the desired conclusion is obtained.

The first inequality in (\ref{e:BrunD}) easily follows from Corollary~\ref{cor:Brun.1}.
In order to prove the second one let us
fix a real $p_1>1$. By Proposition~\ref{prop:Brun.2} we have $u\in \mathcal{A}_{p_1}$ such
that
  \begin{eqnarray}\label{e:Brun.9}
  (c_{\rm EHZ,\tau}(D))^{\frac{1}{2}}=(c_{\rm EHZ,\tau}(D-q))^{\frac{1}{2}}
 &=&\min_{x\in\mathcal{A}_{p_1}}\frac{1}{2}\int_0^1
  h_{D-q}(-J_0\dot{x}))\nonumber\\
  &=&\frac{1}{2}\int_0^1h_{D-q}(-J_0\dot{u}))
  \end{eqnarray}
and that for some ${\bf a}_0\in L={\rm Fix}(\tau)$
\begin{eqnarray}\label{e:Brun.10}
x^\ast(t)=\left(c_{\rm EHZ,\tau}(D)\right)^{1/2}u(t)+ \frac{2}{p_1}\left(c_{\rm EHZ,\tau}(D)\right)^{(1-p_1)/2}{\bf a}_0
 \end{eqnarray}
is  a $c_{\rm EHZ,\tau}$ carrier for $\partial (D-q)$ by Remark~\ref{rem:Brun.1.5}.
Proposition~\ref{prop:Brun.2} also leads to
\begin{eqnarray}\label{e:Brun.11}
  (c_{\rm EHZ,\tau}(D+ \varepsilon K))^{\frac{1}{2}}&=&(c_{\rm EHZ,\tau}((D-q)+ \varepsilon (K-p)))^{\frac{1}{2}}\\
  &=&\min_{x\in\mathcal{A}_{p_1}}\frac{1}{2}\int_0^1
  (h_{D-q}(-J_0\dot{x})+ \varepsilon h_{K-p}(-J_0\dot{x}))\nonumber\\
  &\le& \frac{1}{2}\int_0^1
  h_{D-q}(-J_0\dot{u})+\frac{\varepsilon}{2}\int_0^1
  h_{K-p}(-J_0\dot{u})\nonumber\\
  &=&(c_{\rm EHZ,\tau}(D))^{\frac{1}{2}}+\frac{\varepsilon}{2}\int_0^1
  h_{K-p}(-J_0\dot{u})
\end{eqnarray}
because of (\ref{e:Brun.9}). Let $z_D(t)=x^\ast(t)+q$ for $0\le t\le 1$. Since $q$ and ${\bf a}_0$
are fixed points of $\tau$ it is easily checked that $z_D$ is
a $c_{\rm EHZ,\tau}$-carrier for $\partial D$. From (\ref{e:Brun.11}) it follows that
\begin{equation}\label{e:Brun.12}
\frac{(c_{\rm EHZ,\tau}(D+\varepsilon K))^{\frac{1}{2}}-(c_{\rm EHZ,\tau}(D))^{\frac{1}{2}}}{\varepsilon}\le
\frac{1}{2}\left(c_{\rm EHZ,\tau}(D)\right)^{-\frac{1}{2}}
\int_0^1h_{K-p}(-J_0\dot{z}_D).
\end{equation}
Since $h_{K-p}(-J_0\dot{z}_D)=h_{K}(-J_0\dot{z}_D)+\langle p, J_0\dot{z}_D\rangle_{\mathbb{R}^{2n}}$ (see page 37 and Theorem~1.7.5 in \cite{Sch93}) and
$$
\int_0^1\langle p, J_0\dot{z}_D\rangle_{\mathbb{R}^{2n}}=\langle p, J_0(z_D(1)-z_D(0))\rangle_{\mathbb{R}^{2n}}=0
$$
(by the fact $z_D(1)=z_D(0)$), letting $\varepsilon\to 0+$ in (\ref{e:Brun.12})
we arrive at the second inequality in (\ref{e:BrunD}).
\end{proof}
%\hfill$\Box$\vspace{2mm}

\begin{proof}[Proof of Corollary~\ref{cor:Brun.4}]
By the proof of \cite[Theorem~2]{BoLiMi88}, for any $\varepsilon\in (0, 1/2)$
there are orthogonal matrixes $A_j\in O(n)$, $j=1,\cdots,N$, such that
 \begin{eqnarray}\label{e:BrunN}
&&(1-\varepsilon)r_\Delta B^n(1)\subset \frac{1}{2N}\sum^N_{j=1}(A_j\Delta+ (-A_j)\Delta)\subset
(1-\varepsilon)r_\Delta B^n(1),\\
&&(1-\varepsilon)r_\Lambda B^n(1)\subset \frac{1}{2N}\sum^N_{j=1}(A_j\Lambda+ (-A_j)\Lambda)\subset
(1-\varepsilon)r_\Lambda B^n(1),\label{e:BrunO}
\end{eqnarray}
As above it follows from \cite[Theorem~1.4]{AAO14} that
\begin{eqnarray*}
&& c_{\rm EHZ}(\Delta\times\Lambda)\le c_{\rm EHZ}\left(\frac{1}{2N}\sum^N_{i=1}\Psi_{A_i}(\Delta\times\Lambda)+
\frac{1}{2N}\sum^N_{i=1}\Psi_{-A_i}(\Delta\times\Lambda)\right)\\
&\le& c_{\rm EHZ}\left(\frac{1}{2N}\sum^N_{i=1}((A_i\Delta)\times(A_i\Lambda))+
\frac{1}{2N}\sum^N_{i=1}((-A_i)\Delta\times(-A_i)\Lambda)\right)\\
&\le& c_{\rm EHZ}\left(\left(\frac{1}{2N}\sum^N_{j=1}(A_j\Delta+ (-A_j)\Delta)
\right)\times\left(\frac{1}{2N}\sum^N_{j=1}(A_j\Lambda+ (-A_j)\Lambda)\right)\right).
\end{eqnarray*}
Since the Hofer-Zehnder capacity $c_{\rm HZ}$ is continuous with respect to
the Hausdorff metric on the class of convex domains this and
(\ref{e:BrunH})-(\ref{e:BrunI}) lead to
\begin{eqnarray}\label{e:BrunP}
c_{\rm EHZ}(\Delta\times\Lambda)\le c_{\rm EHZ}\left((r_\Delta B^n(1))\times (r_\Lambda B^n(1))
\right).
\end{eqnarray}
Using the symplectomorphism
\begin{eqnarray}\label{e:BrunQ}
&&B^n(r_\Delta)\times B^n(r_\Lambda)\to B^n(\sqrt{r_\Delta r_\Lambda})\times B^n(\sqrt{r_\Delta r_\Lambda}),\\
&&\hspace{18mm} (q,p)\mapsto (\sqrt{r_\Lambda/r_\Delta}q, \sqrt{r_\Delta/r_\Lambda}p),\nonumber
\end{eqnarray}
we deduce
\begin{eqnarray}\label{e:BrunR}
c_{\rm EHZ}\left((r_\Delta B^n(1))\times (r_\Lambda B^n(1))\right)&=&c_{\rm EHZ}(\sqrt{r_\Lambda r_\Delta}(B^n(1)\times B^n(1))\nonumber\\
&=&r_\Lambda r_\Delta c_{\rm EHZ}(B^n(1)\times B^n(1)).
\end{eqnarray}
Hence (\ref{e:BrunL}) follows from (\ref{e:BrunH}),
(\ref{e:BrunP}) and (\ref{e:BrunR}).

Since the symplectomorphism in (\ref{e:BrunQ}) commutes with $\tau_0$ and $\hat{\tau_0}$,
(\ref{e:BrunM}) may be derived by the same arguments with (\ref{e:BrunI}).
\end{proof}
%\hfill$\Box$\vspace{2mm}

%\noindent{\bf Proof of Proposition~\ref{prop:Brun.3}}.

\section{Proofs of Theorems~\ref{th:billT.2},~\ref{th:billT.3}, \ref{th:billT.4}}\label{sec:BillT}
\setcounter{equation}{0}
%,~\ref{th:billT.4}

\noindent{\it Proof of Theorem~\ref{th:billT.2}.}\quad
%As done at the beginning of Section~\ref{sec:convex},
%the original question can be boiled down to the case $\tau=\tau_0$.
%In this case $\bar{p}=0$ and $\bar{q}$ can be an arbitrary point in ${\rm Int}(\Delta_1)\cap {\rm Int}(\Delta_2)$. Since $(\bar{q},0)$ and $(2\bar{q},0)$ are fixed points of $\tau_0$,
%Proposition~\ref{prop:EH.1.7}(ii) implies
%\begin{eqnarray*}
%c_{\rm EHZ,\tau_0}\bigl((\Delta_1-\bar{q})\times \Lambda\bigr)
%&=&c_{\rm EHZ,\tau_0}\bigl(\Delta_1\times \Lambda\bigr),\\
%c_{\rm EHZ,\tau_0}\bigl((\Delta_2-\bar{q})\times \Lambda\bigr)
%&=&c_{\rm EHZ,\tau_0}\bigl(\Delta_2\times \Lambda\bigr),\\
%c_{\rm EHZ,\tau_0}\bigl(\bigl((\Delta_1-\bar{q})+ (\Delta_2-\bar{q})\bigr)\times \Lambda\bigr)
%&=&c_{\rm EHZ,\tau_0}\bigl(\bigl(\Delta_1+ \Delta_2\bigr)\times \Lambda\bigr).
%\end{eqnarray*}
%Hence {\bf we may also assume $\bar{q}=0$ below}, that is,
%$0\in{\rm Int}(\Delta_1)\cap {\rm Int}(\Delta_2)$ and $\Lambda$ is  centrally symmetric.
For any $\lambda\in (0,1)$, since
\begin{eqnarray*}
&&\bigl(\lambda\Delta_1\bigr)\times \bigl(\lambda\Lambda\bigr)+
\bigl((1-\lambda)\Delta_2\bigr)\times \bigl((1-\lambda)\Lambda\bigr)\\
&=&\bigl(\lambda\Delta_1+(1-\lambda)\Delta_2\bigr)\times \bigl(\lambda\Lambda+
(1-\lambda)\Lambda\bigr)\\
&=&\bigl(\lambda\Delta_1+(1-\lambda)\Delta_2\bigr)\times \Lambda
\end{eqnarray*}
it follows from Corollary~\ref{cor:Brun.1} that
\begin{eqnarray}\label{e:BillT.1}
&&\bigl(c_{\rm EHZ,\tau_0}\bigl(\lambda(\Delta_1)\times \lambda(\Lambda)\bigr)\bigr)^{\frac{1}{2}}+
\bigl(c_{\rm EHZ,\tau_0}\bigl((1-\lambda)(\Delta_2)\times (1-\lambda)(\Lambda)\bigr)\bigr)^{\frac{1}{2}}\nonumber\\
&\le &\bigl(c_{\rm EHZ,\tau_0}\bigl(\bigl(\lambda\Delta_1+(1-\lambda)\Delta_2\bigr)\times \Lambda\bigr)\bigr)^{\frac{1}{2}},
\end{eqnarray}
which is equivalent  to
\begin{eqnarray}\label{e:BillT.2}
&&\lambda\bigl(c_{\rm EHZ,\tau_0}\bigl(\Delta_1\times \Lambda\bigr)\bigr)^{\frac{1}{2}}+
(1-\lambda)\bigl(c_{\rm EHZ,\tau_0}\bigl(\Delta_2\times \Lambda\bigr)\bigr)^{\frac{1}{2}}\nonumber\\
&\le &\bigl(c_{\rm EHZ,\tau_0}\bigl(\bigl(\lambda\Delta_1+(1-\lambda)\Delta_2\bigr)\times \Lambda\bigr)^{\frac{1}{2}}.
\end{eqnarray}
By this and the weighted arithmetic-geometric mean inequality
\begin{eqnarray*}
&&\lambda\bigl(c_{\rm EHZ,\tau_0}\bigl(\Delta_1\times \Lambda\bigr)\bigr)^{\frac{1}{2}}+
(1-\lambda)\bigl(c_{\rm EHZ,\tau_0}\bigl(\Delta_2\times \Lambda\bigr)\bigr)^{\frac{1}{2}}\\
&&\ge \left(\bigl(c_{\rm EHZ,\tau_0}\bigl(\Delta_1\times \Lambda\bigr)\bigr)^{\frac{1}{2}}\right)^\lambda
\left(\bigl(c_{\rm EHZ,\tau_0}\bigl(\Delta_2\times \Lambda\bigr)\bigr)^{\frac{1}{2}}\right)^{(1-\lambda)},
\end{eqnarray*}
we get
\begin{eqnarray}\label{e:BillT.3}
&&\left(\bigl(c_{\rm EHZ,\tau_0}\bigl(\Delta_1\times \Lambda\bigr)\bigr)^{\frac{1}{2}}\right)^\lambda
\left(\bigl(c_{\rm EHZ,\tau_0}\bigl(\Delta_2\times \Lambda\bigr)\bigr)^{\frac{1}{2}}\right)^{(1-\lambda)}
\nonumber\\
&\le& \bigl(c_{\rm EHZ,\tau_0}\bigl(\bigl(\lambda\Delta_1+(1-\lambda)\Delta_2\bigr)\times \Lambda\bigr)^{\frac{1}{2}}.
\end{eqnarray}
Replacing $\Delta_1$ and $\Delta_2$ by $\Delta_1':=\lambda^{-1}\Delta_1$ and $\Delta_2':=(1-\lambda)^{-1}\Delta_2$,
respectively, we arrive at
\begin{equation}\label{e:BillT.4}
\left(\bigl(c_{\rm EHZ,\tau_0}\bigl(\Delta'_1\times \Lambda\bigr)\bigr)^{\frac{1}{2}}\right)^\lambda
\left(\bigl(c_{\rm EHZ,\tau_0}\bigl(\Delta'_2\times \Lambda\bigr)\bigr)^{\frac{1}{2}}\right)^{(1-\lambda)}
\le \bigl(c_{\rm EHZ,\tau_0}\bigl(\bigl(\Delta_1+\Delta_2\bigr)\times \Lambda\bigr)^{\frac{1}{2}}.
\end{equation}
For any $\mu>0$, since
$$
\phi:(\Delta_1\times\Lambda, \mu\omega_0)\to ((\mu\Delta_1)\times\Lambda, \omega_0),\;(x,y)\mapsto (\mu x,y)
$$
is a symplectomorphism which commutes with $\tau_0$, (i) and (ii) of Proposition~\ref{MonComf} lead to
\begin{eqnarray*}
&&c_{\rm EHZ,\tau_0}\bigl(\Delta'_1\times \Lambda\bigr)=\lambda^{-1}c_{\rm EHZ,\tau_0}\bigl(\Delta_1\times \Lambda\bigr),\\
&& c_{\rm EHZ,\tau_0}\bigl(\Delta'_2\times \Lambda\bigr)=(1-\lambda)^{-1}c_{\rm EHZ,\tau_0}\bigl(\Delta_2\times \Lambda\bigr).
\end{eqnarray*}
Let us choose $\lambda\in (0,1)$ such that $\Upsilon:=c_{\rm EHZ,\tau_0}\bigl(\Delta'_1\times \Lambda\bigr)=c_{\rm EHZ,\tau_0}\bigl(\Delta'_2\times \Lambda\bigr)$, i.e.,
\begin{equation}\label{e:BillT.5}
\lambda=\frac{c_{\rm EHZ,\tau_0}(\Delta_1\times \Lambda)}{
c_{\rm EHZ,\tau_0}(\Delta_1\times \Lambda)+ c_{\rm EHZ,\tau_0}(\Delta_2\times \Lambda)}.
\end{equation}
Then (\ref{e:linesymp6}) for $\tau=\tau_0$ may be obtained as follows:
\begin{eqnarray}\label{e:BillT.6}
\zeta^{\tau_0}_\Lambda(\Delta_1+\Delta_2)&=&c_{\rm EHZ,\tau_0}\bigl(\bigl(\Delta_1+\Delta_2\bigr)\times \Lambda\bigr)\nonumber\\
&\ge&\left(c_{\rm EHZ,\tau_0}\bigl(\Delta'_1\times \Lambda\bigr)\right)^\lambda
\left(c_{\rm EHZ,\tau_0}\bigl(\Delta'_2\times \Lambda\bigr)\right)^{(1-\lambda)}\nonumber\\
&=&\Upsilon=\lambda\Upsilon+(1-\lambda)\Upsilon\nonumber\\
&=&\lambda c_{\rm EHZ,\tau_0}\bigl(\Delta'_1\times \Lambda\bigr)+(1-\lambda)c_{\rm EHZ,\tau_0}\bigl(\Delta'_2\times \Lambda\bigr)
\nonumber\\
&=&c_{\rm EHZ,\tau_0}\bigl(\Delta_1\times \Lambda\bigr)+c_{\rm EHZ,\tau_0}\bigl(\Delta_2\times \Lambda\bigr)\nonumber\\
&=&\zeta^{\tau_0}_\Lambda(\Delta_1)+ \zeta^{\tau_0}_\Lambda(\Delta_2).
\end{eqnarray}

The final claim follows from Corollary~\ref{cor:Brun.1}.
\qed

\begin{proof}[Proof of Theorem~\ref{th:billT.3}]
 By the assumptions and Proposition~\ref{MonComf}(ii) we have
\begin{eqnarray}\label{e:BillT.8}
\zeta^{\tau_0}(\Delta)&=&c_{\rm EHZ,\tau_0}(\Delta\times B^n(1))\nonumber\\
&\ge& c_{\rm EHZ,\tau_0}(B^{n}(\bar{q},r)\times B^n(1))\nonumber\\
&=&c_{\rm EHZ,\tau_0}(B^{n}(r)\times B^n(1))
\end{eqnarray}
for any ball $B^{n}(\bar{q},r)\subset\Delta$ (since $(\bar{q},0)$ is a fixed point of $\tau_0$).
As in the arguments from (\ref{e:BrunP}) to (\ref{e:BrunR}) we use Proposition~\ref{MonComf}(i) to get
\begin{eqnarray*}
c_{\rm EHZ,\tau_0}(B^{n}(r)\times B^n(1))=c_{\rm EHZ,\tau_0}(B^n(\sqrt{r})\times B^n(\sqrt{r}))
=4r
\end{eqnarray*}
and hence $\zeta^{\tau_0}(\Delta)\ge 4r(\Delta)$.
Similarly,   we deduce
\begin{eqnarray*}
\xi^{\tau_0}(\Delta)&=&c_{\rm EHZ,\tau_0}(\Delta\times B^n(1))\\
&\le& c_{\rm EHZ,\tau_0}(B^{n}(\bar{q},R)\times B^n(1))\nonumber\\
&=&c_{\rm EHZ,\tau_0}(B^{n}(R)\times B^n(1))=4R.
\end{eqnarray*}
for any ball $B^{n}(\bar{q}, R)\supseteq\Delta$.
 This and Theorem~\ref{th:billT.1} yield $\xi^{\tau_0}(\Delta)\le 4R(\Delta)$.

Note that the symplectomorphism in (\ref{e:BrunQ})
also commutes with $\hat{\tau_0}$. The same arguments yield the final results.
\end{proof}
%\hfill$\Box$\vspace{2mm}

\begin{proof}[Proof of Theorem~\ref{th:billT.4}]
For any $u\in S^{n-1}_\Delta$, $\Delta$ sits between support planes $H(\Delta,u)$ and $H(\Delta,-u)$,
 the hyperplane $H_u$  is between  $H(\Delta,u)$ and $H(\Delta,-u)$ and
 has distance ${\rm width}(\Delta)/2$ to $H(\Delta,u)$ and $H(\Delta,-u)$ respectively.
 Choose any ${\bf O}\in O(n)$ such that ${\bf O}u=e_1=(1,0,\cdots,0)\in\mathbb{R}^n$.
 Then  the composition of translation $(q,v)\mapsto (q-\bar{q},v)$
and $\Psi_{\bf O}$ defined in (\ref{e:BrunF.1}),
$$
\Psi_{{\bf O}, \bar{q}}:\mathbb{R}^n_q\times \mathbb{R}^n_p\to \mathbb{R}^n_q\times
\mathbb{R}^n_p,\;(q, v)\mapsto ({\bf O}(q-\bar{q}), {\bf O}v),
$$
 maps $\Delta\times B^n(1)$ into $Z^{2n}_\Delta$. From this and Proposition~\ref{prop:EH.1.7}
 and $\Psi_{\bf O}\tau_0=\tau_0\Psi_{\bf O}$
it follows that
\begin{eqnarray*}
\zeta^{\tau_0}(\Delta)&=&c_{\rm EHZ,\tau_0}(\Delta\times B^n)
\le  c_{\rm EHZ,\tau_0}(Z^{2n}_\Delta).
\end{eqnarray*}
Note that $Z^{2n}_\Delta\subset\mathbb{R}^n_q\times \mathbb{R}^n_p\equiv\mathbb{R}^{2n}$
may be identified with symplectic product
$$
([-{\rm width}(\Delta)/2, {\rm width}(\Delta)/2]
\times[-1,1])\times\mathbb{R}^{2(n-1)}\subset\mathbb{R}^{2}\times \mathbb{R}^{2(n-1)}.
$$
 Hence Theorem~\ref{th:EHproduct} yields
$$
c_{\rm EHZ,\tau_0}(Z^{2n}_\Delta, \omega_0)=c_{\rm EHZ,\tau_0}([-{\rm width}(\Delta)/2, {\rm width}(\Delta)/2]
\times[-1,1]).
$$
Since $[-{\rm width}(\Delta)/2, {\rm width}(\Delta)/2]\times[-1,1]$ can be approximated by
centrally symmetric $2$-dimensional ellipsoids, Corollary~\ref{cor:ellipsoid+} gives rise to
$$
c_{\rm EHZ,\tau_0}([-{\rm width}(\Delta)/2, {\rm width}(\Delta)/2]
\times[-1,1])=2{\rm width}(\Delta).
$$
The desired result is proved.
\end{proof}
%\hfill$\Box$\vspace{2mm}

\appendix
\section{Some facts  on symplectic matrixes}\label{app:A}\setcounter{equation}{0}

\begin{lemma}[\hbox{\cite[Lemma~5]{AlbersF12}}]\label{lem:alb}
Let $A^T$ denote the transpose of $A\in{\rm GL}(n,\mathbb{R})$. Then
$$
\{\Psi\in{\rm Sp}(2n,\mathbb{R})\,|\,\Psi\tau_0=\tau_0\Psi\}=\left\{
\left(
  \begin{array}{cc}
    A & 0 \\
    0 & (A^T)^{-1} \\
  \end{array}
\right)\,|\, A\in{\rm GL}(n,\mathbb{R})
\right\}.
$$
\end{lemma}
It is easily checked that this is true if $\tau_0$ is replaced by $\hat{\tau}_0$.

\begin{theorem}[\hbox{\cite[Theorem~10.2.5]{Xu92}}]\label{th:xu}
  For two semi-positive definite real symmetric matrixes $A,B\in\mathbb{R}^{n\times n}$
   there exists a nonsingular real matrix $P\in \mathbb{R}^{n\times n}$ such that
   \begin{eqnarray*}
   P^TAP&=&{\rm diag}(\lambda_1,\cdots,\lambda_r, 0_{n-r}),\\
   P^TBP&=&{\rm diag}(1-\lambda_1,\cdots, 1-\lambda_r, 0_{n-r}),
   \end{eqnarray*}
   where $r={\rm rank}(A+B)$ and $1\ge \lambda_1\ge\lambda_2\ge\cdots\ge\lambda_r\ge 0$.
\end{theorem}

\begin{proposition}\label{prop:xu}
A matrix $S\in\mathbb{R}^{2n\times 2n}$ commutes with $\tau_0$ if and only if
 $$
   S=\left(
     \begin{array}{cc}
       S_{11} & 0 \\
       0 & S_{22}
     \end{array}
   \right)
   $$
where $S_{11}, S_{12}\in\mathbb{R}^{n\times n}$. If $S$ is positive definite (so are $S_{11}$ and $S_{12}$),
then there exists nonsingular real matrix $P\in \mathbb{R}^{n\times n}$
and positive numbers  $1>\lambda_1\ge\lambda_2\ge\cdots\ge\lambda_n>0$ such that
 \begin{eqnarray*}
  && P^TS_{11}P={\rm diag}(\lambda_1,\cdots,\lambda_n),\\
  && P^T(S_{22})^{-1}P={\rm diag}(1-\lambda_1,\cdots, 1-\lambda_n),\\
  &&\Psi^TS\Psi={\rm diag}(\lambda_1,\cdots,\lambda_n, 1/(1-\lambda_1),\cdots, 1/(1-\lambda_n)),
   \end{eqnarray*}
      where
      $$
   \Psi=\left(
     \begin{array}{cc}
       P & 0 \\
       0 & (P^T)^{-1}
     \end{array}
   \right)
   $$
   is a symplectic matrix. Clearly there exists symplectic matrix $\Phi$ such that $\Phi\circ\tau_0=\tau_0\circ\Phi$ and
   $$
   \Phi^T{\rm diag}(\lambda_1,\cdots,\lambda_n, 1/(1-\lambda_1),\cdots, 1/(1-\lambda_n))\Phi={\rm diag}(r_1^2,\cdots,r_n^2, r_1^2,\cdots, r_n^2),
   $$
   where $r_1\le\cdots\le r_n$.
   Then $\widehat{\Phi}:=\Psi\Phi$ is a symplectic matrix that commutes with $\tau_0$ and satisfies $\widehat{\Phi}^TS\widehat{\Phi}={\rm diag}(r_1^2,\cdots,r_n^2, r_1^2,\cdots, r_n^2)$.
\end{proposition}

\section{Proof of (\ref{e:EH.4.9})}\label{app:B}
\setcounter{equation}{0}

Recall that $\phi^t$ is the flow of the Liouville vector field $X$. Then $(\phi^s)^\ast\omega_0=e^s\omega_0$,
and for any $p$ near ${\mathcal S}$ we have $X(p)=\frac{d}{ds}|_{s=0}\phi^s(p)$ and thus
 \begin{eqnarray}\label{e:AppB.1}
   d\phi^\nu(p)[X(p)]=\frac{d}{ds}\Big|_{s=0}\phi^\nu(\phi^s(p))
    =\frac{d}{ds}\Big|_{s=0}\phi^{s}(\phi^\nu(p))=X(\phi^\nu(p)).
\end{eqnarray}
% we derive  $\omega_{0}(d\phi^s(x)[u], d\phi^s(x)[v])=
% e^s\omega_0(u,v)$ for any $x\in {\mathcal S}$ and $u,v\in \mathbb{R}^{2n}$. In particular, this implies
%  \begin{eqnarray}\label{e:B.4.6}
% &&(\phi^\tau)^\ast(\omega_0|_{{\mathcal S}_\tau})=e^\tau\omega_0|_{\mathcal S},\\
% &&X=\psi_\ast(\frac{\partial}{\partial s}),\;{\rm i.e.}, \;d\psi(\tau,x)[\frac{\partial}{\partial s}|_{(\tau,x)}]=X(\psi(\tau,x)),\;\forall (\tau,x)\in(-\varepsilon, \varepsilon)\times{\mathcal S}.\label{e:B.4.7}
%\end{eqnarray}
Put $\lambda:=\omega_0(X,\cdot)=i_X\omega_0$.
  Then $d\lambda=d(i_X\omega_0)=d(i_X\omega_0)+ i_X(d\omega_0)=L_X\omega_0=\omega_0$. For
  any $p$ near ${\mathcal S}$,
  %$p=\phi^\nu(x)=\psi(\nu,x)$,
%  (\ref{e:EH.4.7}) implies  $K_\psi(\phi^\nu(x))=\nu$.
 a straightforward computation yields
  \begin{eqnarray}\label{e:B.4.8}
 \lambda_p(X_{K_\psi})&=&(\omega_0)_p(X(p), X_{K_\psi}(p))\nonumber\\
 &=&dK_\psi(X(p))=\frac{d}{ds}\Big|_{s=0}K_\psi(\phi^s(p))
 =\frac{d}{ds}\Big|_{s=0}s=1.
\end{eqnarray}
%$\omega(X_H, v)=-dH(v)$
Moreover, by (\ref{e:EH.4.7}) we have  $K_\psi(\psi(\nu,y))=\nu$
for all $(\nu,y)\in (-\delta,\delta)\times {\mathcal S}$, and so
 \begin{eqnarray}\label{e:B.4.5}
(\mathcal{L}_{{\mathcal S}_\nu})_x=X_{K_\psi}(x)\mathbb{R},\quad\forall
      x\in {\mathcal S}_\nu.
\end{eqnarray}
Since $d\phi^\nu(x):(\mathcal{L}_{\mathcal S})_x\to (\mathcal{L}_{{\mathcal S}_\nu})_{\phi^\nu(x)}
$ is an isomorphism for any $x\in {\mathcal S}$, equation (\ref{e:B.4.5}) implies
  \begin{eqnarray}\label{e:B.4.10}
  X_{K_\psi}(\phi^\nu(x))=ad\phi^\nu(x)[X_{K_\psi}(x)]
  \end{eqnarray}
  for some $a\in\mathbb{R}$. Let $\lambda$ act two sides of this equation.
Using (\ref{e:B.4.8}) we arrive at
 \begin{eqnarray}\label{e:B.4.10+}
 1=\lambda(X_{K_\psi}(\phi^\nu(x)))=a\lambda (d\phi^\nu(x)[X_{K_\psi}(x)]).
  \end{eqnarray}
On the other hand %from (\ref{e:B.4.10}) we derive
 \begin{eqnarray*}
  \lambda\left(d\phi^\nu(p)[X_{K_\psi}(p)]\right)&=&
  i_X\omega_0\left(d\phi^\nu(p)[X_{K_\psi}(p)]\right)\\
 &=&\omega_0(X(\phi^\nu(p)), d\phi^\nu(p)[X_{K_\psi}(p)])\\
  &=&\omega_0(d\phi^\nu(p)[X(p)], d\phi^\nu(p)[X_{K_\psi}(p)])\\
  &=&((\phi^\nu)^\ast\omega_0)(X(p), X_{K_\psi}(p))\\
  &=&e^\nu\omega_0(X(p), X_{K_\psi}(p))\\
  &=&e^\nu (i_X\omega_0)(X_{K_\psi}(p))\\
 &=&e^\nu\lambda(X_{K_\psi}(p))=e^\nu,
\end{eqnarray*}
where the third and final equalities come from (\ref{e:AppB.1}) and (\ref{e:B.4.8}), respectively.
%(Actually, $(\phi^\tau)^\ast\lambda=e^\tau\lambda$.)
This and (\ref{e:B.4.10+}) produce  $a=e^{-\nu}$ and hence (\ref{e:EH.4.9})
by (\ref{e:B.4.10}).

\section*{Acknowledgments} We would like to
thank Professor Jean-Claude Sikorav  for sending us his beautiful \cite{Sik90}
and explaining some details.
The authors are deeply grateful to the anonymous referees for giving many very helpful comments
and suggestions to improve the exposition.

\medskip
\begin{tabular}{l}
 School of Mathematical Sciences, Beijing Normal University\\
 Laboratory of Mathematics and Complex Systems, Ministry of Education\\
 Beijing 100875, The People's Republic of China\\
 E-mail address: rrjin@mail.bnu.edu.cn,\hspace{5mm}gclu@bnu.edu.cn\\
\end{tabular}

\medskip
% The data information below will be filled by AIMS editorial staff
Received  June 2019; revised February 2020.
\medskip

\end{document}